\documentclass[12pt]{article}
\usepackage{amsmath,amssymb,amsthm,amsfonts}
\usepackage{paralist}
\usepackage{graphics} 
\usepackage{epsfig} 
\usepackage{graphicx}
\usepackage{wrapfig}
\usepackage{caption}
\usepackage{subcaption}
\usepackage[colorlinks=true]{hyperref}
\usepackage{txfonts}
\usepackage{enumerate}

\usepackage[numbers,sort&compress]{natbib}
\numberwithin{equation}{section}

\newcommand{\be}{\begin{equation}}
\newcommand{\ee}{\end{equation}}

\newcommand{\ben}{\begin{eqnarray*}}
\newcommand{\enn}{\end{eqnarray*}}

\newtheorem{proposition}{Proposition}[section]
\newtheorem{theorem}{\textbf Theorem}[section]
\newtheorem{lemma}{\textbf Lemma}[section]
\newtheorem{remark}{\textbf Remark}[section]
\newtheorem{corollary}{\textbf Corollary}[section]

 \numberwithin{equation}{section}
\begin{document}

\title{\textbf{Stationary Ring and Concentric-Ring Solutions of the Keller-Segel Model with Quadratic Diffusion}}
\author{Lin Chen \thanks{Department of Mathematics, Southwestern University of Finance and Economics, 555 Liutai Ave, Wenjiang, Chengdu, Sichuan 611130, China; lchen@smail.swufe.edu.cn },
Fanze Kong \thanks{Department of Mathematics, Southwestern University of Finance and Economics, 555 Liutai Ave, Wenjiang, Chengdu, Sichuan 611130, China; fzkong@smail.swufe.edu.cn}
and Qi Wang \thanks{Department of Mathematics, Southwestern University of Finance and Economics, 555 Liutai Ave, Wenjiang, Chengdu, Sichuan 611130, China; qwang@swufe.edu.cn}
        }

\maketitle
\vspace{-8mm}
\begin{abstract}

This paper investigates the Keller-Segel model with quadratic cellular diffusion over a disk in $\mathbb R^2$ with a focus on the formation of its nontrivial patterns.  We obtain explicit formulas of radially symmetric stationary solutions and such configurations give rise to the ring patterns and concentric airy patterns.  These explicit formulas empower us to study the global bifurcation and asymptotic behaviors of these solutions, within which the cell population density has $\delta$-type spiky structures when the chemotaxis rate is large.  The explicit formulas are also used to study the uniqueness and quantitative properties of nontrivial stationary radial patterns ruled by several threshold phenomena determined by the chemotaxis rate.  We find that all nonconstant radial stationary solutions must have the cellular density compactly supported unless for a discrete sequence of bifurcation values at which there exist strictly positive small-amplitude solutions.  The hierarchy of free energy shows that in the radial class the inner ring solution has the least energy while the constant solution has the largest energy, and all these theoretical results are illustrated through bifurcation diagrams.  A natural extension of our results to $\mathbb R^2$ yields the existence, uniqueness and closed-form solution of the problem in this whole space.  Our results are complemented by numerical simulations that demonstrate the existence of non-radial stationary solutions in the disk.
\end{abstract}

%
%
%

\maketitle

\section{Introduction and Main Results}
In this paper, we investigate the following system for functions $(u,v)$ of space-time variable $(\textbf{x},t)$
\begin{equation}\label{01}
\left\{\begin{array}{ll}
u_t=\nabla\cdot(u \nabla u-\chi u \nabla v),&\textbf{x} \in B_0(R),t>0,\\
v_t=\Delta v -v+u, & \textbf{x} \in B_0(R),t>0,\\
u(\textbf{x},0),v(\textbf{x},0)\geq 0, \not\equiv 0,& \textbf{x} \in B_0(R),\\
\partial_\nu u(\textbf{x},t)=\partial_\nu v  (\textbf{x},t)=0,& \textbf{x}\in \partial B_0(R), t>0,
\end{array}
\right.
\end{equation}
where $B_0(R)\subset \mathbb R^2$ is the disk centered at the origin and with a radius $R$, and $\nu$ is the unit outer normal on the boundary $\partial B_0(R)$.  We want to study the formation of nontrivial patterns within (\ref{01}) by looking at its nonconstant steady states.  In particular, we restrict our interest to the nonnegative radially symmetric steady states $(u,v)=(u(r),v(r))$ with $r=|\textbf{x}|$ such that
\begin{equation}\label{ss}
\left\{\begin{array}{ll}
0=(ru(u-\chi v)_r)_r,&r\in (0,R),\\
0= v_{rr}+\frac{1}{r}v_r-v+u, &r\in (0,R),\\
u\in C^0([0,R]), v\in C^2([0,R]), u(r)\geq 0, v(r)>0,&r\in (0,R),\\
\int_{B_0(R)} u(r)d\textbf{x}=M; u_r(r)=v_r(r)=0 \text{~for~} r=0,R.
\end{array}
\right.
\end{equation}
Then we will obtain explicit solutions of the stationary system (\ref{ss}) and classify all of them in terms of the value of parameter $\chi$.  These explicit formulas will be used to analyze the uniqueness and qualitative behaviors of the steady states presented in various bifurcation diagrams as we shall see later.

(\ref{01}) arises as a model of chemotaxis that describes the evolution of cellular distribution at population level due to random noise and stimulating chemical in environment.  This biological phenomenon was described through PDE systems proposed by Evelyn Keller and Lee Segel in the 1970s \cite{KS,KS1,KS2}.  A general Keller--Segel model consists of two strongly coupled parabolic equations of $(u(\textbf{x},t),v(\textbf{x},t))$, $u$ the cellular population density and $v$ the chemical concentration at space--time location $(\textbf{x},t)$
 \begin{equation}\label{11}
\left\{
\begin{array}{ll}
u_t=\nabla \cdot (\overbrace{\mu \nabla u}^{\text{random (flux)}}-\overbrace{ \phi(u,v) \nabla v}^{\text{chemotactic (flux)}}),&\textbf{x} \in \Omega,t>0, \\
v_t=\overbrace{d\Delta v}^{\text{chemical diffusion}}+\overbrace{k(u,v)}^{\text{chemical creation/consumption}},&\textbf{x} \in \Omega,t>0,
\end{array}
\right.
\end{equation}
given non-negative initial data $u(\textbf{x},0),v(\textbf{x},0)\geq,\not \equiv 0$ over the spatial region $\Omega$, which is usually taken to be the whole space $\mathbb R^N$, $N\geq1$, or its bound domain with additional non-flux boundary conditions imposed on $u$ and $v$ in an enclosed environment.  Here $\mu\geq0$ is the cellular motility and $d>0$ is the chemical diffusion rate; $\phi$ is the so-called sensitivity function and it measures intensity of chemotactic movement due to the variation of cell population density and chemical concentration; $k$ describes the chemical creation and degradation rate and it can depend on the cellular density and chemical concentration in general.

\subsection{Keller-Segel Chemotaxis Models}
Though bacteria may behave independently, their distribution exhibits regularities through collective behaviors at population level.  One of the most impressive experimental findings in bacterial chemotaxis is the self--organized cellular aggregation that initially evenly distributed cells can sense and move along chemical distribution, and group with others into one or several spatial aggregates eventually.  The Keller--Segel type models can capture such aggregation behavior through the blow-up of cell density in finite or infinite time \cite{Nanjundiah1973,CP1981,HV1996,HW2005}, or the concentrating profiles of stationary solutions \cite{W,CKWW,WXu,NT1,NT2,Schaaf}, and they have achieved great academic success over the past few decades.

The studies of stationary solutions with large amplitude were initiated by the seminal works of C.-S Lin, W.-M. Ni and I. Takagi \cite{LNT,NT1,NT2}.  They considered the stationary system of (\ref{11}) with logarithmic sensitivity and linear chemical creation and degradation rate over a general multi--dimensional bounded domain $\Omega \subset\mathbb R^N$, $N\geq1$ of the following form
\begin{equation}\label{12}
\left\{\begin{array}{ll}
\nabla\cdot(\mu \nabla u-\chi u\nabla\ln v)=0, &\textbf{x} \in\Omega, \\
\varepsilon^2 \Delta v-v+u=0,&\textbf{x} \in\Omega,\\
\partial_\nu u=\partial_\nu v=0, & \textbf{x} \in \partial \Omega,\\
\int_\Omega ud \textbf{x}=\int_\Omega vd \textbf{x} =\bar u|\Omega|,
\end{array}\right.
\end{equation}
where $\bar u$ is the fixed average population density from the conserved total population $\int_\Omega u(\textbf{x},t)d \textbf{x}=\int_\Omega u_0(\textbf{x})d \textbf{x}=\bar u|\Omega|$ in the time--dependent system.  System (\ref{12}) can be solved in light of the following Neumann boundary value problem with $p:=\frac{\chi}{\mu}$
\begin{equation}\label{nbvp}
\left\{\begin{array}{ll}
\varepsilon^2 \Delta w-w+w^p=0,&\textbf{x} \in\Omega,\\
\partial_\nu w=0, & \textbf{x} \in \partial \Omega,
\end{array}\right.
\end{equation}
because if $w$ is a positive solution of (\ref{nbvp}), the pair $(u,v)$ given by
\[u:=\Bigg(\frac{1}{\bar u|\Omega|}\int_\Omega w(\textbf{x})^p d \textbf{x}\Bigg)^{-1}w^p, v:=\Bigg(\frac{1}{\bar u|\Omega|}\int_\Omega w(\textbf{x}) d \textbf{x}\Bigg)^{-1}w\]
is a solution of (\ref{12}).  Assuming that $p\in(1,\infty)$ for $N=1,2$ and $p\in(1,(N+2)/(N-2))$ for $N\geq3$, they proved in \cite{LNT} that (\ref{nbvp}) has only constant solution if $\varepsilon$ is large, and it admits nontrivial solutions if $\varepsilon$ is small, which are critical points of an energy functional in certain Sobolev space.  Moreover, they proceeded to prove in \cite{NT1,NT2} that if $\varepsilon$ is sufficiently small, the least energy solution $w_\varepsilon$ must achieve its unique local (hence global) maximum at a single boundary point $\textbf{x}_\varepsilon\in\partial \Omega$; furthermore, as $\varepsilon\rightarrow 0^+$, $\textbf{x}_\varepsilon\rightarrow \textbf{x}_0\in\partial \Omega$, where the mean curvature of the boundary achieves its maximum.  Since then (\ref{nbvp}) has been extensively studied by various authors, and the readers can find its further development in \cite{delpinoJEMS2014,GuiDuke1996,GWJDE1999,LNWCPAM2007,WangProca1995,WeiJDE1997} and the references therein.  It seems necessary to mention that this approach heavily depends on the smallness of chemical diffusion rate $\varepsilon$, hence requires a primitive understanding about the ``ground-state" of the whole space counterpart of (\ref{nbvp}); moreover, this method is not applicable in general when cellular growth is considered since (\ref{12}) with cellular growth can not be converted into a single equation any more.

In \cite{W}, X. Wang initiated a completely different approach to tackle this model directly without converting it into a single equation.  With the aid of the global bifurcation theories \cite{R,PR,SW}, this technique is further developed and successfully applied to a wide class of Keller--Segel models in \cite{CKWW,W,WXu,Lhc}.  They take $\chi$ as a bifurcation parameter and show that the first bifurcation branch must extend to right infinity without intersecting with the $\chi$--axis, which implies the existence of nonconstant steady states whenever $\chi$ surpasses a critical threshold value, given explicitly in terms of system parameters.  Moreover, by Helly's compactness theorem, they obtained the spiky and layer structures of the steady states when the chemotaxis rate is sufficiently large (compared to the cell motility rate).  The stability and dynamics of these spiky solutions are investigated in \cite{CHWWZ,ZCHLQ}.  Though this approach is currently restricted to 1D settings, one of the advantages over the methodology in \cite{LNT} is its applicability in problems concerning cellular growth \cite{ZAMP,WYZ,KRM}.

One very important extension of (\ref{11}) is to include density-dependent diffusion, assuming that cellular dispersal is anti--crowding and recedes as the population thins.  For example, K. Painter and T. Hillen \cite{PH} proposed and studied the following volume--filling model
 \begin{equation}\label{17}
\left\{
\begin{array}{ll}
u_t=\nabla \cdot (D(u) \nabla u-S(u) \nabla v),&\textbf{x} \in \Omega,t>0, \\
v_t=\Delta v-v+u,&\textbf{x} \in \Omega,t>0,
\end{array}
\right.
\end{equation}
with $D(u)=Q(u)-uQ'(u)$ and $S(u)=uQ(u)$, where $Q(u)$ denotes the density--dependent probability that the cell reaches its neighboring sites.  Systems with other general nonlinear diffusion and sensitivity functions $D(u)$ and $S(u)$ have been extensively studied by various authors \cite{WinklerM2AS,WinklerJDE2010,TW,ISY,IY2,IY,SK}.  In simple words, there exists a criticality of the global existence and blow--up which can be formally stated as follows: assume that $\Omega=\mathbb R^N$, $N\geq2$, or a bounded domain with Neumann boundary conditions on $u$ and $v$.  Denote $\frac{S(u)}{D(u)}\simeq u^\alpha$ for $u$ large, then the index $\frac{2}{N}$ is critical to (\ref{17}) in the sense that its solution exists globally and remains bounded in time if $\alpha<\frac{2}{N}$, while there may exist solutions which blow up in finite time if $\alpha\geq\frac{2}{N}$, in particular when initial cell population is large.  Other extensions such as to include different sensitivity functions, kinetics and cellular growth have been studied by many authors.  It seems necessary to mention that there are many works \cite{Bian-LiuCMP2013,CCPARMA2015,CCVSIAM2015,CHVYInvent,CSIndiana2018,CKYSIAM2013,CKYARMA,DYY,KaibSIAM2017,KYSIAM2012} on (\ref{17}) in the whole space $\mathbb R^N$, the parabolic-elliptic of which becomes a nonlinear aggregation-diffusion equation that serves as a paradigm in studying collective animal behaviors \cite{BTSIAM2011,Mogilner,TopazBertozzi2}.  There are also some works that study the dynamics of (\ref{17}) in bounded domain \cite{CGPSARMA,JiangZAMP,XJMYJDE}.  For a survey of the Keller-Segel models, see \cite{HP} and the references therein.  One recognizes (\ref{01}) in the form of (\ref{17}) with $D(u)=S(u)=u$.  According to \cite{ISY,TW}, system (\ref{01}) has a global weak solution and this global solution is uniformly bounded in time hence blow-up does not occur.

\subsection{Main Results}
In this paper, we study the existence, uniqueness and qualitative properties of radially symmetric stationary solutions of (\ref{01}).  We shall achieve these goals by obtaining the explicit formulas of the radial solutions.  Of our concern is the stationary solutions of (\ref{ss}) that capture the cell aggregation phenomenon within porous medium diffusion, and in particular, we will scrutinize the effects of chemotaxis rate $\chi$ on the dynamics of (\ref{01}) by setting the rest parameters to one.

An immediate consequence of the zero-flux boundary condition is the conservation of cell population
\[M=\int_{B_0(R)}  u(\textbf{x},t)\,d \textbf{x}=\int_{B_0(R)}  u(\textbf{x},0)\,d \textbf{x}=2\pi\int_0^R u(r) r\,dr, \mbox{ for all } t>0,\]
thanks to which both system \eqref{01} and its stationary system \eqref{ss} admit the following constant solution
\[(\bar u,\bar v):=\Big(\frac{M}{\pi R^2},\frac{M}{\pi R^2}\Big).\]

In contrast to previous works \cite{LNT,NT1,NT2,WXu} with non-degenerate diffusion where both $u$ and $v$ are strictly positive, the appearance of degenerate diffusion in the $u$-equation of (\ref{01}) and (\ref{ss}) causes the lack of (strong) maximum principle and $u$ does not have to be strictly positive.  Indeed, we shall see that in most cases $u$ is compactly supported, while only for very special case does one collect positive solutions $(u,v)$.  Here and in the sequel by compactly supported and strictly positive we refer it for $u$ since $v$ is always strictly positive by maximum principle.

For simplicity of notation we introduce
\begin{equation}\label{bifvalue}
\chi_{k}:=\Big(\frac{j_{1,k}}{R}\Big)^2+1,k\in \mathbb N^+,
\end{equation}
where $j_{1,k}$ is the $k$-th positive root of the first kind Bessel function $J_1$ with
\[j_{1,1}\approx3.8317, j_{1,2}\approx7.0156, j_{1,3}\approx 10.1735, j_{1,4}\approx13.3237, j_{1,5}\approx16.4706,...\]
to name the first few explicit values.  According to \cite{CCWWZ}, if $\chi<\chi_1$, the constant solution $(\bar u, \bar v)$ is globally asymptotically stable with respect to (\ref{01}) in the radial setting, therefore throughout this paper we assume $\chi\geq \chi_1$ in order to study its nonconstant radial stationary solutions.

Concerning the stationary problem (\ref{ss}), the first set of our results can be summarized as follows:
\begin{theorem}\label{theorem11}
Let $R>0$ and $M>0$ be arbitrary constants.  The following statements hold:

(i) if $\chi<\chi_1$, (\ref{ss}) has only the constant solution $(\bar u,\bar v)$, and if $\chi>\chi_1,\neq\chi_k,k\geq2$, any solution of (\ref{ss}) must have $u$ be compactly supported in $B_0(R)$;

(ii) for each $\chi=\chi_k,k\in\mathbb N^+$, there exists a one-parameter family of solutions given by
\begin{equation}\label{bifurcation}
(u^{(k)}_\varepsilon(r),v^{(k)}_\varepsilon(r))=(\bar u,\bar v)+\varepsilon(\chi_k,1)J_0\Big(\frac{j_{1,k}r}{R}\Big), r\in (0,R); -\frac{\bar u}{\chi_k}\leq \varepsilon \leq \frac{\bar u}{-J_0(j_{1,1})\chi_k}\Big(\approx\frac{2.482\bar u}{\chi_k}\Big);
\end{equation}
where $u^{(k)}_{\varepsilon}$ is strictly positive in $(0,R)$ for each $k$; moreover, any strictly positive solutions of (\ref{ss}) must be of the form (\ref{bifurcation});

(iii) for each $\chi>\chi_1$, (\ref{ss}) admits a unique pair of nonconstant solutions $(u^-(r),v^-(r))$ and $(u^+(r),v^+(r))$ which are explicitly given by (\ref{interring}) and (\ref{outerring}); moreover, if $\chi\in(\chi_1,\chi_2)$, any nonconstant solution of (\ref{ss}) must be given by one of the pair $(u^\pm,v^\pm)$;

(iv) as $\chi\rightarrow \infty$, the solution $u^-(r)$ converges to the $\delta$-function centered at the origin $r=0$, and $u^+(r)$ converges to the $\delta$-function centered at $r=R$, whereas $v^\pm$ converge to corresponding Green's functions.
\end{theorem}
\emph{(i)}-\emph{(iii)} are presented in the bifurcation diagrams in Figure \ref{branch1}, and the asymptotic behaviors in \emph{(vi)} are illustrated by Figure \ref{innerasymptotic} and Figure \ref{outerasymptotic}.
We would like to point out that the novelty of quadratic diffusion structure is utilized to obtain the explicit formulas for (\ref{ss}) as described above; more importantly, besides the above-mentioned priorities, one can state more about each solution such as its support monotonically shrinks to zero as $\chi$ goes to infinity, the inner ring $u^-$ has a smaller energy than the outer ring $u^+$, while both of them are smaller than that of the constant solution.  Now that our approach is constructive, we can obtain the explicit formula for any radially symmetric solution of (\ref{ss}).  We shall provide details in the coming sections.

Theorem \ref{theorem11} indicates that if $\chi\in(\chi_1,\chi_2)$, radially monotone solutions are uniquely given by the pair $(u^\pm,v^\pm)$ and all nonconstant radial solutions must be one of the pair.  Therefore we can find radially non-monotone solutions only for $\chi\geq\chi_2$.  Indeed, our next main results state that there are (infinitely) many radially non-monotone solutions in this case.
\begin{theorem}\label{theorem12}
Let $R>0$ and $M>0$ be arbitrary constants.  For each $\chi>\chi_2$, the following statements hold:

(i) there exist $\underline R_0(\chi)$ and $\bar R_0(\chi)$ (defined by (\ref{underlineR0})) such that for each $R_0\in[\underline R_0,\bar R_0]$, there exists a non-monotone solution $(u_d(r),v_d(r))$ explicitly given by (\ref{mexicanhat}), where $u_d$ is compactly supported in $[0,R]$, and $v_d$ is monotone decreasing in $(0,R_0)$ and increasing in $(R_0,R)$; moreover, the support of $u_d$ is of the form $[0,r_1)\cup(r_4,R]$ for some $r_1<R_0$ and $r_4<R-R_0$;

(ii) there exists $\chi_2^*>\chi_2$ such that (\ref{ss}) has another solution $(u_i,v_i)$ described as follows

\hspace{4mm} (ii-1) for each $\chi\in(\chi_2,\chi_2^*]$, there exist a unique $\bar R_0$ and non-monotone $(u_i(r),v_i(r))$ explicitly given by (\ref{volcano1}), such that $u_i$ is compactly supported in $[0,R]$, and $v_i$ is monotone increasing in $(0,R_0)$ and decreasing in $(R_0,R)$; moreover, the support of $u_i$ is of the form $(\bar R_0-r_2,R]$ for some $r_2<\bar R_0$;

\hspace{4mm} (ii-2) for each $\chi\in(\chi_2^*,\infty)$, there exist a unique $R_0^*$ and non-monotone $(u_i(r),v_i(r))$ explicitly given by (\ref{volcano2}), such that $u_i$ is compactly supported in $[0,R]$, and $v_i$ is monotone increasing in $(0,R_0^*)$ and decreasing in $(R_0^*,R)$; moreover, the support of $u_i$ is of the form $(R_0^*-r_2,R_0^*+r_3)$ for some $r_2<R_0^*$ and $r_3<R-R_0^*$;

(iii) for each $\chi\in(\chi_2,\chi_3)$, all non-monotone solutions of (\ref{ss}) must be either $(u_d,v_d)$ given in (i) or $(u_i,v_i)$ given in (ii).
\end{theorem}
Non-monotone radial solutions obtained in Theorem \ref{theorem12} are illustrated in Figure \ref{branch2} and Figure \ref{mexicanhatasymptotic}.

It is well known that system (\ref{01}) has the following free energy functional
\begin{equation}\label{freeenergy}
\mathcal E(u(\textbf{x},t),v(\textbf{x},t))=\frac{1}{\chi}\int_{B_0(R)} u^2d \textbf{x}+\int_{B_0(R)} (\vert \nabla v\vert^2+v^2-2 uv)d \textbf{x}
\end{equation}
which is non--increasing along the solution trajectory with its dissipation given by
\begin{equation}\label{dissipation}
  \frac{d\mathcal E}{dt}=-\frac{2}{\chi}\int_{B_0(R)} u |\nabla u -\chi \nabla v|^2d \textbf{x}- 2\int_{B_0(R)} |v_t|^2d \textbf{x}:= \mathcal I\leq 0, \quad \mbox{for all } t>0.
\end{equation}
Moreover, this energy is a Lyapunov functional since steady states $(u_s, v_s)$ are characterized by the zero dissipation $\mathcal I(u_s, v_s)=0$.  We would like to remark that to derive (\ref{dissipation}) one starts with an approximation problem of (\ref{01}) that possesses a unique solution, and then collect this inequality by passing to this approximation limit.  After establishing the explicit solutions of (\ref{ss}), we next provide a hierarchy of free energies of these nontrivial patterns.  Our results indicate that in the radial setting the constant solution is the global minimizer of the free energy, and the inner ring solution $(u^-,v^-)$ has the least energy; moreover, the constant solution has a larger energy than any compactly supported solution of (\ref{ss}).  The significance of our results is that they hold for any $\chi>\chi_1$ without further assumptions on other system parameters.  Another set of our results are summarized in the following:
\begin{theorem}\label{theorem13}
Let $R>0$ and $M>0$ be arbitrary.  Then the following statements concerning the explicit solutions obtained in Theorem \ref{theorem11} and Theorem \ref{theorem12} hold:

(i) for each $\chi>\chi_1$, $\mathcal E(u^-,v^-),\mathcal E(u^+,v^+)<\mathcal E(\bar u,\bar v)=\mathcal E(u^{(k)}_\varepsilon, v^{(k)}_\varepsilon)$, given by (\ref{bifurcation});

(ii) for each $\chi>\chi_2$, let $R_0\in[\underline R_0,\bar R_0]$, $R_0^*$ be given, and the solutions $(u_d(r;R_0),v_d(r;R_0))$ and $(u_i(r),v_i(r))$ be obtained in Theorem \ref{theorem12}.  Then we have
$\mathcal E(u_d,v_d), \mathcal E(u_i,v_i)<\mathcal E(\bar u,\bar v)$; moreover, for $\chi\gg 1$, $\mathcal E(u_d,v_d)|_{R_0=\bar R_0}\rightarrow \mathcal E(u^-,v^-)=-\frac{M^2\ln \chi}{4\pi}+O(1)$ and $\mathcal E(u_d,v_d)|_{R_0=\underline R_0}\rightarrow \mathcal E(u^+,v^+)=-\frac{M^2I_0(R)}{2\pi RI_1(R)}+o(1)$;

(iii) let $(u_\text{cpt},v_\text{cpt})$ be any solution with $u$ being compactly supported, then $\mathcal E(u_\text{cpt},v_\text{cpt})<\mathcal E(\bar u,\bar v)$.
\end{theorem}
Theorem \ref{theorem13} provides a hierarchy of energies of the explicit solutions obtained in this paper.  These statements are illustrated by the energy diagrams in Figure \ref{energy}.  From the viewpoint of the free energy, Theorem \ref{theorem13} suggests that in the radial class the inner ring solution $(u^-,v^-)$ is the most stable, while the constant solution $(\bar u,\bar v)$ and the bifurcation solutions (\ref{bifurcation}) are the most unstable.  Moreover, now that (\ref{01}) admits multiple stationary concentric rings when $\chi$ is large, the configurations with a large inner spiky structure (spike at the origin $r=0$) tend to be more stable than those with a large outer spiky structure (spike on the boundary $r=R$).  We would like to comment that for solution $(u_d(r;R_0),v_d(r;R_0))$, the variation of $R_0$ or $\chi$ does not necessarily induce monotonicity in the configuration $u$ in general.  For instance, one does not expect a larger $\Vert u_d(r;R_0)\Vert_{L^\infty}$ for a larger $\chi$ if $R_0=\underline R_0$.  See the second column of Figure \ref{branch2} for an illustration.  

\subsection{Global Stability v.s. Chemotaxis--Driven Instability}
A striking difference between (\ref{01}) and its counterpart in the whole space is that the former has the constant pair $(\bar u,\bar v)$ as a solution.  According to \cite{CCWWZ}, this constant solution is globally asymptotically stable if $\chi<\chi_1$ and it becomes unstable as $\chi$ surpasses $\chi_1$.

First of all, we show that if $\chi>\chi_1,\neq \chi_k$, then any solution $u$ must be of compact support in $B_0(R)$.  To see this, we argue by contradiction and assume that $u>0$ in $B_0(R)$ for some $\chi>\chi_1,\neq \chi_k$.  Then the $u$-equation implies that $u-\chi v$ equals some constant in $B_0(R)$ and the $v$--equation becomes
\[\left\{\begin{array}{ll}
\Delta_r v+(\chi-1)(v-\bar v)=0, & r \in (0,R),\\
v\in C^2((0,R))\cap C^1([0,R]), v(r)>0,& r \in (0,R),\\
\partial_r v (r)=0,& r=0,R.
\end{array}
\right.
\]
Therefore $(v-\bar v,\chi-1)$ is an eigen--pair of $-\Delta_r$ in the Neumann radial class, hence we must have that $\chi=\chi_k$, which is a contradiction to our assumption, therefore $u$ must be compactly supported.  In the next, we will obtain explicit formulas of these solutions with compact support for each $\chi>\chi_1$.  However, when $\chi=\chi_k$, (\ref{ss}) also has solutions which are positive in $B_0(R)$.  Indeed, according to our discussions above, when $\chi=\chi_k$, $v-\bar v$ is a multiplier of the Neumann eigen-function $J_0(\frac{j_{1,k}r}{R})$ of Laplacian.  Therefore, we have proved the following results.
\begin{lemma}\label{lemma11}
Let $(u(r),v(r))\in C^0([0,R])\times C^2((0,R))$ be any solution of (\ref{ss}).  Then for each $\chi>\chi_1$, $\neq\chi_k$, $k\in \mathbb N^+$, we must have that $u(\text{x})$ is compactly supported in $B_0(R)$; moreover, when $\chi=\chi_k$, (\ref{ss}) admits a one-parameter family of positive solutions $(u^{(k)}_\varepsilon(r),v^{(k)}_\varepsilon(r))$ explicitly given by (\ref{bifurcation}).
\end{lemma}
It is perhaps worthwhile mentioning that one can apply bifurcation theory to establish nonconstant solutions of (\ref{ss}) out of $(\bar u,\bar v)$.  To see this, let us rewrite it into the abstract form $\mathcal F(u,v,\chi)=0$ by treating $\chi$ as the bifurcation parameter
\[
\mathcal{F}(u,v,\chi)=
\begin{pmatrix}
(ru(u-\chi v)_r)_r \\
v_{rr}+\frac{1}{r}v_r-v+u\\
\int_{B_0(R)}ud\textbf{x}-M
\end{pmatrix},\mathcal X\times\mathcal X\times\mathbb R\rightarrow \mathcal Y\times \mathcal Y\times\mathbb R,
\]
with $\mathcal{X}:=\{w \in H^2((0,R)) \vert w'(0)=w'(R)=0\}$ and $\mathcal Y:=L^2((0,R))$.  Then it is not difficult to show that its Fr$\acute{\text{e}}$chet derivative $D _{(u,v)} \mathcal F(u_0,v_0,\chi)$ is a Fredholm operator with zero index for any $(u_0,v_0)$ in a small neighbourhood of $(\bar u,\bar v)$, and $D_{(u,v)}\mathcal F(\bar u,\bar v,\chi)$ satisfies the so--called transversality condition.  Therefore by the well-known Crandall-Rabinowtiz theorem on the bifurcation from simple eigenvalue \cite{CR,SW}, bifurcation occurs at $(\bar u,\bar v,\chi_k)$ for each $k\in\mathbb N^+$, and there exist a constant $\delta>0$ and continuous functions $(u_k(r,s),v_k(r,s),\chi_k(s)): s\in(-\delta, \delta) \rightarrow \mathcal X\times\mathcal X\times \mathbb R^+$ such that any solution of (\ref{ss}) around $(\bar u,\bar v,\chi_k)$ must be of the form $(u_k,v_k)=(\bar u,\bar v)+s(\chi_k,1)J_0(j_{1,k}r/R)+\mathcal O(s^2), \chi_k(s)=\chi_k+O(s)$.  Indeed, one can further show by fitting this solution with the equations of (\ref{ss}) that both $\mathcal O(s^2)$ and $O(s)$ are identically zeros, hence this bifurcating expansion is reduced to (\ref{bifurcation}), which presents the unique solution around $(\bar u,\bar v,\chi_k)$.  Note that the bifurcation solution has a small amplitude when $\chi$ is around $\chi_1$, therefore to look for large amplitude solutions one needs to study (\ref{ss}) with large $\chi$, when the bifurcation branches are far way from $(\bar u,\bar v,\chi_1)$.  However, the abstract global bifurcation does not apply anymore due to the curse of diffusion degeneracy.  In this paper, we obtain the explicit formulas for solutions of (\ref{ss}), thanks to which the global bifurcation diagrams naturally emerge; moreover, among others our results indicate that each branch is neither pitch-fork nor transcritical locally, and its continuum connects the bifurcation solutions with the stationary solutions that are compactly supported.

\subsection{Paper Organization}
The rest part of this paper is organized as follows.  In Section \ref{section2}, for each $\chi>\chi_1$ we first obtain explicit solutions of (\ref{ss}) such that $u$ is radially monotone decreasing/increasing within its support.  This gives rise to the so-called inner/outer ring solutions.  Moreover, we show that as $\chi$ tends to infinity, both solutions converge to a Dirac-delta function, the former centered at the origin and the latter centered on the boundary.  In Section \ref{section3}, we study two types of non--monotone solutions, i.e., the so-called Mexican-hat and Volcano solutions.  Our results readily imply that there are infinitely many radial solutions once $\chi>\chi_2$.  Asymptotic behaviors of these solutions are also established in the large limit of chemotaxis rate.  Section \ref{section4} is an extension of the previous sections from simple and lower modes to complex and higher modes for large $\chi$.  In particular, for an arbitrarily given but fixed $\chi$, we classify all solutions of (\ref{ss}) in terms of the size of this parameter.  In simple words, an intense chemotaxis gives rise to airy patterns that have a bright central region or circle in the middle, surrounded by a sequence of concentric rings.  As an immediate consequence of the explicit solutions, section \ref{section5} is devoted to the analysis of (\ref{ss}) over the whole space $\mathbb R^2$.  We show that the problem in the whole space has a unique radial solution which is to be explicitly given; moreover, this solution is radially monotone with $u$ being compactly supported in a disk.  Thanks to the radial symmetry result in \cite{CHVYInvent}, our results imply that this solution is actually the only stationary solution to (\ref{01}) over the whole space $\mathbb R^2$.  In Section \ref{section6}, we calculate the free energies of the stationary solutions obtained above.  Our results indicate that the inner ring has the least energy among all steady states, while the constant solution has the largest energy.  Moreover, we provide some numerical evidence on the existence of non-radial stationary solutions of (\ref{01}), whereas the theoretical analysis is restricted to the radial setting.  Finally, we include in Section \ref{section7} several important facts needed for the previous analysis.

\section{Radially Monotone Solutions}\label{section2}

Lemma \ref{lemma11} indicates that $u$ is compactly supported in $[0,R]$ whenever $\chi>\chi_1$, $\neq \chi_k$.  In this section, we will construct explicit radially decreasing and increasing solutions of (\ref{01}), which we call the (single) inner ring solution and the outer ring solution, respectively.  These radially monotone solutions serve as blocks that we use to construct (all) the non-monotone solutions of (\ref{ss}) other than those obtained in (\ref{bifurcation}).  We then proceed to study their qualitative and quantitative properties with respect to the size of chemotaxis rate.  Before proceeding, we first illustrate some of our main results in this section in Figure \ref{branch1}.
\begin{figure}[h!]
\vspace{-3mm}  \centering
  \includegraphics[width=1\textwidth]{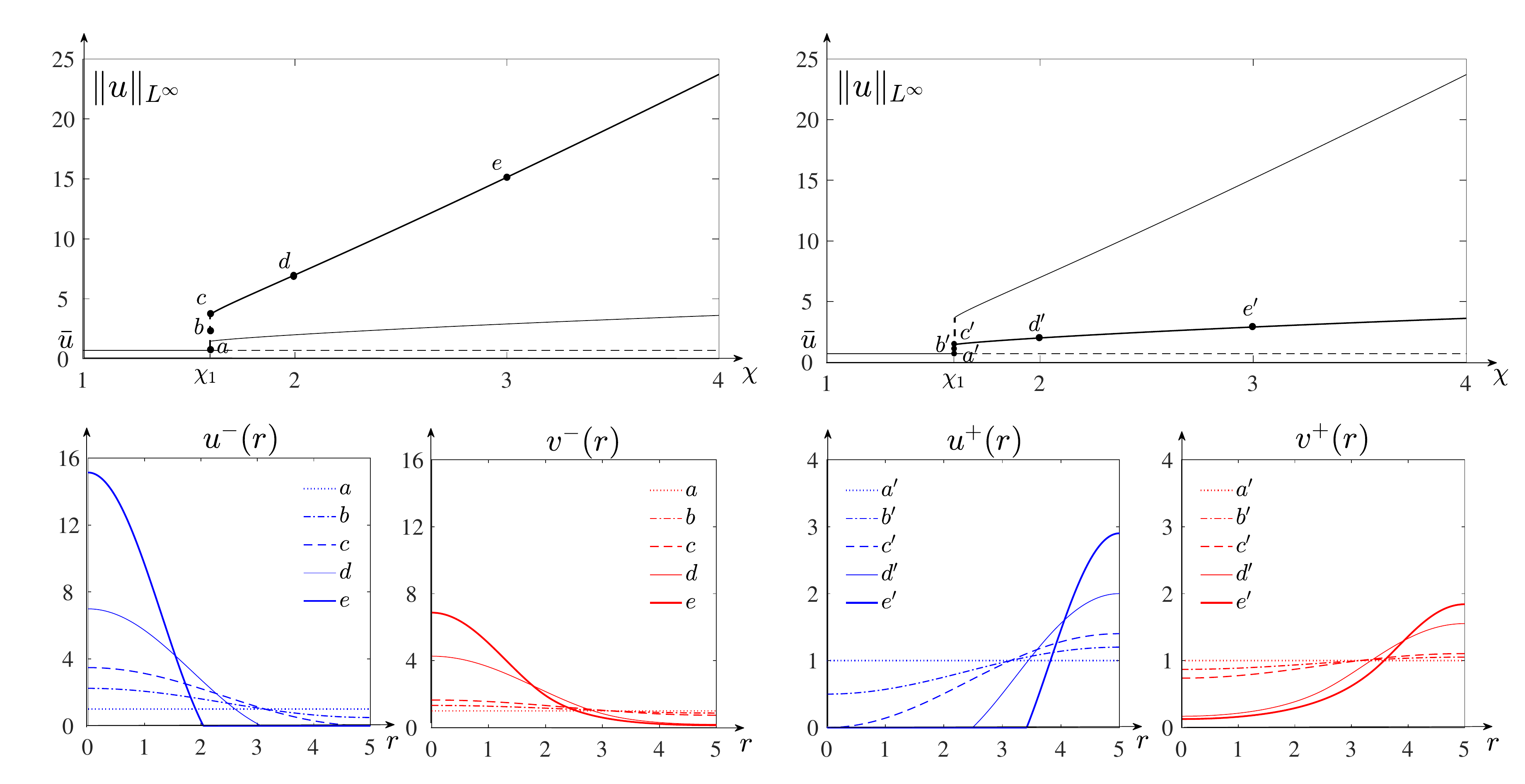}\vspace{-5mm}
  \caption{The local bifurcation branches (solutions) at $\chi=\chi_1$ and their global continuums.   For each $\chi>\chi_1$, (\ref{ss}) has a unique (pair) of monotone solutions and $u$ must be compactly supported.  \textbf{Left Column:}  When $\chi=\chi_1$, (\ref{bifurcation}) for an interval of $\varepsilon$ there gives a family of monotone decreasing solutions for (\ref{ss}).  This is represented by the vertical (dashed) line over $\chi_1$, and three such positive solutions $a,b$ and $c$ are plotted.  As $\chi$ surpasses $\chi_1$, solution $u$ becomes compactly supported and it stays so along the global branch.  \textbf{Right Column}:  Similarly, the positive solutions in (\ref{bifurcation}) are plotted at $a',b'$ and $c'$; moreover, for each $\chi>\chi_1$, outer ring $u$ also stays compactly supported along the global branch.}\label{branch1}
\end{figure}

\subsection{Inner Ring Solution}
We first look for the radially decreasing solution such that $u$ is supported in $[0,r_1)$ with $r_1<R$ to be determined.  We denote this radial solution as $(u^-,v^-)$ and call it the inner ring solution or inner ring for brevity, since $u^-$ is configured as an inner ring supported over the disk $B_0(r_1)$ for the original problem (\ref{01}).  We shall see that this extends the first bifurcation branch given by (\ref{bifurcation}) at $\varepsilon=\bar u/(-J_0(j_{1,1})\chi_1)$.  Now that $u^->0$ in $[0,r_1)$, there exists some constant $\bar{C}$ to be determined such that $u^- -\chi v^-=\bar{C}$ for $r\in[0,r_1)$ and $u\equiv 0$ for $r\in[r_1,R]$, hence the $v$--equation implies
\begin{equation}\label{22}
\left\{\begin{array}{ll}
v^-_{rr}+\frac{1}{r}v^-_r+(\chi-1)v^-+\bar{C} =0,& r\in (0,r_1),\\
v^-_{rr}+\frac{1}{r}v_r-v=0,& r\in[r_1,R],\\
v^-_r(r)=0,&r=0,R.
\end{array}
\right.
\end{equation}
Let us denote $\omega:=\sqrt{\chi-1}$, then solving (\ref{22}) gives in terms of constants $C_i$ and $\bar C$
\begin{equation}\label{ir}
v^-(r)=\left\{\begin{array}{ll}
C_1J_0(\omega r)-\frac{\bar{C}}{\chi-1},&r\in (0,r_1),\\
C_2T_0(r;R),&r\in(r_1,R),
\end{array}
\right.
\end{equation}
where $J_0$ is the Bessel function of the first kind and $T_0$ is the following compound Bessel function
\begin{equation}\label{T0}
 T_0(r; R):=K_1(R)  I_0(r)+I_1(R)  K_0(r),
\end{equation}
with $I_\alpha$ and $K_\alpha$ ($\alpha=0,1$) being modified Bessel functions of the first and second kind.

To determine $C_1,C_2$ and $\bar C$ in \eqref{ir}, we first find $r_1$ by enforcing the continuities of $v'(r)$ and $v''(r)$ at $r=r_1$
\[
\left\{\begin{array}{ll}
-C_1\omega J_1(\omega r_1)=C_2T_1(r_1;R),\\
-C_1\omega^2\Big(J_0(\omega r_1)-\frac{J_1(\omega r_1)}{\omega r_1}\Big)=C_2 \Big(T_0(r_1;R)-\frac{T_1(r_1;R)}{r_1}\Big),
\end{array}
\right.
\]
where $T_1(r; R):=\partial_r T_0(r; R)=K_1(R)  I_1(r)-I_1(R)  K_1(r)$.  Note that we always have $T_0(r;R)>0$ and $T_1(r;R)<0$ in $[0,R]$.  Therefore these identities hold if and only if $r_1$ is a root of the algebraic equation
\begin{equation}\label{23}
f(r_1; \omega,R):=\frac{\omega J_0(\omega r_1)}{J_1(\omega r_1)}-\frac{T_0(r_1; R)}{T_1(r_1; R)}=0.
\end{equation}
We now give the following lemma which promises the solvability of (\ref{23}) for some $r_1\in(\frac{j_{0,1}}{\omega},\frac{j_{1,1}}{\omega})$.
\begin{lemma}\label{lemma21}
For each $\chi>\chi_1$ the function $f(r;\omega,R)$ in (\ref{23}) admits a unique root $r_1$ in $(\frac{j_{0,1}}{\omega},\frac{j_{1,1}}{\omega})$, where $j_{0,1}\approx 2.4048$ and $j_{1,1}\approx 3.8317$ are the first positive roots of the Bessel functions $J_0$ and $J_1$.
\end{lemma}
\begin{proof}
To prove this, we first see that $f(r; \omega,R)>0$ for all $r\in[0,\frac{j_{0,1}}{\omega}]$ hence (\ref{23}) admits no root in this interval.  On the other hand, one can find that $f\Big(\frac{j_{0,1}}{\omega};\omega,R\Big)=-\frac{T_0(\frac{j_{0,1}}{\omega}; R)}{T_1(\frac{j_{0,1}}{\omega}; R)}>0$ and $f\Big(\Big(\frac{j_{1,1}}{\omega}\Big)^-;\omega,R\Big)=-\infty$, therefore there exists at least one root $r_1\in (\frac{j_{0,1}}{\omega},\frac{j_{1,1}}{\omega})$ since $f(r;\omega,R)\in C^\infty((\frac{j_{0,1}}{\omega},\frac{j_{1,1}}{\omega}))$.  To show the uniqueness of $r_1$, we calculate
 \begin{equation}\label{24}
 f_r(r;\omega,R)= -\omega^2-1+f(r;\omega,R)\overbrace{\left(-\frac{\omega J_0(\omega r)}{J_1(\omega r)}-\frac{T_0(r; R)}{T_1(r; R)}+\frac{1}{r} \right)}^{>0 \text{~for any~}r\in (\frac{j_{0,1}}{\omega},\frac{j_{1,1}}{\omega})}.
 \end{equation}
Now we argue by contradiction and assume that $\tilde{r_1}$ and $\tilde{r_2}$ are two ordered adjacent roots of $f(r;\omega, R)$ in $(\frac{j_{0,1}}{\omega},\frac{j_{1,1}}{\omega})$, i.e., $f(r;\omega,R)$ is of one sign in $(\tilde r_1,\tilde r_2)$.  This implies $f_r(\tilde{r_1};\omega,R) f_r(\tilde {r_2};\omega,R)\leq 0$ thanks to the continuity of $f_r$; however, $f_r(\tilde{r_1};\omega,R)  f_r(\tilde{r_2};\omega,R)=(\omega^2+1)^2>1$ according to (\ref{24}), which is a contradiction.  Therefore $f(r;\omega, R)$ admits a unique root $r_1\in(\frac{j_{0,1}}{\omega},\frac{j_{1,1}}{\omega})$ as claimed.
\end{proof}

\begin{remark}\label{remark21}
For $\chi \in (\chi_1,\chi_2)$, Lemma \ref{lemma21} states that there exists a unique $r_1$ such that $f(r_1;\omega, R)=0$ and this root $r_1\in (\frac{j_{0,1}}{\omega},\frac{j_{1,1}}{\omega})$; however, for $\chi \in[\chi_2,\infty)$, the function $f(r;\omega,R)$ admits multiple roots besides $r_1$. For instance, one can always find $\bar{r}_2\in (\frac{j_{0,2}}{\omega},\frac{j_{1,2}}{\omega})$ such that $f(\bar{r}_2;\omega,R)=0$; indeed in each interval $(\frac{j_{0,k}}{\omega},\frac{j_{1,k}}{\omega})$ one can find a unique $\bar r_k$ such that $f(\bar{r}_k;\omega,R)=0$.  However, all these roots are ruled out since we look for non--negative solutions.  If not, say we choose the root $\bar r_2$ as the size of support, then $u(\bar{r}_2)=0$ and the solution takes the form $u^-(r)=\bar{\mathcal A_1}(J_0(\omega r)-J_0(\omega \bar{r}_2))$ for $r\in (0,\bar{r}_2)$, then one can find from the monotonicity of $J_0$ that $u^-(r)<0$ for $r\in(\frac{j_{0,2}}{\omega}, \bar{r}_2)$, and this is not biologically realistic.  Similarly, one can show that the other roots of $f(r;\omega,R)=0$ other than $r_1$ are not applicable, if they exist at all.  Therefore $r_1\in (\frac{j_{0,1}}{\omega},\frac{j_{1,1}}{\omega})$ is always the unique root that we look for in (\ref{ir}).
\end{remark}

We would like to point out that for $\chi<\chi_1$, $f(r;\omega,R)$ is strictly positive in $(0,R)$ hence admits no root.  Indeed, (\ref{ss}) has only constant solution $(\bar u,\bar v)$ in this case according to our discussions above.

With $r_1$ obtained through (\ref{23}) in Lemma \ref{lemma21}, we readily have from the fact $u=\chi v+\bar C$ that $u^-(r)=\mathcal A_1\big(J_0(\omega r)-J_0(\omega r_1)\big)$ in $(0,r_1)$, where $\mathcal A_1$ is a positive constant determined through the conservation of total cell population $2\pi\int_0^{r_1} u^-(r)rdr=2\pi\mathcal A_1\Big(\frac{r}{\omega}J_1(\omega r)-\frac{r^2}{2}J_0 (\omega r_1) \Big)\big|_0^{r_1} =M$ and is explicitly given by
\begin{equation}\label{26}
\mathcal A_1=\frac{M \omega}{\pi \left(2r_1J_1(\omega r_1)-\omega r_1^2J_0(\omega r_1)\right)}\Bigg(=\frac{M}{\pi r_1^2J_2(\omega r_1)}\Bigg).
\end{equation}
Since $\omega r_1 \in (j_{0,1},j_{1,1})$, one has that $\mathcal A_1>0$ is well-defined.

To find $v^-(r)$, we combine (\ref{ir}) with the fact that $u^-=\chi v^-+\bar{C}$ for $r\in[0,r_1]$ to see that $\mathcal A_1\big(J_0(\omega r)-J_0(\omega r_1)\big)=\chi C_1 J_0(\omega r)-\frac{\bar{C} }{\chi-1}$.  Therefore $C_1=\frac{\mathcal A_1}{\chi}$ and $\bar{C} =\mathcal A_1(\chi-1)J_0(\omega r_1)$, i.e., $\bar{C} =\frac{M \omega^2J_0(\omega r_1)}{\pi r_1^2J_2(\omega r_1)}$.  Moreover, by the continuity of $v^-(r)$ at $r=r_1$, we equate $\mathcal A_1 (\frac{1}{\chi}-1) J_0(\omega r_1)$ with $\mathcal B_1 T_0(r_1;R)$ and obtain
\begin{equation}\label{27}
\mathcal B_1=\frac{-M\omega^2J_0(\omega r_1)}{\chi\pi r_1^2J_2(\omega r_1)T_0(r_1;R)}.
\end{equation}
To conclude, we find the solution of (\ref{ss}) described above is explicitly given by
\begin{equation}\label{interring}
u^-(r)=\left\{\begin{array}{ll}
\!\!\mathcal A_1\big(J_0(\omega r)-J_0(\omega r_1)\big),\!\!&\!\!r\in[0,r_1),\\
\!\!0,\!\!&\!\!r\in[r_1,R],
\end{array}
\right.
v^-(r)=\left\{\begin{array}{ll}
\!\!\mathcal A_1\big(\frac{J_0(\omega r)}{\chi}-J_0(\omega r_1)\big),\!\!&\!\!r\in[0,r_1),\\
\!\!\mathcal B_1 T_0(r;R),\!\!&\!\!r\in[r_1,R],
\end{array}
\right.
\end{equation}
where $\mathcal A_1$ and $\mathcal B_1$ are given by (\ref{26}) and (\ref{27}), and $T_0(r;R)$ is (\ref{T0}).  We would like to point out that both $u$ and $v$ achieve the unique maximum at the origin, and are radially decreasing within their support.  We refer to (\ref{interring}) as the inner ring solutions.  One sees that they extend the first (local) bifurcation branch at $\chi=\chi_1$, the global continuum of which extends to infinity and gives rise to a unique inner ring solution for each $\chi>\chi_1$.  This is illustrated in Figure \ref{branch1}.

\subsubsection{Asymptotic behavior of inner ring in the limit of $\chi\rightarrow \infty$}
We now study the effect of large chemotaxis rate $\chi$ on the qualitative behaviors of the steady state $(u^-,v^-)$ given by (\ref{interring}).  First of all, we know from above that the size $r_1$ of support of $u^-$ depends on $\chi$ continuously, then the fact $r_1\in(\frac{j_{0,1}}{\omega},\frac{j_{1,1}}{\omega})$ readily implies that $r_1\rightarrow0^+$ as $\chi\rightarrow \infty$; moreover, one infers from the conservation of total population that $u^-(r)\rightarrow M\delta_0(r)$, where $\delta_0(r)$ is the Dirac delta function centered at the origin, and $v^-(r)\rightarrow \frac{M}{2\pi I_1(R)}T_0(r;R)$ accordingly.  These asymptotic behaviors suggest that intense chemotactic movement contributes the formation of spiky structures of $u$ and $v$; moreover, it gives the global extension of the first bifurcation branch(es) as we shall see in Figure \ref{branch1}.

We can provide some refined asymptotic properties of the profile due to, again, the explicit formula (\ref{interring}).  First of all, we show that the support of $u^-$ shrinks as the chemotaxis intensifies.  With that being said, we will show that $r_1$ is strictly decreasing in $\chi$ as $\frac{\partial r_1}{\partial \omega}<0$ for $\omega>0$.  To this end, one calculates to find
\begin{align}
  \frac{\partial r_1}{\partial \omega} =&-\frac{f_{\omega}(r_1;\omega,R)}{f_r(r_1;\omega,R)}= \frac{f_{\omega}(r_1;\omega,R)}{\omega^2+1} \notag \\
   =&\frac{J_0(\omega r_1)}{(\omega^2+1)J_1(\omega r_1)}+\frac{y_1(r_1;\omega)}{(\omega^2+1)J_1^2(\omega r_1)}, \notag
\end{align}
where $y_1(r;w):=-\omega r J_1^2(\omega r)-\omega r J_0^2(\omega r)+J_0(\omega r)J_1(\omega r)$; by further computations we have $y'_1(r;\omega)=-\frac{J_0(\omega r)J_1(\omega r)}{r}\geq0$ in $(\frac{j_{0,1}}{\omega},\frac{j_{1,1}}{\omega})$, therefore $y_1(r_1;\omega)\leq y_1(\frac{j_{1,1}}{\omega};\omega)<0$ and $\frac{\partial r_1}{\partial \omega}<0$ as expected.

We next show that the maximum of $u^-(r)$ is strictly increasing in $\chi$ with $\frac{\partial \Vert u^-\Vert_{L^\infty}}{\partial \omega}>0$.  Rewrite $\Vert u^- \Vert_{L^\infty}=\frac{M}{\pi}\frac{\omega^2 (1-J_0(z))}{2zJ_1(z)-z^2J_0(z)}$ with $z=\omega r_1$, then we find
\begin{align}
\frac{\pi}{M}\frac{\partial \Vert u^-\Vert_{L^\infty}}{\partial \omega}
=&\frac{2\omega (1-J_0(z))}{2zJ_1(z)-z^2J_0(z)}+\frac{\omega^2J_1(z)(2J_1(z)-z)}{z(2J_1(z)-zJ_0(z))^2}\cdot\frac{\partial z}{\partial\omega}\nonumber
\end{align}
which, in light of the identity $\frac{\partial z}{\partial\omega }=r_1+\omega \frac{\partial r_1}{\partial\omega }$, implies
\begin{align}
\frac{\pi}{M}\frac{\partial \Vert u^- \Vert_{L^\infty}}{\partial \omega}
=&\frac{2\omega(1-J_0(z))(2J_1(z)-zJ_0(z))+\omega zJ_1(z)(2J_1(z)-z)}{z(2J_1(z)-zJ_0(z))^2}  +\frac{\omega^3J_1(z)(2J_1(z)-z)}{z(2J_1(z)-zJ_0(z))^2}\cdot \frac{\partial r_1}{\partial \omega}\nonumber\\
\geq &\frac{2\omega(1-J_0(z))(2J_1(z)-zJ_0(z))+\omega zJ_1(z)(2J_1(z)-z)}{z(2J_1(z)-zJ_0(z))^2},\nonumber
\end{align}
where we have applied the facts $\frac{\partial r_1}{\partial \omega}<0$ and $2J_1(z)\leq z$ for $z\in(j_{0,1},j_{1,1})$ for the inequality.  Let us denote
\[y_2(z):=2(1-J_0(z))(2J_1(z)-zJ_0(z))+z J_1(z)(2J_1(z)-z).\]
Now, in order to prove $\frac{\partial \Vert u\Vert_{L^\infty}}{\partial \omega}>0$, it suffices to show that $y_2(z)>0$ for $z\in(j_{0,1},j_{1,1})$.  By straightforward calculations we find
\begin{align*}
&y_2'(z)=\frac{2zJ_1^2(z)+(z^2-4)J_1(z)+4J_0(z)J_1(z)-2zJ_0^2(z)-z(z^2-2)J_0(z)}{z},\\
&y_2''(z)=\frac{4zJ^2_0(z)-\left(z^3+4z-6z^2J_1(z)+8J_1(z)\right)J_0(z)}{z^2}+\frac{\left(z^4-2z^2+8-8zJ_1(z)+2z^2J_0(z)\right)J_1(z)}{z^2}.
\end{align*}
Thanks to the facts that $|J_0(z)|<1$ and $|J_1(z)|<\frac{1}{\sqrt{2}}$ for $z\in(j_{0,1},j_{1,1})$, a lengthy but straightforward calculation gives $z^3+4z-6z^2J_1(z)+8J_1(z)>0$ and $z^4-2z^2+8-8zJ_1(z)+2z^2J_0(z)>0$, which lead to $y_2''(z)>0$.  This conclusion in conjunction with the results that $y_2(j_{0,1})=(4+2j_{0,1}-j_{0,1}^2) J_1(j_{0,1})>0$ and $y'_2(j_{0,1})=\frac{2j_{0,1}J^2_1(j_{0,1})+(j_{0,1}^2-4)J_1(j_{0,1})}{j_{0,1}}>0$ imply $y'_2(z)>0$ and $y_2(z)>0$ for all $z\in (j_{0,1},j_{1,1})$ as expected.  This finishes the proof.

Before proceeding further, we establish the asymptotic $\Vert u^- \Vert_{L^\infty}=O(\omega^2)$ by first showing that $z:=\omega r_1\rightarrow (j_{0,1})^+$ as $\chi\rightarrow \infty$.  If not, say $\omega r_1\rightarrow$ some $\theta\in (j_{0,1},j_{1,1}]$ as $\chi \to \infty$, then $f(r_1;\omega;R)$ given by (\ref{23}) becomes negative for $\chi$ sufficiently large since $r_1\in (\frac{j_{0,1}}{\omega},\frac{j_{1,1}}{\omega})$.  This is a contradiction hence $z \rightarrow j_{0,1}$ as $\chi\rightarrow \infty$.  Therefore we have that
\[\Vert u^-\Vert_{L^\infty}=\frac{M}{\pi}\frac{\omega^2 (1-J_0(z))}{2zJ_1(z)-z^2J_0(z)}=\frac{M\chi}{2\pi j_{0,1}J_1(j_{0,1})}+O(1) \text{~as~} \chi\rightarrow \infty.\]

Finally, since $r_1\rightarrow 0^+$ and $\omega r_1\rightarrow (j_{0,1})^+$ as $\chi\rightarrow \infty$, we have that $\mathcal A_1\rightarrow \infty$ hence $u^-(r)\rightarrow u_\infty=M \delta_0(r)$ pointwisely in $[0,R]$.  On the other hand, $v^-(r)\rightarrow v_\infty(r)=\mathcal B_\infty T_0(r;R)$ pointwisely in $[0,R]$, where $\mathcal B_\infty=\frac{M}{2\pi I_1(R)}$ thanks to the conservation of mass.

\begin{figure}[h!]\vspace{-5mm}
\centering
\includegraphics[width=1\textwidth]{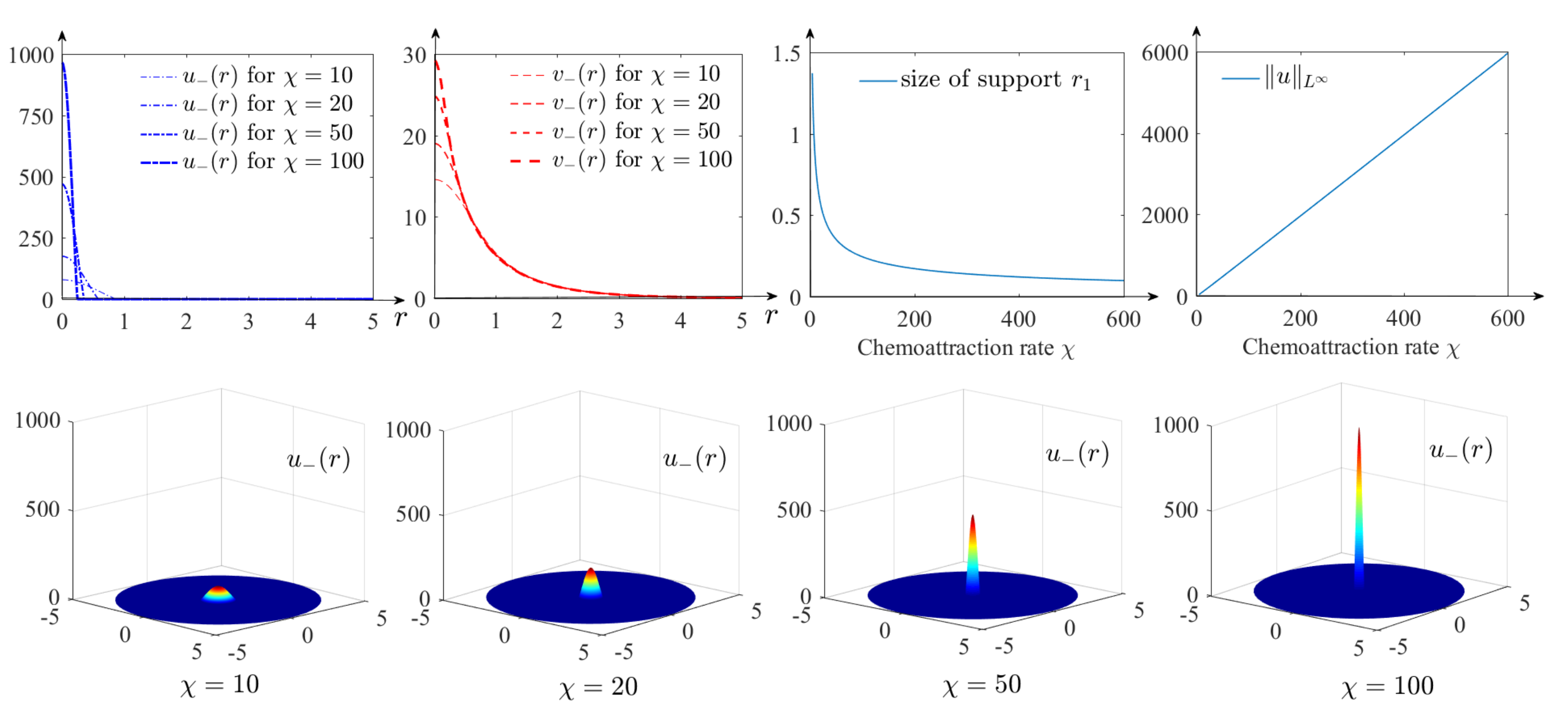}\vspace{-3mm}
\caption{\textbf{Top}: the inner ring solution $(u^-,v^-)$ given by (\ref{interring}) with $M=25\pi$ and $R=5$ for $\chi=10$, 20, 50 and 100.  It is observed that $u^-(r)\rightarrow 25\pi \delta_0(r)$ and $v^-(r)\rightarrow 0.5136 T_0(r;5)$ as $\chi\rightarrow \infty$; moreover, the size $r_1$ of support of $u^-(r)$ shrinks as $O(\frac{1}{\omega})$ or $O(\frac{1}{\sqrt{\chi}})$ and the magnitude of $u^-(r)$ expands as $O(\omega^2)$ or $O(\chi)$.  \textbf{Bottom}: the development of interior spikes in configuration $u$ as $\chi$ expands.  Evidently, a large chemotaxis rate promotes the formation of an interior spike that describes the cellular aggregation.}\label{innerasymptotic}
\end{figure}
Summarizing the facts above, we have proved the following results
\begin{proposition}\label{proposition21}
For each $\chi>\chi_1$, (\ref{ss}) has a solution $(u^-(r),v^-(r))$ explicitly given by (\ref{interring}) in which $u^-(r)$ is supported over a disk $B_0(r_1)$; moreover, we have the following asymptotics of the solutions in the large limit of $\chi$:

(i) $r_1$, the size of support of $u^-$, is monotone decreasing in $\chi$ and $r_1\rightarrow 0^+$ as $\chi\rightarrow \infty$;

(ii) $\max_{\bar B_0(R)}u^-(r)$ is monotone increasing in $\chi$ and $\max_{\bar B_0(R)}u^-(r)=\frac{M \chi}{2\pi j_{0,1}J_1(j_{0,1})}+O(1)$ for $\chi\gg1$;

(iii) $u^-(r)\rightarrow M\delta_0(r)$ and $v^-(r)\rightarrow \frac{M}{2\pi I_1(R)} T_0(r;R)$ pointwisely in $[0,R]$ as $\chi\rightarrow \infty$.
\end{proposition}
Figure \ref{innerasymptotic} presents an illustration on the statements in Proposition \ref{proposition21}.

\subsection{Outer Ring Solution}
Now we look for the other radially monotone solution of (\ref{ss}) such that the support of $u(r)$ is an interval $[R-r_2,R]$ for some $r_2$ to be determined.  This gives rise to a solution supported in an outer ring and we call it the outer ring solution and denote it by $(u^+,v^+)$.  Note that $(u^-(R-r),v^-(R-r))$, the reflection of $(u^-(r),v^-(r))$ about $\frac{R}{2}$, is no longer a solution of the original problem.  This is a strong contrast to the 1D problem \cite{BCR,CCWWZ} and makes the problem intricate as we shall see later.  The idea and procedure of constructing such a solution are the same above, hence we perform necessary calculations for future reference.

Similar as above, in this setting we find that $u^+\equiv 0$ for $r\in[0,R-r_2)$ and $u^+-\chi v^+$ is a constant for $r\in[R-r_2,R]$, where $r_2$ measures the size of support of $u$ and is to be determined.  Solving the $v$-equation in the fashion of (\ref{22}) gives
\begin{equation}\label{29}
v^+(r)=\left\{\begin{array}{ll}
C_3I_0(r),&r\in [0,R-r_2),\\
C_4 S_0(R-r;\omega, R) -\frac{\bar{C}}{\chi-1},&r\in[R-r_2,R],
\end{array}
\right.
\end{equation}
where $S_0$ is the following compound Bessel function
\begin{equation}\label{algeqn}
S_0(r; \omega, R):=Y_1(\omega R)  J_0(\omega(R-r))-J_1(\omega R)  Y_0(\omega(R-r)),
\end{equation}
with $Y_\alpha$ ($\alpha=0,1$) being Bessel function of the second kind.  We match the continuities of $v'(r)$ and $v''(r)$ at $r=R-r_2$ in (\ref{29})
\[
\left\{\begin{array}{ll}
-C_4\omega S_1(r_2;\omega,R)=C_3I_1(R-r_2),\\
-C_4\omega^2\Big(S_0(r_2;\omega,R)-\frac{S_1(r_2;\omega,R)}{\omega (R-r_2)}\Big)=C_3 \Big(I_0(R-r_2)-\frac{I_1(R-r_2)}{R-r_2}\Big),
\end{array}
\right.
\]
and find that $r_2$ must satisfy the following algebraic equation
\begin{equation}\label{210}
f_2(r_2;\omega,R):=\frac{\omega S_0(r_2;\omega,R)}{S_1(r_2;\omega,R)}-\frac{I_0(R-r_2)}{I_1(R-r_2)}=0,
\end{equation}
where $S_1(r; \omega, R):=Y_1(\omega R)  J_1(\omega(R-r))-J_1(\omega R)  Y_1(\omega(R-r))$.  Note that both $S_0$ and $S_1$ are the so--called cylinder functions proposed by Nielsen in \cite{Nielsen}, and both have infinitely many roots that are cross-oscillating with $...<s^{(k)}_0<s^{(k)}_1<s^{(k+1)}_0<s^{(k+1)}_1<...\rightarrow \infty$.  We first give the following result which establishes the existence and uniqueness of $r_2$ for (\ref{210}).
\begin{lemma}\label{lemma22}
Let $\chi_1=\left(\frac{j_{1,1}}{R}\right)^2+1$ be the same as in (\ref{bifvalue}).  Then for each $\chi>\chi_1$, the function $f_2(r;\omega,R)$ in (\ref{210}) has a unique root $r_2$ that lies in $(s_0^{(1)},s_1^{(1)})$, where $s_0^{(k)}$ and $s_1^{(k)}$ are the $k$-th positive root of $S_0(r;\omega,R)$ and $S_1(r;\omega,R)$, respectively.
\end{lemma}

\begin{proof}It suffices to show that $S_1(r;\omega,R)$ admits one root in $(0,R)$ if and only if $\chi>\chi_1$.  To prove the if part, we recall the following identity from Lommel \cite{Lommel} (page 106)
 \begin{equation}\label{211}
   Y_1(s) J_0(s)-J_1(s) Y_0(s)=-\frac{2}{\pi s},
   \end{equation}
which readily implies that $Y_1(s)$ and $J_1(s)$ have distinct roots.  Our discussion is divided into the following two cases. \emph{Case 1}: if $Y_1(\omega R)=0$, then we have that $\omega R\geq y_{1,2}$ since $\omega R>j_{1,1}>y_{1,1}$, where $y_{1,k}$ is the $k$-th positive root of $Y_1(s)$ with
\[y_{1,1}\approx 2.1971, y_{1,2}\approx 5.4296, y_{1,3}\approx 8.5960, y_{1,4}\approx 11.7491, y_{1,5}\approx 14.8974,... \]
to name the first few explicit values.  Write $\tau_1:=R-\frac{y_{1,1}}{\omega}\in(0,R)$, then one has that $\tau_1$ is the root of $S_1(r;\omega,R)$ since $S_1(\tau_1;\omega,R)=-J_1(\omega R) Y_1(y_{1,1})=0$; \emph{Case 2}: if $Y_1(\omega R)\neq 0$, then we rewrite $S_1(r;\omega,R)$ as
\[S_1(r;\omega,R)=Y_1(\omega R)Y_1(\omega(R-r)) \left(\frac{J_1(\omega(R-r))}{Y_1(\omega(R-r))}-\frac{J_1(\omega R)}{Y_1(\omega R)}\right).\]
Denote $y_3(r):=\frac{J_1(r)}{Y_1(r)}$.  A direct computation using (\ref{211}) gives that $y'_3(r)=-\frac{2}{\pi rY_1^2(r)}<0$ hence $y_3(r)$ is monotone decreasing in $\cup_{k\in\mathbb N} (y_{1,k},y_{1,k+1})$ with $y_{1,0}:=0$. Moreover, since $y_3(0)=y_3(j_{1,k})=0$ and $y_3(y^{\pm}_{1,k})=\pm\infty$, we have that whenever $\omega R>j_{1,1}$ there exists at least one $\tau_2\in(0,R)$ such that
\[\frac{J_1(\omega(R-\tau_2))}{Y_1(\omega(R-\tau_2))}=\frac{J_1(\omega R)}{Y_1(\omega R)}\]
and $Y_1(\omega(R-\tau_2))\neq 0$, therefore $S_1(\tau_2;\omega,R)=0$ as expected.  This verifies the claim in both cases.

Now we prove the only if part.  Suppose that there exists $s_1^{(1)}\in (0,R)$, then the ``only if" naturally holds when $Y_1(\omega R)=0$ since in this case $S_1(s_1^{(1)};\omega,R)=-J_1(\omega R)Y_1(\omega(R-s_1^{(1)}))=0$, which readily implies $Y_1(\omega(R-s_1^{(1)}))=0$ and $\omega R\geq y_{1,2}> j_{1,1}$.  When $Y_1(\omega R)\neq0$, with $Y_1(s)$ and $J_1(s)$ having distinctive roots due to (\ref{211}), we must have $Y_1(\omega(R-s_1^{(1)}))\neq0$ hence $s_1^{(1)}$ satisfies $y_3(\omega(R-s_1^{(1)}))=y_3(\omega R)$, which gives rise to $\omega R >j_{1,1}$ by the same arguments as above.

It remains to show that the existence of $s_1^{(1)}$ and $r_2$ in $(0,R)$ are equivalent.  Since $\partial_rS_0(r;\omega,R)=\omega S_1(r;\omega,R)$ and $S_0(0;\omega,R)<S_1(0;\omega,R)=0$, it is straightforward to see that $s_0^{(k)}<s_1^{(k)}$.  If $s_1^{(1)}\in (0,R)$, we claim that (\ref{210}) admits no positive root in $(0,s_0^{(1)})$.  Recall that
\[f_2(r;\omega,R):=\frac{\omega S_0(r;\omega,R)}{S_1(r;\omega,R)}-\frac{I_0(R-r)}{I_1(R-r)},\]
then $f_2(r;\omega,R)<0$ in $(0,s_0^{(1)})$.  For $r\in(s_0^{(1)},s_1^{(1)})$, straightforward calculations yield
\begin{equation}\label{213}
\partial_r  f_2(r;\omega,R) =\omega^2+1+f_2(r;\omega,R) \overbrace{\left(\frac{\omega S_0(r;\omega,R)}{S_1(r;\omega,R)}+\frac{I_0(R-r)}{I_1(R-r)} -\frac{1}{R-r}\right) }^{>0},
\end{equation}
where ``$>0$" holds since $sI_0(s)>I_1(s)$ for any $s>0$.  Note that $f_2(s_0^{(1)};\omega,R)=-\frac{I_0(R-s_0^{(1)})}{I_1(R-s_0^{(1)})}<0$ and $f_2((s_1^{(1)})^-;\omega,R)=+\infty$, therefore one finds that $f_2(r;\omega,R)$ admits at least one root in $(s_0^{(1)},s_1^{(1)})$, whereas the uniqueness can be verified by the same argument for Lemma \ref{lemma21}.  We further note that (\ref{210}) admits no root if $s_1^{(1)}\notin (0,R)$.  For $r\in(s_0^{(1)},R)$, we can readily see that $f_2(s_0^{(1)};\omega,R)<0$ and $f_2(R^-;\omega,R)=-\infty$, and since (\ref{213}) implies that a critical point $\xi_2$ exists only if $f_2(\xi_2;\omega,R)<0$, therefore $f_2(r;\omega,R)<0$ in $(s_0^{(1)},R)$ as a consequence.  The proof completes.
\end{proof}

With the support size $r_2$ obtained in (\ref{210}), we have that $u^+(r)\equiv 0$ for $r\in[0,R-r_2)$ and $u^+(r)=\mathcal{A}_2 \Big(S_0(R-r;\omega,R)-S_0(r_2;\omega,R)\Big)$ for $r\in[R-r_2,R]$, where $\mathcal A_2$, determined by the conservation of total cell population $2\pi\int_{R-r_2}^{R} u^+(r)rdr=M$, is explicitly given by
\begin{equation}\label{215}
\mathcal A_2=-\frac{M \omega}{\pi\big(2(R-r_2)S_1(r_2;\omega,R)+\omega r_2 (2R-r_2)S_0(r_2;\omega,R)\big)}.
\end{equation}
Note that $r_2 \in (s_{0}^{(1)},s_{1}^{(1)})$, then it follows that $\mathcal A_2<0$ is well-defined.

To find $v^+(r)$, we recall that $u^+=\chi v^++\bar{C}$ for $r\in[R-r_2,R]$ for some constant $\bar C$ to be determined (not the same as in the previous section) and
\[\mathcal{A}_2 \Big(S_0(R-r;\omega,R)-S_0(r_2;\omega,R)\Big)=\chi C_4 S_0(R-r;\omega,R)-\frac{\bar{C} }{\chi-1},\]
therefore $C_4=\frac{\mathcal A_2}{\chi}$ and $\bar{C} =\mathcal A_2(\chi-1)S_0(r_2;\omega,R)$, or $\bar{C} =-\frac{I_0(R-r_2)M \omega^2}{\pi\big(2(R-r_2)I_1(R-r_2)+ r_2 (2R-r_2)I_0(R-r_2)\big)}$ to be specific; moreover, the continuity of $v^+$ at $r=R-r_2$ implies $\mathcal A_2 (\frac{1}{\chi}-1) S_0(r_2;\omega,R)=\mathcal B_2 I_0(R-r_2)$ hence
\begin{equation}\label{217}
\mathcal B_2=\frac{M\omega^2}{\pi \chi \Big(2(R-r_2)I_1(R-r_2)+ r_2(2R-r_2)I_0(R-r_2)\Big)}.
\end{equation}

In summary, we obtain the following explicit formulas of the desired radially decreasing solution $(u^+(r),v^+(r))$
\begin{equation}\label{outerring}
u^+(r)=\left\{\begin{array}{ll}
\!\!\!0,&\!\!\!r\in[0,R-r_2),\\
\!\!\!\mathcal{A}_2 \Big(S_0(R-r)-S_0(r_2)\Big),&\!\!\!r\in[R-r_2,R],
\end{array}
\right.
\!\!\!v^+(r)=\left\{\begin{array}{ll}
\!\!\!\mathcal B_2 I_0(r),&\!\!\!r\in[0,R-r_2),\\
\!\!\!\mathcal A_2\Big(\frac{S_0(R-r)}{\chi}-S_0(r_2)\Big),&\!\!\!r\in[R-r_2,R],
\end{array}
\right.
\end{equation}
where $S_0$ is the compound Bessel function in (\ref{algeqn}), and $\mathcal A_2$ and $\mathcal B_2$ are given by (\ref{215}) and (\ref{217}).  The steady state given by (\ref{outerring}) is called the outer ring solution or outer ring for short.

\begin{remark}\label{remark22}
 Thanks to the fact that the zeros of $J_1(s)$ and $Y_1(s)$ are cross-oscillating such that $...<y_{1,k}<j_{1,k}<y_{1,k+1}<j_{1,k+1}<...$, $k\in\mathbb N^+$, one can show that if $\chi>\chi_k$, $S_1(r;\omega,R)$ admits at least $k$ positive roots in $(0,R)$, while there exists a unique root of (\ref{210}) in each interval $(s_0^{(k)},s_1^{(k)})$.  However, by the same arguments in Remark \ref{remark21}, one must restrict the unique root $r_2\in (s_0^{(1)},s_1^{(1)})$ as expected to guarantee the positivity of $u^+(r)$ in $[R-r_2,R]$.
\end{remark}

\subsubsection{Asymptotic behavior of outer ring solution in the limit of $\chi\rightarrow \infty$}
Now we study the asymptotic behaviors of the outer ring solutions in the limit of large chemotaxis rate.  Before proceeding further, let us introduce the following compound Bessel functions
\begin{align}
  &\mathcal V_0(r; \omega, R):=Y_0(\omega R)J_0(\omega(R-r))-J_0(\omega R)Y_0(\omega(R-r)),\label{V0} \\
 & \mathcal V_1(r; \omega, R):=Y_0(\omega R)J_1(\omega(R-r))-J_0(\omega R)Y_1(\omega(R-r)).\label{V1}
\end{align}
Then one finds by straightforward calculations
\begin{align}
   &\frac{\partial S_0(r;\omega,R)}{\partial \omega}=R\mathcal V_0(r;\omega,R)-\frac{S_0(r;\omega,R)}{\omega}-(R-r)S_1(r;\omega,R), \notag\\
   & \frac{\partial S_1(r;\omega,R)}{\partial \omega}=R\mathcal V_1(r;\omega,R)-\frac{2S_1(r;\omega,R)}{\omega}+(R-r)S_0(r;\omega,R),\notag
\end{align}
and infer from (\ref{211}) that
\begin{equation}\label{219}
  S_0(r;\omega,R) \mathcal V_1(r;\omega,R)-\mathcal V_0(r;\omega,R)  S_1(r;\omega,R)=-\frac{4}{\pi^2\omega^2R(R-r)}.
\end{equation}

Similar as above, we first claim that $r_2$, the size of support of $u^+(r)$, shrinks as $\chi$ increases.  According to our previous discussions, we know from (\ref{210}) that $r_2\in(s_0^{(1)},s_1^{(1)})$ is uniquely determined for each $\chi>\chi_1$ and it depends on $\chi$ continuously.  We claim that $\frac{\partial r_2}{\partial \omega}<0$.  Indeed, since $f_2(r_2;\omega,R)=0$, we have from (\ref{219})
\begin{align}
  \frac{\partial r_2}{\partial \omega} =&-\frac{\partial_{\omega}f_2(r_2;\omega,R)}{\partial_rf_2(r_2;\omega,R)}= -\frac{\partial_{\omega}f_2(r_2;\omega,R)}{\omega^2+1} \notag\\
   =&-\frac{S_0(r_2;\omega,R)S_1(r_2;\omega,R)+\omega S_1(r_2;\omega,R)  \frac{\partial S_0(r_2;\omega,R)}{\partial \omega}-\omega S_0(r_2;\omega,R)  \frac{\partial S_1(r_2;\omega,R)}{\partial \omega}}{(\omega^2+1)S^2_1(r_2;\omega,R)} \notag \\
   =&-\frac{\frac{4}{\pi^2}-\omega^2(R-r_2)^2\left(S_0^2(r_2;\omega,R)+S_1^2(r_2;\omega,R)\right)+2\omega(R-r_2)S_0(r_2;\omega,R)S_1(r_2;\omega,R)}{\omega(\omega^2+1)(R-r_2)S_1^2(r_2;\omega,R)}\notag\\
   =&:-\frac{y_4(r_2;\omega,R)}{\omega(\omega^2+1)(R-r_2)S_1^2(r_2;\omega,R)}.\notag
\end{align}
Some straightforward calculations give $y'_4(r;\omega,R)=2\omega^2(R-r)S_1^2(r;\omega,R)>0$, therefore $y_4(r_2;\omega,R)>y_4(0;\omega,R)=0$ and $\frac{\partial r_2}{\partial \omega}<0$ as claimed.  This fact implies that the size of support shrinks as chemotaxis rate $\chi$ increases.

We next show that $r_2\rightarrow 0^+$ as $\chi\rightarrow \infty$.  To prove this, we need some refined estimates for $s_0^{(1)}$ and $s_1^{(1)}$.  For this purpose let us introduce for $s \in(0,\omega R)$
\begin{equation}\label{220}
  \mathcal {C}_n(s):=J_n(s) \cos\alpha-Y_n(s)\sin\alpha, n=1,2,
\end{equation}
with $\alpha:=\arctan\frac{J_1(\omega R)}{Y_1(\omega R)}\in[0,\pi)$.  Then we find that $\omega(R-s_0^{(k)})$ and $\omega(R-s_1^{(k)})$ are the roots of $\mathcal {C}_0(s)$ and $\mathcal {C}_1(s)$, respectively.; on the other hand, according to \cite{Schafheitlin} and \cite{Watson} (Chap. 15.33), all the positive roots of (\ref{220}) must lie in one of the intervals
\[\Big(m\pi+\frac{\pi}{4}(3-3n)-\alpha,\; m\pi+\frac{\pi}{4}(4-3n)-\alpha\Big), m=0,1,2,\cdots\]
Since $\omega R$ is a root of $\mathcal {C}_1(s)$, there exists some $m_1\in\mathbb{N}^+$ such that
\[\omega R \in(m_1\pi-\alpha,\; m_1\pi+\frac{\pi}{4}-\alpha),\]
and the adjacent root of $\mathcal {C}_1(s)$ lies in $(\omega R-\frac{5\pi}{4},\;\omega R-\frac{3\pi}{4})$, i.e., $s_1^{(1)}\in (\frac{3\pi}{4\omega},\frac{5\pi}{4\omega})$. In light of the relationship between the roots of $\mathcal {C}_n(s)$, it follows that $s_0^{(1)}\in (0,\frac{\pi}{2\omega})$.   Therefore we conclude that $r_2\in(s_0^{(1)}, s_1^{(1)})\subset(0, \frac{5\pi}{4\omega})$, hence $r_2\rightarrow 0^+$ as $\chi\rightarrow \infty$.  Indeed, we are able to provide a finer estimate $r_2\in(\frac{\pi}{2\omega},\frac{5\pi}{4\omega})$ by using the asymptotic expansions of Bessel functions and the uniform bounds of the ratios of modified Bessel functions.  The proof is given in the appendix.

Similar as before, we now show that the maximum of $u^+(r)$ is strictly increasing in $\chi$ with $\frac{\partial \Vert u^+\Vert _{L^{\infty}}}{\partial \omega}>0$ and $\Vert u^+ \Vert_{L^{\infty}}= O(\omega)$ as $\chi\rightarrow \infty$.  Combining (\ref{outerring}) with (\ref{215}) gives
\[\Vert u^+ \Vert_{L^{\infty}}=\frac{M}{\pi^2}\frac{(2+\pi\omega R S_0(r_2;\omega,R)}{2(R-r_2)S_1(r_2;\omega,R)+\omega r_2(2R-r_2)S_0(r_2;\omega,R)}.\]
Direct calculations give
\begin{align}
\frac{\pi^2}{M}  \frac{\partial \Vert u^+ \Vert_{L^{\infty}}}{\partial \omega} =&\frac{2\omega^3 S_1  \left(\pi R(R-r_2)S_1- r_2 (2R-r_2) \right)  \frac{\partial r_2}{\partial \omega}+2 \pi R y_4(r_2;\omega,R)}{\omega \left(2(R-r_2)S_1+\omega r_2 (2R-r_2) S_0 \right)^2 }\notag \\
   &- \frac{  2 \left(2 (R-r_2)   (R \mathcal V_1 + (R-r_2) S_0 -\frac{2S_1}{\omega}) + \omega r_2 (2R-r_2)   (R\mathcal V_0-(R-r_2)S_1)  \right)}{\left(2(R-r_2)S_1+\omega r_2(2R-r_2)S_0 \right)^2},\notag
\end{align}
where variables $(r_2;\omega,R)$ are skipped without confusing the reader.  Thanks to (\ref{71}) and (\ref{72}), one finds that $S_1(r_2;\omega,R)<\frac{2}{\pi \omega \sqrt{R(R-r_2)}}$ for any $\chi>\chi_1$, and this implies $\left(\pi R(R-r_2)S_1- r_2 (2R-r_2) \right)<0$ because $\omega r_2>\frac{\pi}{2}$.  Denote
\[y_5(r;\omega,R):=\frac{R\mathcal V_1(r;\omega,R)+(R-r)S_0(r;\omega,R)}{S_1(r;\omega,R)}, y_6(r;\omega,R):=R\mathcal V_0(r;\omega,R)-(R-r)S_1(r;\omega,R).\]
with $\mathcal V_0$ and $\mathcal V_1$ given by (\ref{V0}) and (\ref{V1}), respectively.  Then one easily finds that $y'_5(r;\omega,R)=-\frac{y_4(r;\omega,R)}{\omega(R-r)S^2_1(r;\omega,R)}<0$ and $y'_6(r;\omega,R)=\omega S_1(r;\omega,R) y_5(r;\omega,R)$, hence $y_5(r;\omega,R)<y_5(0;\omega,R)=0$ and $y_6(r;\omega,R)<y_6(0;\omega,R)=0$ in $(0,s_1^{(1)})$.  Combining the facts that $\frac{\partial r_2}{\partial \omega}<0$ and $y_4(r_2;\omega,R)>0$, we find $\frac{\partial \Vert u \Vert_{L^{\infty}}}{\partial \omega}>0$ as expected.

One can prove that $r_2\rightarrow (\frac{\pi}{2\omega})^+$ in the limit of $\chi\rightarrow \infty$ (see the Appendix).  This fact, together with (\ref{74}) and (\ref{75}), implies that $\Vert u^+ \Vert_{L^{\infty}}=\frac{M}{\pi^2(R-\frac{\pi}{2\omega}) S_1(\frac{\pi}{2\omega};\omega,R)}=\frac{M\omega}{2\pi}+O(1)$; moreover, $u^+(r)\rightarrow u_\infty=M\delta_R(r)$ and $v^+(r)\rightarrow v_\infty=\frac{M}{2\pi R I_1(R)} I_0(r)$ pointwisely.

\begin{figure}[h!]\vspace{-5mm}
    \centering
    \includegraphics[width=1\textwidth]{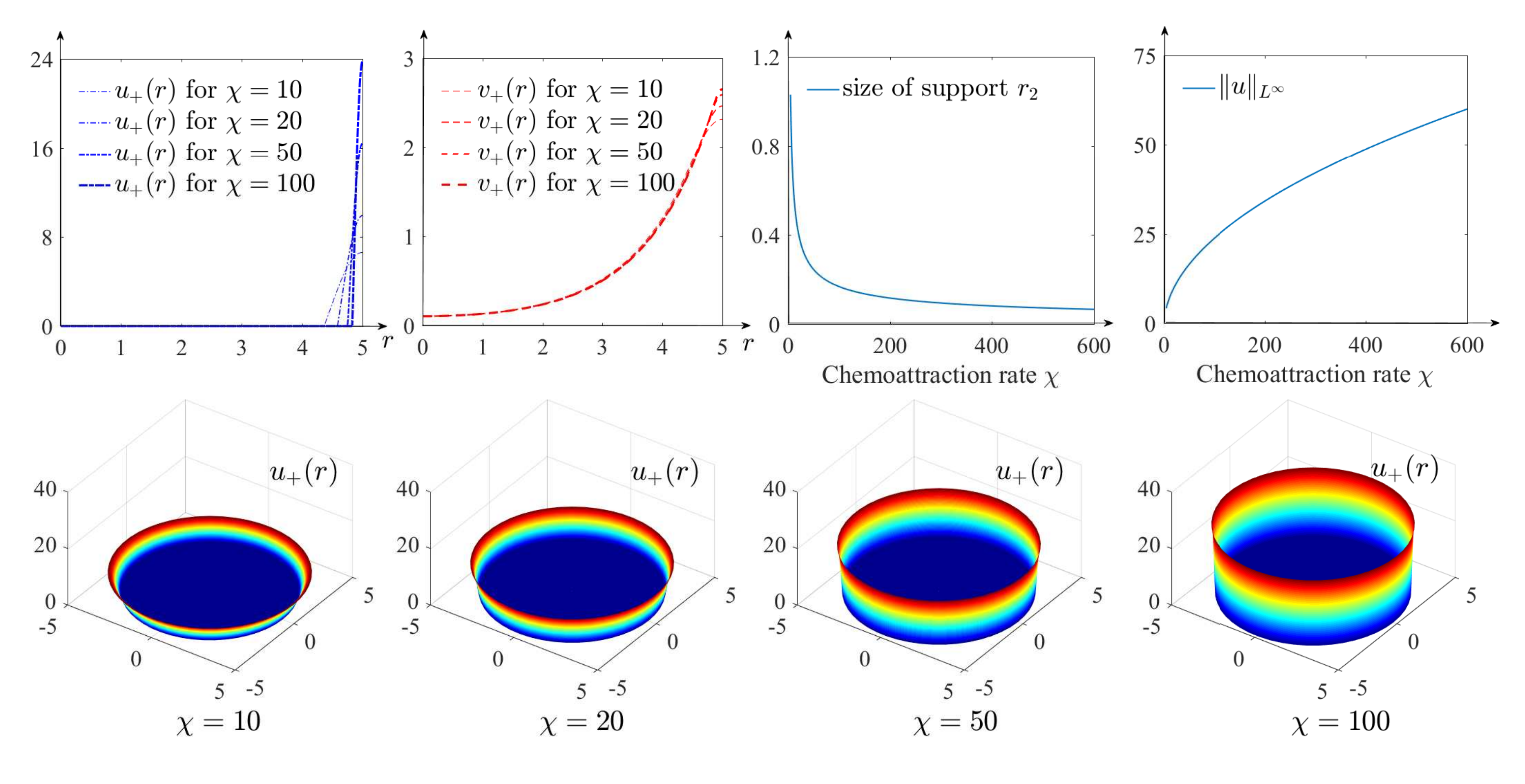}\vspace{-5mm}
    \caption{\textbf{Top}: the outer ring solution $(u^+,v^+)$ with $M=25\pi$ and $R=5$ for $\chi=$10, 20, 50 and 100.  One observes that $u^+(r)\rightarrow 25\pi\delta_R(r)$ and $v^+(r)\rightarrow 0.1027 I_0(r)$ as $\chi\rightarrow \infty$; moreover, the size $r_2$ of support of $u^+(r)$ shrinks as $O(\frac{1}{\omega})$ or $O(\frac{1}{\sqrt{\chi}})$ and the maximum of $u^+(r)$ expands as $O(\omega)$ or $O(\sqrt{\chi})$.  We would like to mention that one always that $\Vert u^+\Vert_{L^\infty}<\Vert u^-\Vert_{L^\infty}$ for each $\chi>\chi_1$.  This is presented in the bifurcation diagram in Figure \ref{branch1}.  \textbf{Bottom}:  Formation and development of outer ring configuration in $u$ as $\chi$ expands from 10 to 100.}\label{outerasymptotic}
\end{figure}

Our results above are summarized in the following proposition:
\begin{proposition}\label{proposition22}
For each $\chi>\chi_1$, (\ref{ss}) has a solution $(u^+(r),v^+(r))$ explicitly given by (\ref{outerring}) in which $u^+(r)$ is supported over the annulus $B_0(R)\backslash B_0(r_2)$.  Moreover, we have the following asymptotics of the solutions in the large limit of $\chi$:

(i) $r_2$, the size of support of $u^+$, is monotone decreasing in $\chi$ and $r_2\rightarrow 0^+$ as $\chi\rightarrow \infty$;

(ii) $\max_{\bar B_0(R)}u^+(r)$ is monotone increasing in $\chi$ and $\max_{\bar B_0(R)}u^+(r)=\frac{M\sqrt{\chi}}{2\pi}+O(1)$ for $\chi\gg 1$;

(iii) $u^+(r)\rightarrow M\delta_R(r)$ and $v^+(r)\rightarrow \frac{M}{2\pi R I_1(R)} I_0(r)$ pointwisely in $[0,R]$ as $\chi\rightarrow \infty$;
\end{proposition}
See Figure \ref{outerasymptotic} for illustration of Proposition \ref{proposition22}.

\subsection{Monotone Solutions in an Annulus: An Auxiliary and Supporting Problem}
Our previous results imply that for each $\chi>\chi_1$, the stationary problem (\ref{ss}) has a pair of solutions $(u^\pm,v^\pm)$ explicitly given by (\ref{interring}) and (\ref{outerring}), both of which have $u$ compactly supported and radially monotone over its support; moreover, if $\chi\in(\chi_1,\chi_2)$, the solution is unique in the sense that all nonconstant solutions of (\ref{ss}) must be one of the pair; furthermore, when $\chi=\chi_2$, there also exist non-monotone solutions of (\ref{ss}) (uniquely) given by the one-parameter family of functions (\ref{bifurcation}), hence we assume $\chi>\chi_2$ to look for radially non-monotone solutions.  Similar as above, it is easy to see that whenever $\chi>\chi_2,\neq \chi_k$, $u(r)$ must be compactly supported, though the support is not necessarily connected.

In this subsection, we prepare ourselves to study non-monotone solutions by solving (\ref{ss}) for radially monotone solutions over (predetermined) subsets of $B_0(R)$, then we concatenate them (with the inner ring and/or outer ring solutions when applicable) by matching the continuities of $v(r)$ at the interface.  To be precise, let us consider a slightly general case than above such that $u(r)$ is supported over an annulus $(a,b)$ for arbitrarily given $b>a\geq0$.

We first solve (\ref{ss}) in $(a,b)$ for $(\mathbb U_3,\mathbb V_3)$ such that $\mathbb U_3$ is supported in $(a,a+r_3)$, for some $r_3$ to be determined.  Solving (\ref{22}) with $(0,R)$ replaced by $(a,b)$ gives the generic decreasing mode
\begin{equation}\label{227}
\left\{ \begin{array}{ll}
\mathbb U_3(r)=&\left\{ \begin{array}{ll}
\mathcal A_3 \Big(S_0(a-r;\omega,a)-S_0(-r_3;\omega,a)\Big),&r\in(a,a+r_3),\\
0,&r\in(a+r_3,b),
\end{array}
\right.\\
\\
\mathbb V_3(r)=& \left\{\begin{array}{ll}
\mathcal A_3\Big(\frac{1}{\chi}S_0(a-r;\omega,a)-S_0(-r_3;\omega,a)\Big),&r\in(a,a+r_3),\\
\mathcal B_3 T_0(r;b),&r\in(a+r_3,b),
\end{array}
\right.
\end{array}
\right.
\end{equation}
where $S_0$ and $T_0$ are given by (\ref{algeqn}) and (\ref{T0}), and $r_3\in(0,b-a)$ is to be determined by the continuity of $\mathbb V_3(r)$ at $r=a+r_3$.  That being said, $r_3$ is the first positive root of the following algebraic equation
\begin{equation}\label{228}
f_3(r_3;\omega,a,b):=\frac{\omega S_0(-r_3;\omega,a)}{S_1(-r_3;\omega,a)}-\frac{T_0(a+r_3;b)}{T_1(a+r_3;b)}=0.
\end{equation}
With $r_3$ determined by (\ref{228}), $\mathcal A_3$ and $\mathcal B_3$ can be evaluated by the conservation of mass
\begin{align}
   & \mathcal A_3=\frac{M \omega }{\pi\big(2(a+r_3)S_1(-r_3;\omega,a)-\omega r_3 (2a+r_3)S_0(-r_3;\omega,a)\big)}, \notag\\
   & \mathcal B_3=-\frac{M\omega^2}{\pi\chi \big(2(a+r_3)T_1(a+r_3;b)- r_3 (2a+r_3)T_0(a+r_3;b)\big)}. \notag
\end{align}

Similarly, the generic increasing mode that extends the outer ring in $(0,R)$ to $(a,b)$ is
\begin{equation}\label{229}
\left\{ \begin{array}{ll}
\mathbb U_4(r)=& \left\{ \begin{array}{ll}
0,&r\in(a,b-r_4),\\
\mathcal A_4 \Big(S_0(b-r;\omega,b)-S_0(r_4;\omega,b)\Big),&r\in(b-r_4,b),
\end{array}
\right.\\
\\
\mathbb V_4(r)=& \left\{ \begin{array}{ll}
\mathcal B_4 T_0(r;a),&r\in(a,b-r_4),\\
\mathcal A_4\Big(\frac{1}{\chi}S_0(b-r;\omega,b)-S_0(r_4;\omega,b)\Big),&r\in(b-r_4,b),
\end{array}
\right.\\
\end{array}
\right.
\end{equation}
where $r_4\in(0,b-a)$ is the first positive root of the following algebraic equation
\begin{equation}\label{231}
f_4(r_4;\omega,a,b):= \frac{\omega S_0(r_4;\omega,b)}{S_1(r_4;\omega,b)}-\frac{T_0(b-r_4;a)}{T_1(b-r_4;a)}=0;
\end{equation}
moreover, by the conservation of total population
\begin{align}
   & \mathcal A_4=-\frac{ M\omega}{\pi\big(2(b-r_4)S_1(r_4;\omega,b)+\omega r_4 (2b-r_4)S_0(r_4;\omega,b)\big)}, \notag\\
   & \mathcal B_4=\frac{M\omega^2}{\pi\chi \big(2(b-r_4)T_1(b-r_4;a)+ r_4 (2b-r_4)T_0(b-r_4;a)\big)}. \notag
\end{align}

Therefore, we are left to find condition(s) on $\chi$ that guarantee the existence of $r_3$ in (\ref{228}) and $r_4$ in (\ref{231}), respectively.  Our main results can be summarized into the following lemma.
 \begin{lemma} \label{lemma23}
Let $0\leq a<b$ be two arbitrary constants.  Denote $\chi_{a,b}=\omega_{a,b}^2+1$, where $\omega_{a,b}$ is given by
\begin{equation}\label{omegaab}
\omega_{a,b}:=\inf_{\omega>\frac{j_{1,1}}{b}}\Bigg\{\omega\in\mathbb R^+\Big| \frac{J_1(\omega a)}{Y_1(\omega a)}=\frac{J_1(\omega b)}{Y_1(\omega b)} \Bigg\}<\infty.
\end{equation}
Then $\omega_{a,b}\in[\frac{j_{1,1}}{b},\frac{j_{1,1}}{b-a}]$ with $\omega_{0,b}=\frac{j_{1,1}}{b}$; moreover, the following statements hold:
\begin{enumerate}[(i)]
     \item  for each $\chi>\chi_{a,b}$, (\ref{228}) has a unique root $r_3\in (0,s_{1a}^{(1)})\subset(0,b-a)$, where $s_{1a}^{(1)}$ is the first positive root of $S_1(-r;\omega,a)$ in $(0,b-a)$, and (\ref{228}) admits no root in $(0,b-a)$ if $\chi<\chi_{a,b}$;
    \item  for each $\chi>\chi_{a,b}$, (\ref{231}) has a unique root $r_4\in (0,s_{1b}^{(1)})\subset(0,b-a)$, where $s_{1b}^{(1)}$ is the first root of $S_1(r;\omega,b)$ in $(0,b-a)$, and (\ref{231}) admits no root in $(0,b-a)$ if $\chi<\chi_{a,b}$;
    \item both $r_3$ and $r_4$ continuously depend on and monotonically decrease in $\chi$, and $r_i\rightarrow 0^+$ as $\chi\rightarrow\infty$.
   \end{enumerate}
 \end{lemma}

\begin{proof}
The proof is the same as that of Lemma \ref{lemma22}, and we will only do it for part \emph{(i)}, while part \emph{(ii)} can be verified similarly.  To show that $f_3(r;\omega,a,b)$ has a solution $r_3\in(0,s_{1a}^{(1)})$ for each $\chi>\chi_{a,b}$, we claim it is sufficient to show that $S_1(-r;\omega,a)$ has at least one solution in $(0,b-a)$.  Rewrite
\[S_1(-r;\omega,a)=Y_1(\omega a)Y_1(\omega(a+r))\Bigg(\frac{J_1(\omega(a+r))}{Y_1(\omega(a+r))}-\frac{J_1(\omega a)}{Y_1(\omega a)}\Bigg), r\in(0,b-a).\]
Thanks to the monotonicity of $\frac{J_1(r)}{Y_1(r)}$, for each $\omega>\omega_{a,b}$ there exists at least one number $s_{1a}^{(1)}$ in $(0,b-a)$ such that $S_1(-s_{1a}^{(1)};\omega,a)=0$.  We are now left to show that the existence of such $r_3$ and $s_{1a}^{(1)}$ in $(0,b-a)$ are equivalent.  To this end, we first find that $\partial_r S_0(-r;\omega,a)=-\omega S_1(-r;\omega,a)$ and $S_0(0;\omega,a)<S_1(0;\omega,a)=0$, therefore $s_{0a}^{(1)}<s_{1a}^{(1)}$, where $s_{0a}^{(1)}$ is the first root of $S_0(-r;\omega,a)$ in $(0,b-a)$.  Note that if $s_{1a}^{(1)}\in(0,b-a)$, then $f_3(r;\omega,a,b)>0$ in $(0,s_{0a}^{(1)})$, which indicates that it has no root in $(0,s_{0a}^{(1)})$.  Now, for each $r\in(s_{0a}^{(1)},s_{1a}^{(1)})$, straightforward calculations yield
\[\partial_r f_3(r;\omega,a,b)=-(\omega^2+1)+f_3(r;\omega,a,b)\overbrace{\Big(-\frac{\omega S_0(-r;\omega,a)}{S_1(-r;\omega,a)}-\frac{T_0(r+a;b)}{T_1(r+a;b)}+\frac{1}{a+r}\Big)}^{>0 \text{~in~}(s_{0a}^{(1)},s_{1a}^{(1)})},\]
where ``$>0$" holds since $T_0(r+a;b)>0$ and $T_1(r+a;b)<0$ for $r\in(0,b-a)$.  Note that $f_3(s_{0a}^{(1)};\omega,a,b)>0$ and $f_3((s_{1a}^{(1)})^-;\omega,a,b)=-\infty$, therefore one finds that $f_3(r;\omega,a,b)$ admits at least one root in $(s_{0a}^{(1)},s_{1a}^{(1)})$ as claimed.  By the same arguments for Lemma \ref{lemma22}, one can show this root is unique.

To show part (iii), we first verify that $r_3$ decreases in $\chi$.  To see this, we differentiate $f_3(r_3;\omega,a,b)=0$ with respect to $\omega$ and collect
\begin{align*}
 \frac{\partial r_3}{\partial \omega}=&-\frac{\partial_\omega f_3(r_3;\omega,a,b)}{\partial_r f_3(r_3;\omega,a,b)}=\frac{\partial_\omega f_3(r_3;\omega,a,b)}{\omega^2+1}\\
    =&\frac{\frac{4}{\pi^2}-\omega^2\Big(S_0^2(-r_3;\omega,a)+S_1^2(-r_3;\omega,a)\Big)+2\omega (a+r_4)S_0(-r_3;\omega,a)S_1(-r_3;\omega,a)  }{\omega(\omega^2+1)(a+r_3)S_1^2(-r_3;\omega,a)}  \\
   =:&\frac{y_7(r_3;\omega,a)}{\omega(\omega^2+1)(a+r_3)S_1^2(-r_3;\omega,a)}.
\end{align*}
We claim that the numerator $y_7(r_3;\omega,a)<0$.  Indeed, we find by straightforward calculations that $\partial_\omega y_7(r_3;\omega,a)=-2\omega^2(r_3+a)S_1^2(-r_3;\omega,a)<0$ and $y_7(0;\omega,a)=0$, this readily verifies the claim hence $\frac{\partial r_3}{\partial \omega}<0$ as expected.  Similarly one can show the monotonicity of $r_4$ in $\omega$.  Moreover, $r_3\rightarrow 0^+$ follows from the fact that $s_{1a}^{(1)}\rightarrow 0^+$ as $\chi\rightarrow \infty$.  The lemma is proved.
\end{proof}

We would like to remark that, (\ref{228}) have other roots in $(0,b-a)$ if $\chi$ is large enough.  However, these roots can not be chosen as the size of support according to the same arguments as in Remark \ref{remark21}.  When $a=0$, we have that $\omega _{0,b}=\frac{j_{1,1}}{b}$ and $\omega_{a,b}\rightarrow \frac{j_{1,1}}{b}$ as $a\rightarrow 0^+$, then Lemma \ref{lemma23} readily recovers the statements about the root $r_1$ of (\ref{23}); moreover, $\omega_{a,b}\rightarrow +\infty$ as $a\rightarrow b^-$.  We can say more about its monotonicity as follows which will be needed for our coming analysis.
\begin{lemma}\label{lemma24}
For each given $b>0$, $\omega_{a,b}$ defined by (\ref{omegaab}) is strictly increasing in $a$, i.e., $\omega_{a_1,b}<\omega_{a_2,b}$ if $0\leq a_1<a_2<b$.
\end{lemma}
\begin{proof}
Let $y_{1,k}$ be the $k$-th positive root of $Y_1$ as in the proof of Lemma \ref{lemma22}.  We divide our discussion into the following two cases.  Case (i): $Y_1(\omega_{a,b} b)=0$.  Then we must have that $\omega_{a,b} a=y_{1,k_0}$ and $\omega_{a,b} b=y_{1,k_0+1}$ for some $k_0\in\mathbb N^+$.  Therefore $\omega_{a,b}=\frac{y_{1,k_0+1}-y_{1,k_0}}{b-a}$, which readily implies the desired claim.  Case (ii): $Y_1(\omega_{a,b} b)\neq 0$, then we know that $Y_1(\omega_{a,b} a)\neq0$.  Let us differentiate the identity $\frac{J_1(\omega_{a,b} a)}{Y_1(\omega_{a,b} a)}=\frac{J_1(\omega_{a,b} b)}{Y_1(\omega_{a,b} b)}$ with respect to $a$, then one finds that
\[\frac{\partial \omega_{a,b}}{\partial a}=\frac{\frac{2\omega}{\pi \omega aY_1^2(\omega a)}}{-\frac{2a}{\pi \omega aY_1^2(\omega a)}+\frac{2b}{\pi \omega bY_1^2(\omega b)}}
=\frac{\omega Y_1^2(\omega b)}{a\big(Y_1^2(\omega a)-Y_1^2(\omega b)\big)}\Bigg|_{\omega=\omega_{a,b}}.\]
One knows that (e.g., Sec. 13. 71 in \cite{Watson}) the compound Bessels function $(J_1^2+Y_1^2)(r)$ is always decreasing hence $J_1^2(\omega_{a,b}b)+Y_1^2(\omega_{a,b}b)<J_1^2(\omega_{a,b}a)+Y_1^2(\omega_{a,b}a)$.  This inequality, together with the identity $\frac{J_1(\omega_{a,b} a)}{Y_1(\omega_{a,b} a)}=\frac{J_1(\omega_{a,b} b)}{Y_1(\omega_{a,b} b)}$, readily implies that $Y_1^2(\omega_{a,b} a)-Y_1^2(\omega_{a,b} b)>0$, hence $\frac{\partial \omega_{a,b}}{\partial a}>0$ in Case (ii).  The lemma is verified in both cases.
\end{proof}

When $\chi=\chi_{a,b}$, we see that $\frac{J_1(\omega a)}{Y_1(\omega a)}=\frac{J_1(\omega b)}{Y_1(\omega b)}$ and $Y_1(s)$ admits only one zero in $(\omega a, \omega b)$, therefore $u$ is strictly positive in the annulus $B_0(b)\backslash B_0(a)$.  By the same arguments for (\ref{bifurcation}), one sees that $\chi_{a,b}-1$ is the principle eigenvalue of $-\Delta_r$ in the annulus and $v$ is an eigenfunction hence must be a multiplier of $S_0(b-r;\omega,b)$.  Then similar as in (\ref{bifurcation}) one sees that (\ref{ss}) has a one-parameter families of solutions given by
\begin{equation}\label{uv5}
(\mathbb U_\varepsilon(r), \mathbb V_\varepsilon(r))=(\bar u_{ab},\bar v_{ab})+\varepsilon(\chi_{a,b},1)S_0(b-r;\omega,b), r\in (a,b); \frac{-\bar u_{ab}/\chi_{a,b}}{S_0(b-a;\omega,b)}\leq \varepsilon\leq \frac{-\bar u_{ab}/\chi_{a,b}}{S_0(0;\omega,b)},
\end{equation}
where $\bar u_{a,b}=\bar v_{a,b}=\frac{M}{\pi(b^2-a^2)}$.  We see that (\ref{uv5}) extends (\ref{bifurcation}) from $(0,R)$ to $(a,b)$ for any $0\leq a<b$.

It seems necessary to mention that, though any constant pair $(\bar u_i,\bar v_i)$ solves the equation in (\ref{ss}) over interval $(a_i,a_{i+1})$, such profile can not be concatenated into the solution that matches the mass constraint and $v$ can not admit a plateau profile.  We have the following statement.
\begin{corollary}
Let $(u,v)$ be any nonconstant solutions of (\ref{ss}).  Then the $v$ has at most finite many critical points in $(0,R)$, i.e., there does not exist an interval $(a,b)$ such that $u=v\equiv$ some constant in $(a,b)$.
\end{corollary}\label{corollary21}
\begin{proof}
Argue by contradiction.  Suppose there exists $(a,b)$ such that $u\equiv v\equiv c$  in $[a,b]$.  Then we must have that $a>0$ or $b<R$ since otherwise $u\equiv v\equiv \bar u$ if $(a,b)=(0,R)$ and we are done.  For $a>0$, there exists $a_0\in[0,a)$ such that $v'(a_0)=0$ and either $v$ is monotone decreasing or monotone increasing in $(a_0,a)$.  In the former case, we have from (\ref{227}) that $(u,v)=(\mathbb U_3,\mathbb V_3)$ with $(a,b)=(a_0,a)$, therefore $u(a)=\lim_{r\rightarrow a^-} \mathbb U_3(a)=\mathcal A_3(S_0(a_0-a;\omega,a_0)-S_0(-r_3;\omega,a_0))$ and $v(a)=\lim_{r\rightarrow a^-} \mathbb V_3(a)=\mathcal A_3(\frac{1}{\chi}S_0(a_0-a;\omega,a_0)-S_0(-r_3;\omega,a_0))$, which readily implies that $u(a)<v(a)$, however, this is a contradiction to the assumption that $u(a)=v(a)=c$; in the latter case, we use (\ref{229}) and can find that $u(a)>v(a)$ which is also a contradiction to our assumption.  Therefore, $a>0$ is impossible.  Similarly we can show that $b<R$ is also impossible unless $(a,b)=(0,R)$.
\end{proof}

It might seem redundant, but for the reader's reference we present the
\begin{proof} [Proof\nopunct]  \emph{of Theorem} \ref{theorem11} and \emph{Theorem} \ref{theorem12}.

For Theorem \ref{theorem11}, \emph{(i)} and \emph{(ii)} follow from \cite{CCWWZ} and Lemma \ref{lemma11}; \emph{(iii)} and \emph{(iv)} follow from Proposition \ref{proposition21} and  Proposition \ref{proposition22}, respectively.  For Theorem \ref{theorem12}, \emph{(i)} follows from Lemma \ref{lemma32} and Proposition \ref{proposition31}; \emph{(ii-1)} and \emph{(ii-2)} follow from Proposition \ref{proposition32} and Proposition \ref{proposition33}, respectively.  Finally, \emph{(iii)} holds in light of Lemma \ref{lemma24} (with $a_1=0$ and $a_2=\frac{j_{1,2}}{\omega}$) since $\chi<\chi_3$.
\end{proof}

\section{Radially Non-Monotone Solutions}\label{section3}
In Section 2, for each $\chi\in(\chi_1,\infty)$ we have established the inner ring solution (\ref{interring}) and the outer ring solution (\ref{outerring}), both of which converge to Dirac delta function in the limit of large chemotaxis rate; moreover, we show that if $\chi\in(\chi_1,\chi_2)$, then any nonconstant radial solution of (\ref{ss}) must be either the inner ring or the outer ring solution given above.  In this section, we proceed to look for non-monotone solutions of (\ref{ss}) and for this purpose we shall assume $\chi\geq \chi_2$ from now on.  Recall that for each $\chi=\chi_k$, $k\geq2$, (\ref{ss}) has solutions explicitly given by the one--parameter family bifurcation solutions in (\ref{bifurcation}) such that both $u(r)$ and $v(r)$ are positive in $(0,R)$.  We are interested in solutions that extend the global continuum of each bifurcation branch, which consists of solutions such that $u$ is compactly supported.  Our results in this section can be summarized as follows: for $\chi\in(\chi_2,\infty)$, there exists two classes of non-monotone solutions of (\ref{ss}), explicitly given by (\ref{mexicanhat}) and (\ref{volcano1}); moreover, if $\chi\in(\chi_k,\chi_{k+1})$, $k\geq2$, the sign of $v'(r)$ changes at most $(k-1)$ times in $(0,R)$.

\subsection{Radially Non-Monotone Solution: The Mexican-Hat Solutions}
In order to construct non-monotone solutions, we first choose some $R_0\in(0,R)$ such that $f(r_1;\omega,R_0)$ in (\ref{23}) and $f(r_4;\omega,a,b)$ in (\ref{231}) are solvable with $a=R_0$ and $b=R$.  Then $(u,v)$ takes the form (\ref{227}) with $(a,b)=(0,R_0)$ and (\ref{229}) with $(a,b)=(R_0,R)$, and $v(r)$ achieves its minimum at $r=R_0$.  These radial profiles give rise to an interior in the center and an outer ring on the boundary, which we refer to as the Mexican-hat solutions.  See the left column of Figure \ref{volcanomexican} for illustration of its configuration.
\begin{figure}[h!]\vspace{-3mm}\hspace{15mm}
\begin{minipage}{0.4\columnwidth}
\includegraphics[width=1\columnwidth,height=35mm]{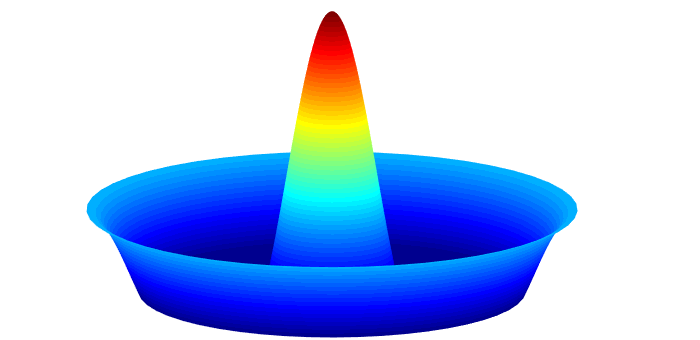}\vspace{-1mm}
\caption*{Mexican-hat}
\end{minipage}\hspace{-5mm}
\begin{minipage}{0.4\columnwidth}
\includegraphics[width=1\columnwidth,height=35mm]{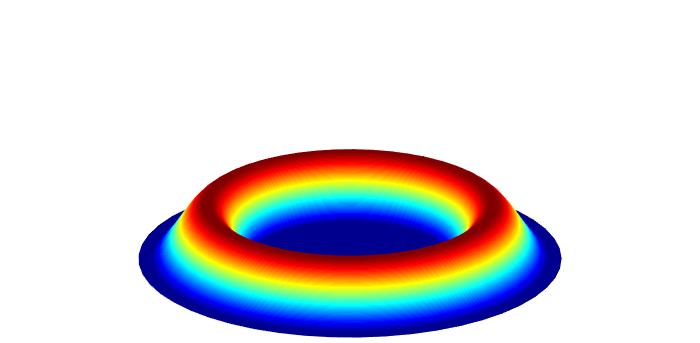}\vspace{-1mm}
\caption*{Volcano}
\end{minipage}
\vspace{-1mm}\caption{Configurations $u$ of Mexican-hat and Volcano solutions of (\ref{ss}) in disk $B_0(5)$, where $R_0\approx 2.7309$ in the Mexican-hat solution (\ref{mexicanhat}), and $\bar R_0^*\approx 3.3831$ in the Volcano solution (\ref{volcano2}).}\label{volcanomexican}
\end{figure}
Then we adjust the mass of each aggregate to match the regularities of $v$ at $r=R_0$.  Before presenting the explicit formula for this non-monotone solution, we first give the necessary and sufficient condition on $R_0$ which guarantees the existence of the roots $r_1$ and $r_4$ for (\ref{23}) and (\ref{231}).  Indeed, there exists one interval of such $R_0$ which gives rise to infinitely many such non-monotone solutions described above.  We give the following lemma.
\begin{lemma}\label{lemma31}
For each $\chi>\chi_2$, let us introduce $\underline R_0$ and $\bar R_0$ as functions of $\chi$ by
\begin{equation}\label{underlineR0}
\underline R_0:=\frac{j_{1,1}}{\omega}
\text{~and~} \bar R_0:=\sup_{\tilde R_0<\frac{j_{1,k}}{\omega}}\Bigg\{\tilde R_0\in(0,R)\Big| \frac{J_1(\omega \tilde R_0)}{Y_1(\omega \tilde R_0)}= \frac{J_1(\omega R)}{Y_1(\omega R)}, \omega R\in(j_{1,k},j_{1,k+1}],k\geq 2\Bigg\}.
\end{equation}
Then $\underline R_0<\bar R_0$; moreover, for each $R_0\in[\underline R_0,\bar R_0]$, $f(r;\omega,R_0)$ in (\ref{23}) admits a unique root $r_1\in(0,\frac{j_{1,1}}{\omega}]$ and $f_4(r;\omega,R_0,R)$ in (\ref{231}) admits a unique root $r_4\in(0,\frac{5\pi}{4\omega})$.
\end{lemma}
\begin{proof}
We first show that $\underline R_0<\bar R_0$.  Since $\omega R\in(j_{1,k},j_{1,k+1}]$, we readily have that $\omega \bar R_0\in(j_{1,k-1},j_{1,k}]$, therefore $\bar R_0>\frac{j_{1,k-1}}{\omega}\geq\frac{j_{1,1}}{\omega}:=\underline R_0$.  To prove the existence and uniqueness of $r_1$, we see that $\omega=\frac{j_{1,1}}{\underline R_0}\geq\frac{j_{1,1}}{R_0}$ for each $R_0\in[\underline R_0,\bar R_0]$, then Lemma \ref{lemma21} readily implies that $f(r;\omega,R_0)$ in (\ref{23}) admits a unique root $r_1\in(\frac{j_{0,1}}{\omega},\frac{j_{1,1}}{\omega}]$.  To prove the statement for $r_4$, we claim that $\chi>\chi_{R_0,R}$ for each $R_0<\bar R_0$, and then according to part \emph{(ii)} of Lemma \ref{lemma23}, (\ref{231}) with $a=R_0$ and $b=R$ admits a unique root $r_4\in (0,s^{(1)}_{1})$.  To verify the claim, we observe from the definitions (\ref{omegaab}) and (\ref{underlineR0}) that $\omega>\omega_{\bar R_0,R}$, and then Lemma \ref{lemma24} implies $\omega_{\bar R_0,R}>\omega_{R_0,R}$ since $\bar R_0>R_0$.  This verifies the claim and completes the proof.
\end{proof}
\begin{remark}
The proof above readily implies that $\bar R_0\in(\frac{j_{1,k-1}}{\omega},\frac{j_{1,k}}{\omega}]$ whenever $\omega R\in(j_{1,k},j_{1,k+1}]$, $k\geq 2$.
\end{remark}
See Figure \ref{functionYJ} for illustration of the verification.  According to Lemma \ref{lemma31}, with $r_1$ and $r_4$ obtained there in hand one is able to solve (\ref{ss}) for inner ring solution over $(0,R_0)$ and outer ring solution over $(R_0,R)$, and then concatenate them to form radially non-monotone solution in the whole region $(0,R)$.  This is what we will show later; moreover, $R_0\rightarrow \frac{j_{1,1}}{j_{1,2}}R$ as $\chi \rightarrow \chi_2$.  Note that here $r_1$ and $r_4$ denote the sizes of support of $u_d(r)$ on the left patch and right patch.

There exists one interval of such $R_0$ which therefore gives rise to a (one-parameter) family of the concatenated non-monotone radial solutions whenever $\chi>\chi_2$.  Though there is no uniqueness for this family of solutions, we are able to present their explicit formulas with a pre-determined parameter $R_0$ in this interval.
\begin{remark}

(i) Lemma \ref{lemma31} states that for each $R_0\in[\underline R_0,\bar R_0]$, $v(r)$ decreases in $(0,R_0)$ and increases in $(R_0,R)$, hence $R_0$ is the (unique) interior critical point of $v(r)$ in $(0,R)$.  For each fixed $\chi>\chi_2$, it is easy to see that the valley of $v(r)$ at $R_0$ can neither be too close to the center nor too close to the boundary.  If one would like shift the valley sufficiently close to the center or the boundary, a sufficiently large $\chi$ is required to this end;

(ii) continued from (i), if we set $R_0$ arbitrary, an equivalent statement of Lemma \ref{lemma31} is that for any $\chi>\max\left\{\left(\frac{j_{1,1}}{R_0}\right)^2,\omega^2_{R_0,R}\right\}+1$, with $\omega_{R_0,R}$ defined in Lemma \ref{lemma23}, the conclusion there holds.

(iii) we claim that $\underline R_0\rightarrow 0^+$ and $\bar R_0\rightarrow R^-$ as $\chi\rightarrow \infty$.  The former is obvious; to show that later, we recall that $y_{1,k+1}-y_{1,k}\rightarrow \pi$ as $k\rightarrow \infty$, therefore for sufficiently large $\omega$, $\omega R-\omega \bar R_0\leq 2\pi$, hence  the later holds.  We would like to mention that Lemma \ref{lemma31} holds for $R_0=\underline R_0$ or $\bar R_0$, however, the asymptotic behaviors of $(u,v)$ when $R_0=\underline R_0$ or $\bar R_0$ are different from those when $R_0\in(\underline R_0, \bar R_0)$.
\end{remark}

Now by the same arguments as above, for each $\chi>\chi_2$ we can construct non-monotone solutions $(u_d(r),v_d(r))$ explicitly given by
\begin{equation}\label{mexicanhat}
\left\{\begin{array}{ll}
u_d(r)=&\left\{\begin{array}{ll}
\mathcal A_1\Big(J_0(\omega r)-J_0(\omega r_1)\Big),&r\in[0,r_1],\\
0,& r\in (r_1,R-r_4),\\
\mathcal A_4 \Big(S_0(R-r;\omega,R)-S_0(r_4;\omega,R)\Big),&r\in[R-r_4,R],\\
\end{array}
\right. \\
\\
v_d(r)=&\left\{\begin{array}{ll}
\mathcal A_1 \Big(\frac{J_0(\omega r)}{\chi}-J_0(\omega r_1)\Big),&r\in[0,r_1],\\
\mathcal B T_0(r;R_0),& r\in (r_1,R-r_4),\\
\mathcal A_4\Big(\frac{1}{\chi}S_0(R-r;\omega,R)-S_0(r_4;\omega,R) \Big),&r\in[R-r_4,R],\\
\end{array}
\right. \\
\end{array}
\right.
\end{equation}
where the coefficients, determined by the continuity of $v(r)$ and conservation of total mass, are explicitly given by
\begin{equation}\label{A1A4B}
\mathcal A_1= \frac{\chi \mathcal B}{1-\chi}\frac{T_0(r_1;R_0)}{J_0(\omega r_1)},
\mathcal A_4= \frac{\chi \mathcal B}{1-\chi}\frac{T_0(R-r_4;R_0)}{S_0(r_4;\omega,R)} \text{~~and~~} \mathcal B =\frac{\bar C_1\bar C_2M}{\bar C_1+\bar C_2},
\end{equation}
where
\[\bar C_1:=\frac{-\omega^2J_0(\omega r_1)}{\pi\chi r_1^2J_2(\omega r_1)T_0(r_1;R_0)} \text{~and~} \bar C_2:=\frac{\omega^2}{\pi\chi \Big(2(R-r_4)T_1(R-r_4;R_0)+ r_4 (2R-r_4)T_0(R-r_4;R_0)\Big)}.\]
Moreover, one can find that the cellular population $m_1$ within inner support disk $B_0(r_1)$ and the cell population $m_4$ within outer support annulus $B_0(R)\backslash B_0(R-r_4)$ can be explicitly presented as
\begin{equation}\label{33}
m_1=\frac{\bar C_2M}{\bar C_1+\bar C_2} \text{~and~} m_4=\frac{\bar C_1M}{\bar C_1+\bar C_2}.
\end{equation}

Then our results can be summarized as follows.
\begin{proposition}\label{proposition31}
Assume that $\chi>\chi_2$.  Then for each $R_0\in[\underline R_0,\bar R_0]$, (\ref{ss}) has a non-monotone solution given by (\ref{mexicanhat}).
\end{proposition}
For each $\chi>\chi_2$, Proposition \ref{proposition31} implies that there exists a one-parameter family of solution of (\ref{ss}).  These solutions are illustrated by shaded area of Figure \ref{branch2}.  As $R_0$ varies, the configuration of the solutions are maintained while their asymptotics can be dramatically different in the large limit of $\chi$, in particular when $R_0=\underline R_0$ or $\bar R_0$ is chosen.  We shall see details in the next subsection.
\begin{figure}[h!]
\vspace{-5mm}
  \centering
\includegraphics[width=1\textwidth]{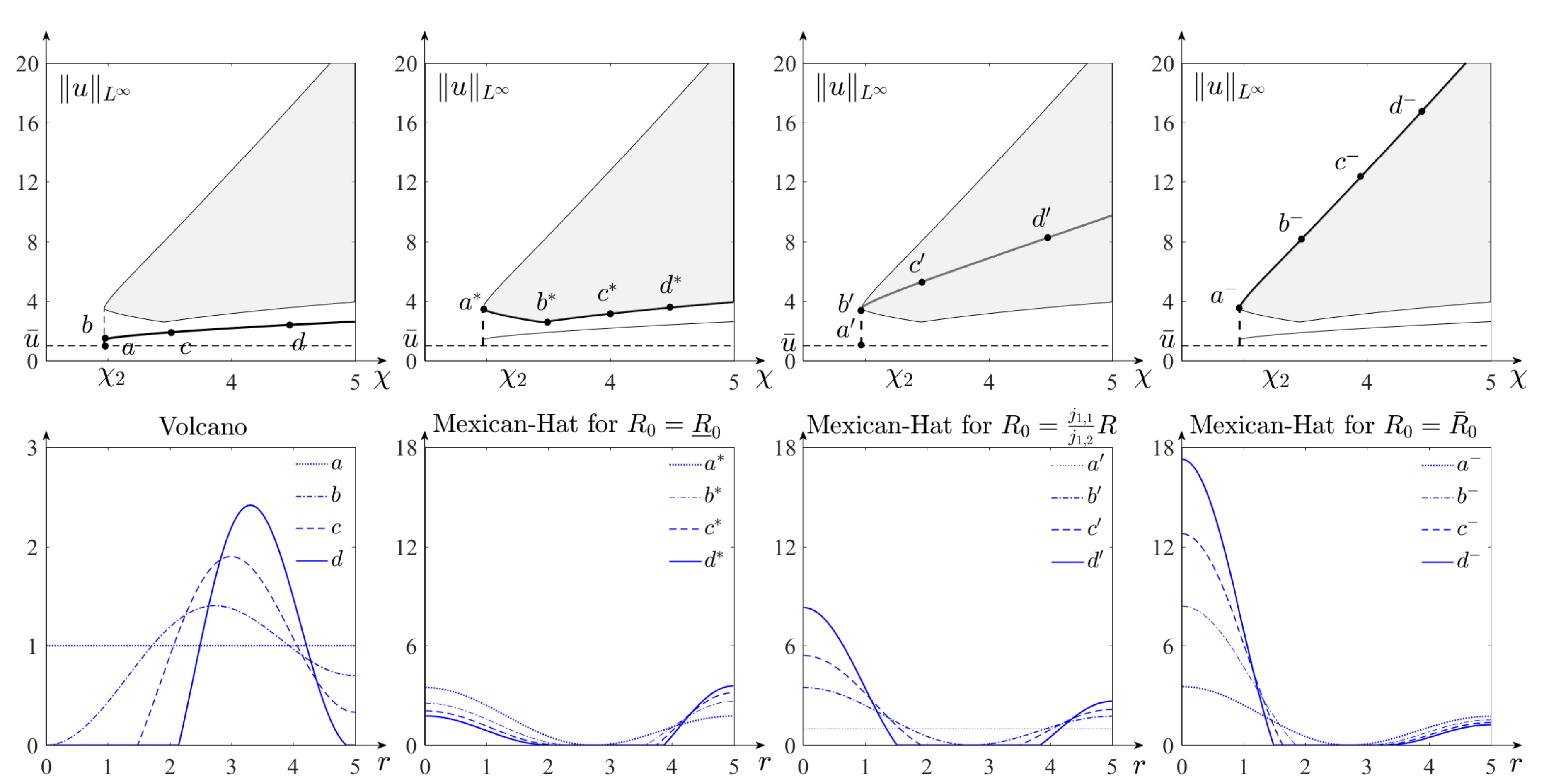}\\
  \caption{Mexican-hat and Volcano solutions on their local and global bifurcation branches out of $\chi\geq\chi_2$.  Each point of the shaded area and the bold curve in the top four graphs gives a solution of (\ref{ss}), whereas for each $\chi>\chi_2$ it has a unique Volcano solution and infinitely many Mexican-hat solutions.  Our results also imply that the magnitude of a Volcano is always smaller than that of the Mexican-hat, while the latter is monotone increasing in $R_0$ in $[\underline R_0,\bar R_0]$.  \textbf{Column 1}: we present  at $a$ and $b$ the local bifurcation solutions given by (\ref{bifurcation}) for $k=2$ with $\varepsilon=0$ and $\varepsilon=-\frac{1}{\chi_2}$, respectively.  When $\chi$ surpasses $\chi_2$, solution $u$ becomes compactly supported, and it stays so and preserves the Volcano profile globally as illustrated by $c$ and $d$.  \textbf{Columns 2-4} present the Mexican-hat solutions obtained in (\ref{mexicanhat}) with $R_0=\underline R_0,\frac{j_{1,1}}{j_{1,2}}R$ and $\bar R_0$ from the second to the fourth graph.  In each class, compactness of support and profile of the solution are preserved along its global branch.}\label{branch2}
\end{figure}

\subsubsection{Asymptotic behaviors of the Mexican-hat solutions}
Next we investigate the asymptotic behaviors of the Mexican-hat solutions in the large limit of $\chi$.  Our result is summarized in the following lemma.
\begin{lemma}\label{lemma32}
Assume that $\chi>\chi_2$ and let $(u_d(r),v_d(r))$ be the solution given by (\ref{mexicanhat}) for arbitrary $R_0\in[\underline R_0,\bar R_0]$.  The followings hold as $\chi\rightarrow \infty$

(i) if $R_0=\underline R_0$, then
\[
u_d(r)\rightarrow
\left\{\begin{array}{ll}
\underline C_0 ,& r=0,\\
M\delta_R(r), & r\in(0,R],\\
\end{array}\right.
\text{~and~}
v_d(r)\rightarrow \frac{MI_0(r)}{2\pi RI_1(R)},
\]
where $\underline C_0=\frac{(J_0(j_{1,1})-1)M}{2\pi RI_1(R)J_0(j_{1,1})}$; moreover, $u_d(R)=\frac{M\omega}{2\pi R}+O (1)$ for $\chi \gg 1$;

(ii) if $R_0\in(\underline R_0,\bar R_0)$, then
\[u_d(r)\rightarrow \frac{C_0 M}{1+C_0}\delta_0(r)+\frac{M}{1+C_0}\delta_R(r) \text{~and~}v_d(r)\rightarrow \frac{T_0(r;R_0)}{2\pi R(1+C_0)T_1(r;R_0)} \text{~pointwisely in $[0,R]$ as~}\chi\rightarrow \infty,\]
where $C_0=\frac{I_1(R_0)}{RT_1(R;R_0)}$; moreover, $u_d(0)=\frac{C_0 M}{1+C_0}\frac{\omega^2}{2\pi j_{0,1}J_1(j_{0,1})}+O(1)$ and $u_d(R)=\frac{M}{1+C_0}\frac{\omega}{2\pi R}+O (1)$ for $\chi \gg 1$;

(iii) if $R_0=\bar R_0$, then
\[
u_d(r)\rightarrow
\left\{\begin{array}{ll}
M\delta_0(r), & r\in[0,R),\\
\underline C_0,& r=R,
\end{array}\right.
\text{~and~}
v_d(r) \rightarrow \frac{MT_0(r;R)}{2\pi I_1(R)},
\]
where $\bar C_0=\frac{M}{\pi RI_1(R)}$; moreover, $u_d(0)=\frac{M\omega^2}{2\pi j_{0,1}J_1(j_{0,1})}+O(1)$ for $\chi\gg 1$;
\end{lemma}
\begin{proof}
\emph{(i)} and \emph{(iii)} are relatively simple and can be verified by the same fashion, and we only do for the first case.  If $R_0=\underline R_0$, we readily know that $r_1=\frac{j_{1,1}}{\omega}$; on the other hand, since $r_1\rightarrow 0^+$ as $\chi\rightarrow \infty$, we can infer from (\ref{33}) that $\frac{m_1}{m_4}\rightarrow 0^+$ hence $m_1\rightarrow 0^+$ and $m_4\rightarrow M^-$.  Therefore, one can find by straightforward calculations that $u_d(0)=\mathcal A_1(1-J_0(j_{1,1}))=\frac{\omega^2m_1(1-J_0(j_{1,1}))}{-\pi j_{1,1}^2J_0(j_{1,1})}\rightarrow \underline C_0$ as claimed.

To prove \emph{(ii)}, we first claim that for each fixed $R_0\in(\underline R_0, \bar R_0)$, the proportion of mass $m_1$ and $m_4$ approaches a positive constant $ C_0:=\frac{I_1(R_0)}{R T_1(R;R_0)}$.  Since $\mathcal B_1=\mathcal{B}_4$, one has for any $\chi>\chi_2$
\[\frac{-m_1}{2r_1 T_1(r_1;R_0)-r_1^2T_0(r_1;R_0)}=\frac{m_4}{2(R-r_4)T_1(R-r_4;R_0)+r_4(2R-r_4)T_0(R-r_4;R_0)};\]
on the other hand, since $\omega r_1 \rightarrow j_{0,1}$ and $r_4 \rightarrow s_0^{(1)}$ as $\chi $ goes to infinity, one can apply the facts that $I_1(0)=0$ and $K_1(r)-\frac{1}{r}=o(r)$ on $\frac{m_1}{m_4}=-\frac{r_1 T_1(r_1;R_0)}{(R-r_4) T_1(R-r_4;R_0)}$ to show that $\frac{m_1}{m_4}\rightarrow C_0$ as $\chi \rightarrow \infty$. Thanks to the asymptotic behavior of $u^-(r)$ and $u_4(r)$, we conclude that $u_d(r)\rightarrow m_1\delta_0(r)+m_4\delta_R(r)$ as described.
\end{proof}

\subsection{Radially Non-Monotone Solutions: The Volcano Solutions}
We next construct $(u_i(r),v_i(r))$ such that the solution first increases and then decreases in $(0,R)$, hence $v(r)$ attains a single local maximum in $(0,R)$ at some $R_0^*$.  In contrast to Mexican-hats where a family of solutions are available (as $R_0$ there varies), such $R_0^*$ is uniquely determined, and so is the solution $(u_i(r),v_i(r))$ as we shall see later on.  This leads to a population density profile that have a volcano-like shape, with maximal bacterial densities occurring on a ring away from the center and boundary of the region.  We call it the Volcano Solution since the configuration of $u$ resembles a volcano pit and it simulates the `volcano effect' \cite{BLL2007,SimoneM2011} that bacteria populations regularly overshoot the peak in chemoattractant concentration and aggregate into a stationary ring of higher density some distance away from an optimal environment.  See the right column of Figure \ref{volcanomexican} for illustration.

These solutions, contrasted with counterparts of the classical Keller-Segel models, shed the light that degenerate diffusion that can cause a volcano effect.
Let us denote $\chi^*_2=(\omega_2^*)^2+1$ with
\[\omega_2^*:=\inf_{w>\frac{j_{1,2}}{R}}\left\{\omega\in\mathbb R^+\Big \vert
\frac{J_1(\omega R)}{Y_1(\omega R)}=\frac{J_0(\omega R)-J_0(\omega (R-\xi))}{Y_0(\omega R)-Y_0(\omega (R-\xi))}, \xi\in(s_0^{(2)},s_1^{(2)}) \text{~s.t~}f_2(\xi;\omega,R)=0\right\}.\]
Our results are described as follows:
\begin{proposition} \label{proposition32}
Let $R$ and $M$ be arbitrary.  For each $\chi>\chi_2$, (\ref{ss}) has a unique solution $(u(r),v(r))$ such that $v(r)$ increases in $(0,R_0^*)$ and decreases in $(R_0^*,R)$ with a constant $R_0^*\in(0,R)$ dependent on $\chi$; moreover, if $\chi\in (\chi_2,\chi^*_2)$, $R_0^*=\bar R_0$ as given by (\ref{underlineR0}) and this solution is explicitly given by (\ref{volcano1}); if $\chi \in [\chi^*_2, \infty)$, this solution is explicitly given by (\ref{volcano2}).
\end{proposition}
The first column of Figure \ref{branch2} illustrates the unique solution obtained in Proposition \ref{proposition32}.  We note that the branch of solution (\ref{volcano1}) extends the local branch at $\chi=\chi_2$; moreover, as $\chi$ expands from $\chi_2$ to $\chi^*_2$, the support of $u(r)$ is of the form $(R_0^*-r_2,R)$, and it attains the critical (maximum) point at $r=\bar R_0$.  The value $u(R)$ drops as $\chi$ increases and touches zero as $\chi$ reaches $\chi^*_2$; then as $\chi$ surpasses $\chi^*_2$, $u$ remains compactly supported in an interval $(R_0^*-r_2^*,R_0^*+r_3^*)$ and this interval shrinks as $\chi$ expands.

According to our discussions above, to obtain this unique solution we shall do the followings: (i) for $\chi \in (\chi_2, \chi^*_2)$, find solution of the form (\ref{outerring})  in $(0,\bar R_0)$ and of the form (\ref{uv5}) in $(\bar R_0,R)$, and then concatenate them at $r=\bar R_0$ by matching the $C^1$ and $C^2$ continuities of $v(r)$ at $r=\bar R_0$; (ii) for $\chi>\chi^*_2$, find solution of the form (\ref{interring}) in $(0,R_0^*)$ and of the form (\ref{227}) in $(R_0^*,R)$, and do the same as in (i).

For each $\chi \in (\chi_2, \chi^*_2)$, Lemma \ref{lemma23} implies that there exists $r_2$ in $(\frac{\pi}{2 \omega}, \frac{5 \pi}{4 \omega})$ such that $(u_i(r),v_i(r))$ is of the form (\ref{229}) with $(a,b)=(0,\bar R_0)$, and $u_i(r)$ is supported in $(\bar R_0,R]$.  Moreover, matching the continuities of $u_i(r)$ and $v_i(r)$ at $r=\bar R_0$ gives for each $\chi \in (\chi_2, \chi^*_2)$
\begin{equation}\label{volcano1}
 \left\{\begin{array}{ll}
   u_i(r)=&\left\{\begin{array}{ll}
0,&r\in[0,\bar R_0-r_2],\\
\mathcal A_2 \left(S_0(\bar R_0-r;\omega,\bar R_0)-S_0(r_2;\omega,\bar R_0)\right),& r\in (\bar R_0-r_2,R],\\
\end{array}
\right.\\
\\
 v_i(r)=&\left\{\begin{array}{ll}
\mathcal B_2 I_0(r),&r\in[0,\bar R_0-r_2],\\
\mathcal A_2 \left(\frac{S_0(\bar R_0-r;\omega,\bar R_0)}{\chi}-S_0(r_2;\omega,\bar R_0)\right),& r\in (\bar R_0-r_2,R],\\
\end{array}
\right.
\end{array}
\right.\\
\end{equation}
where $\mathcal A_2$ and $\mathcal B_2$, determined by the mass $M$ and the continuity at $\bar R_0$, are explicitly given by
\begin{equation}\label{volA2}
\mathcal A_2=\frac{-M \omega}{\pi\Big(2(\bar R_0-r_2)S_1(r_2;\omega,\bar R_0)+\omega [R^2-(\bar R_0-r_2)^2]S_0(r_2;\omega,\bar R_0)\Big)}
\end{equation}
and
\begin{equation}\label{volB2}
\mathcal B_2=\frac{M\omega^2}{\pi\chi\Big(2(\bar R_0-r_2)T_1(\bar R_0-r_2;\bar R_0)+ [R^2-(\bar R_0-r_2)^2]T_0(\bar R_0-r_2;\bar R_0)\Big)}.
\end{equation}

Next we consider the case $\chi>\chi^*_2$.  First note that $u_i(r)$ given by (\ref{volcano1}) is no longer a solution since it becomes negative somewhere in $(0,R)$.  To see this, we claim that $u_i(R)<0$ if and only if $\chi>\chi^*_2$.  Indeed, let us denote $r_{\nu}:=R-\bar R_0+r_2$, then $S_0(r_2;\omega,\bar R_0)=\frac{Y_1(\omega \bar R_0)}{Y_1(\omega R)} S_0(r_{\nu};\omega,R)$, hence we see that $r_{\nu}$ is also the second positive root of (\ref{210}) since $\bar R_0$ satisfies (\ref{underlineR0}).  Thanks to (\ref{volcano1}) we find $u_i(R) =- \mathcal A_2 \frac{Y_1(\omega \bar R_0)}{Y_1(\omega R)} \left(S_0(r_\nu;\omega,R)+\frac{2}{\pi \omega R} \right)$, with $Y_1(\omega \bar R_0) Y_1(\omega R)<0$ and $\mathcal A_2<0$.  On the other hand, we find that $\frac{\partial r_{\nu}}{\partial \omega}<0$, and $S_0(r;\omega,R)+\frac{2}{\pi \omega R}=0$ for $\chi=\chi^*_2$, therefore $u_i(R)<0$ if and only if $\chi>\chi^*_2$ as claimed.

The discussions above indicate that (\ref{volcano1}) only holds for $\chi \in (\chi_2, \chi^*_2]$, but not for $\chi>\chi_2^*$.  Moreover, as $\chi$ surpasses $\chi_2^*$, $u$ is supported in $(R_0^*-r_2^*,R_0^*+r_3^*)$ for some $r_2^*$ and $r_3^*$ to be determined and the solution now takes the following form
\begin{equation}\label{volcano2}
\left\{\begin{array}{ll}
u_i^*(r)=&\left\{\begin{array}{ll}
0,&r\in[0,R_0^*-r_2^*],\\
\mathcal A_2^* \left(S_0(R_0^*-r;\omega,R_0^*)-S_0(r_2^*;\omega,R_0^*)\right),& r\in (R_0^*-r_2^*,R_0^*+r_3^*),\\
0,&r\in[R_0^*+r_3^*,R],\\
\end{array}
\right.\\
\\
v_i^*(r)=&\left\{\begin{array}{ll}
\mathcal B_2^* I_0(r),&r\in[0,R_0^*-r_2^*],\\
\mathcal A_2^* \left(\frac{S_0(R_0^*-r;\omega,R_0^*)}{\chi}-S_0(r_2^*;\omega,R_0^*)\right),& r\in (R_0^*-r_2^*,R_0^*+r_3^*),\\
\mathcal B_3^* T_0(r;R),&r\in[R_0^*+r_3^*,R],\\
\end{array}
\right.
\end{array}
\right.
\end{equation}
where $r_2^*$ and $r_3^*$ denote the support of $u_i^*(r)$ in $(0,R_0^*)$ and $(R_0^*,R)$, and $r_2^*$ satisfies (\ref{210}) with $R$ replaced by $R_0$ and $r_3^*$ satisfies (\ref{228}) with $(a,b)$ replaced by $(R_0^*,R)$, respectively, and
\[\mathcal B_2^*:=\frac{(1-\chi)\mathcal A_2^*}{\chi} \frac{S_0(r_2^*;\omega,R_0^*)}{I_0(R_0^*-r_2^*)}, \mathcal B_3^*:=\frac{(1-\chi)\mathcal A_2^*}{\chi} \frac{S_0(r_2^*;\omega,R_0^*)}{T_0(R_0^*+r_3^*;R)} \text{~and~} \mathcal A_2^*:=\frac{\bar C_3\bar C_4M}{\bar C_3+\bar C_4},\]
where
\[\bar C_3:=-\frac{\omega}{\pi\left(2(R_0^*-r_2^*)S_1(r_2^*;\omega,R_0^*)+\omega r_2^* (2R_0^*-r_2^*)S_0(r_2^*;\omega,R_0^*)\right)}\]
and
\[\bar C_4:=\frac{\omega}{\pi \left(2(R_0^*+r_3^*)S_1(-r_3^*;\omega,R_0^*)-\omega r_3^* (2R_0^*+r_3^*)S_0(-r_3^*;\omega,R_0^*)\right)};\]
moreover, we find from the conservation of mass that the left (half) aggregate weights $m_2$ and the right (half) aggregate weights $m_3$ with
\begin{equation}\label{m2m3}
m_2=\frac{\bar C_4 M}{\bar C_3+\bar C_4} \text{~and~} m_3=\frac{\bar C_3 M}{\bar C_3+\bar C_4}.
\end{equation}

In strong contrast to (\ref{mexicanhat}) where $R_0\in[\underline R_0,\bar R_0]$ can be chosen arbitrarily, $R_0^*$ in (\ref{volcano2}) is uniquely determined and we have the following result.
\begin{proposition}\label{proposition33}
Assume that $\chi>\chi^*_2$, and the solution of (\ref{ss}) takes the form (\ref{volcano2}) with some $R_0^*$.  Then such $R_0^*$ is unique hence the solution of this form is unique.
\end{proposition}
\begin{proof}
For the solution of the form (\ref{volcano2}), the continuity of $v$ at $r=R_0^*$ requires
\[G(R_0^*;\omega):=S_0(r_2^*;\omega,R_0^*)-S_0(-r_3^*;\omega,R_0^*)=0.\]
We are left to verify that such $R_0^*\in(0,R)$ exists and is indeed unique.  According to Lemma \ref{lemma23}, both $r_2^*$ and $r_3^*$ exist in $(0,R_0^*)$ and $(R_0^*,R)$ respectively for each $R_0^*\in(\underline R_0,\bar R_0 )$, with $\underline R_0 =\frac{j_{1,1}}{\omega}$ and $\bar R_0$ given by (\ref{underlineR0}).

We first study the limits of $G$ as follows: as $R_0^* \rightarrow \underline R_0^+ $, the facts $\omega R_0^* \rightarrow j_{1,1}$ and $r_2^*\rightarrow \underline R_0 $ imply that $J_1(\omega R_0^*)\rightarrow 0$ and $J_0(\omega (R_0^*-r_2^*))\rightarrow 1$, therefore we have
\[G\left(\big(\underline R_0  \big)^+;\omega\right)=Y_1(j_{1,1})(1-J_0(\omega(\underline R_0 +r_3^*)))>0.\]
Indeed one can show that $G\left(\big(\underline R_0 \big)^+;\omega\right)>0.2$ using the estimates of Bessel functions in \cite{Landau}. On the other hand, as $R_0^*\rightarrow (\bar R_0)^-$, $r_3^*\rightarrow R-\bar R_0$ and the definition of $\chi_{\bar R_0,R}$ implies that $\frac{J_1(\omega R_0^*)}{Y_1(\omega R_0^*)} \rightarrow \frac{J_1(\omega R)}{Y_1(\omega R)}$, therefore
\[G\left(\big(\bar R_0  \big)^-;\omega \right)=\frac{Y_1(\omega \bar R_0 )}{Y_1(\omega R)} \left(S_0(r_\nu;\omega,R)+\frac{2}{\pi \omega R} \right)<0,\]
where we have applied the fact that $Y_1(\omega \bar R_0 ) Y_1(\omega R)<0$.  Therefore, for each $\chi>\chi^*_2$ there exists at least one $R_0^*\in(\underline R_0,\bar R_0)$ such that $G(R_0^*;\omega)=0$.

Next we prove the uniqueness of $R_0^*$.  According to our discussions above, it is sufficient to show that $\frac{\partial G(R_0^*;\omega)}{\partial R_0^*}+\frac{1}{R}  G(R_0^*;\omega)<0$.  Straightforward calculations give
\begin{align}
  \frac{\partial G(R_0^*;\omega)}{\partial R_0^*}  =&\frac{\partial S_0(r_2^*;\omega,R_0^*)}{\partial R_0^*}+\frac{\partial S_0(r_2^*;\omega,R_0^*)}{\partial r_2^*}  \frac{\partial r_2^*}{\partial R_0^*}-\frac{\partial S_0(-r_3^*;\omega,R_0^*)}{\partial R_0^*}-\frac{\partial S_0(-r_3^*;\omega,R_0^*)}{\partial r_3^*}  \frac{\partial r_3^*}{\partial R_0^*} \notag\\
    =&-\frac{1}{R}  G(R_0^*;\omega)+\omega S_1(r_2^*;\omega,R_0^*)  \left(\frac{\partial r_2^*}{\partial R_0^*}-1+\frac{\mathcal V_0(r_2^*;\omega, R_0^*)}{S_1(r_2^*;\omega,R_0^*)} \right)\notag \\
   &+ \omega S_1(-r_3^*;\omega,R_0^*)  \left(\frac{\partial r_3^*}{\partial R_0^*}+1-\frac{\mathcal V_0(-r_3^*;\omega, R_0^*)}{S_1(-r_3^*;\omega,R_0^*)} \right).\notag
\end{align}
Recall from Section 2.2 that $y_6(r_2^*;\omega,R_0^*)<0$, that is, $\frac{\mathcal V_0(r_2^*;\omega, R_0^*)}{S_1(r_2^*;\omega,R_0^*)}<1$.  Now we claim that $\frac{\partial r_2^*}{\partial R_0^*}\leq 0$.  To see this, some tedious but straightforward calculations give
\[\frac{\partial r_2^*}{\partial R_0^*} =-\frac{\partial_{R_0^*}f_2(r_2^*;\omega,R_0^*)}{\partial_rf_2(r_2^*;\omega,R_0^*)}= -\frac{\partial_{R_0^*}f_2(r_2^*;\omega,R_0^*)}{\omega^2+1}=1-\frac{4}{\pi^2(\omega^2+1)R_0^*(R_0^*-r_2^*) S_1^2(r_2^*;\omega,R_0^*)},\]
thanks to the fact that $r_2^*$ is a root of (\ref{210}) with $R$ replaced by $R_0^*$.  Then $S_1^2(r_2^*;\omega,R_0^*)\leq \frac{4}{\pi^2(\omega^2+1)R_0^*(R_0^*-r_2^*)}$ implies the claim and proves the uniqueness, and we are left to verify this inequality.  Denote
\[y_8(R_0^*;\omega):=S_1^2(r_2^*;\omega,R_0^*)-\frac{4}{\pi^2(\omega^2+1)R_0^*(R_0^*-r_2^*)}.\]
First of all, one sees that $y_8((\underline R_0)^+;\omega)\rightarrow -\infty$.  We next prove $y_8(R_0^*;\omega)\leq 0$ for each $R_0^* >\underline R_0$ to verify the claim.  We argue by contradiction: if not and there exists some $R_1>\underline R_0$ such that $y_8(R_1;\omega)=0$ and $y'_8(R_1;\omega)\geq0$.  Then we find that $y'_8(R_1;\omega)=2 S^2_1(r_2^*;\omega,R_1) y_9(r_2^*;\omega,R_1)$ with
\[y_9(r;\omega,R_1):=\frac{\omega\big(\mathcal V_1(r;\omega,R_1)+S_0(r;\omega,R_1)\big)}{S_1(r;\omega,R_1)}-\frac{2R_1-r}{2R_1(R_1-r)},\]
and further calculations give $y'_9(r;\omega,R_1)=\frac{y_{10}(r;\omega,R_1)}{(R_1-r)S^2_1(r;\omega,R_1)}$, where
\begin{align*}
  y_{10}(r;\omega,R_1):=& \omega^2(R_1-r)\Big(S^2_1(r;\omega,R_1)+S^2_0(r;\omega,R_1)\Big)\\
   &-\omega S_0(r;\omega,R_1)S_1(r;\omega,R_1)-\frac{S^2_1(r;\omega,R_1)}{2(R_1-r)}-\frac{4}{\pi^2R_1};
\end{align*}
moreover $y_{10}(0;\omega,R_1)=0$ and $y'_{10}(r;\omega,R_1)=-\frac{3S^2_1(r;\omega,R_1)}{2(R_1-r)^2}<0$, therefore $y_{10}(r;\omega,R_1)<0$ hence $y_9(r;\omega,R_1)<y_9(0;\omega,R_1)<0$ in $(0,s_1^{(1)})$, which further implies that $y'_8(R_1;\omega)<0$, a contradiction.

One can use the same arguments as above to show that $\mathcal V_0(-r_3^*;\omega,R_0^*)>S_1(-r_3^*;\omega,R_0^*)$ and $\frac{\partial r_3^*}{\partial R_0^*}\geq 0$ for each $\chi>\chi_2^*$, then applying the facts $S_1(r_2^*;\omega,R_0^*)>0$ and $S_1(-r_3^*;\omega,R_0^*)<0$ gives us
\[\omega S_1(r_2^*;\omega,R_0^*)  \left(\frac{\partial r_2^*}{\partial R_0^*}-1+\frac{\mathcal V_0(r_2^*;\omega, R_0^*)}{S_1(r_2^*;\omega,R_0^*)} \right)+\omega S_1(-r_3^*;\omega,R_0^*)  \left(\frac{\partial r_3^*}{\partial R_0^*}+1-\frac{\mathcal V_0(-r_3^*;\omega, R_0^*)}{S_1(-r_3^*;\omega,R_0^*)} \right)<0,\]
which indicates that the derivative of the root of $G(R_0^*;\omega)$ is strictly negative and eventually gives rise to the uniqueness of $R_0^*$ as claimed.
\end{proof}

\subsubsection{Asymptotic behaviors of the volcano solution}
To study the asymptotics of the unique solution (\ref{volcano2}) as the chemotaxis rate $\chi$ goes to infinity, we first claim that in this limit $R_0^*$ approaches a constant $\hat R_0$ in $(0,R)$ that satisfies
\begin{equation}\label{314}
  \frac{I_0(\hat R_0)}{I_1(\hat R _0)}=-\frac{T_0(\hat R _0;R)}{T_1(\hat R _0;R)}.
\end{equation}
Note that both $\frac{I_0(r)}{I_1(r)}$ and $\frac{T_0(r;R)}{T_1(r;R)}$ are monotone decreasing in $(0,R)$, hence this $\hat R _0$ is uniquely determined.  Then we will show that that $u_i(r)\rightarrow M\delta_{\hat R _0}(r)$; moreover, our results indicate that the mass of $u_i(r)$ on the left hand side and the right hand side of $R_0^*$ goes to $\frac{M}{2}$.

To prove the claim, we first observe that
\begin{equation}\label{315}
  \frac{I_0(R_0^*-r_2) S_1(r_2;\omega,R_0^*)}{I_1(R_0^*-r_2)}=\frac{T_0(R_0^*+r_3;R) S_1(-r_3;\omega,R_0^*)}{T_1(R_0^*+r_3;R)}.
\end{equation}
Then in light of (\ref{75}), one can apply the facts that both $r_2,r_3 \rightarrow \frac{\pi}{2 \omega}$ as $\chi\rightarrow \infty$ to obtain
 \begin{equation}\label{317}
   S_1(r_2;\omega,R_0^*)\rightarrow -S_1(-r_3;\omega,R_0^*).
 \end{equation}
Now that $r_2,\;r_3 \rightarrow 0^+$, we conclude that $u_i(r)\rightarrow M\delta_{\hat R _0}(r)$, (\ref{315}) converges to (\ref{314}) and its root approaches $\hat R _0$ as claimed.  Furthermore, it is straightforward to see from (\ref{m2m3}) and (\ref{317}) that both $m_2$ and $m_3$ approach $\frac{M}{2}$.  The results above are summarized into the following lemma.
\begin{figure}[h!]\vspace{-5mm}
  \centering
\includegraphics[width=1\textwidth]{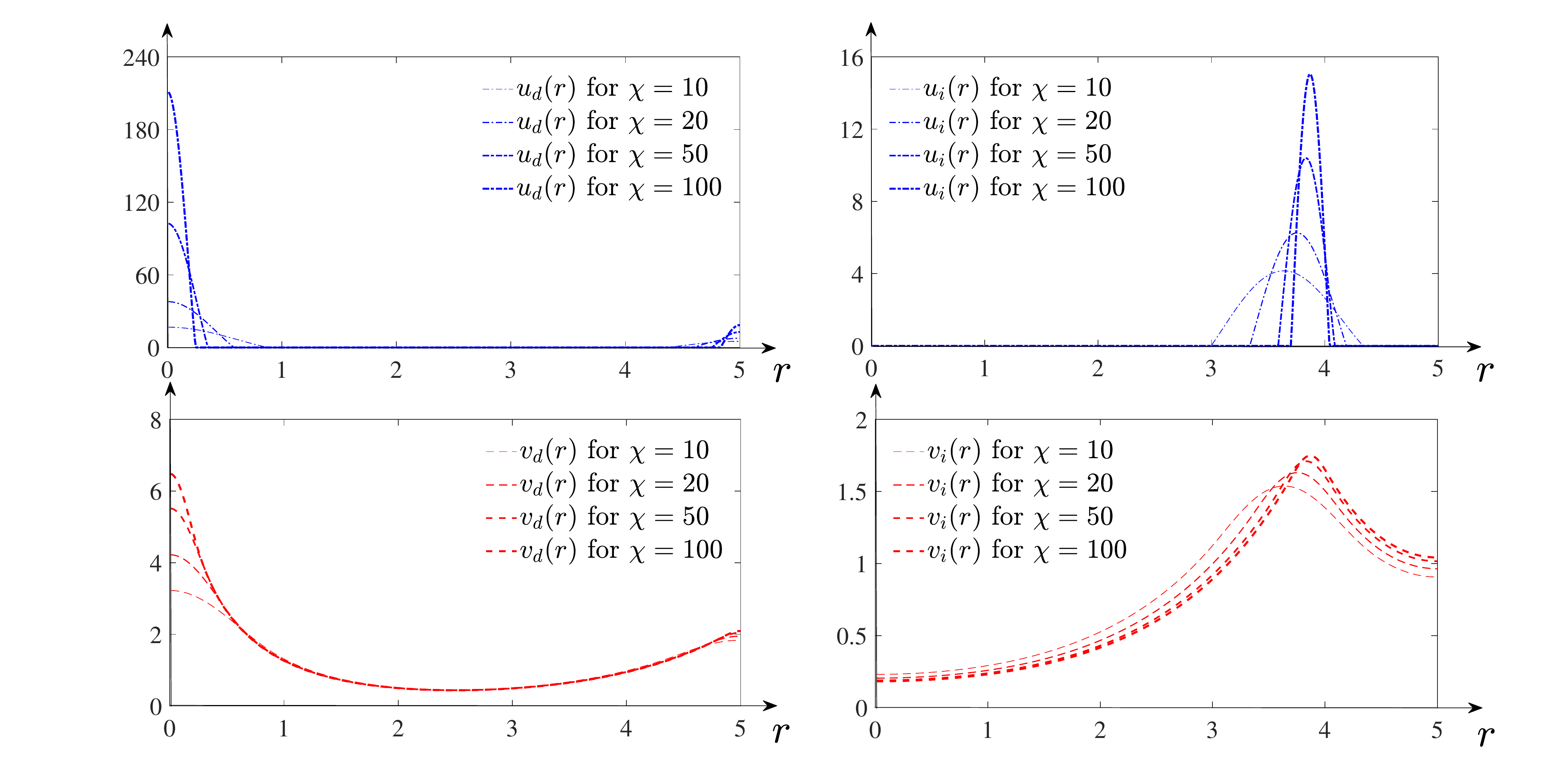}\vspace{-5mm}
\caption{Illustration of the asymptotic behaviors obtained in Lemma \ref{lemma32} and Lemma \ref{lemma33} with $R=5$, $M=25\pi$.  For each $\chi$ given, a Mexican-hat has a much larger spiky structure than the Volcano.  \textbf{Left Column}: the Mexican-hat solution $(u_d,v_d)$ in (\ref{mexicanhat}) with $R_0=2.5\in(\underline R_0,\bar R_0)$ for $\chi=10,20,50$ and 100.  It is observed that $u_d(r)\rightarrow 17.2263\delta_0(r)+61.3135\delta_R(r)$ as $\chi\rightarrow \infty$;  \textbf{Right Column}: the Volcano solution $(u_i,v_i)$ in (\ref{volcano2}) for $\chi=5$, 10, 50 and 100.  We find that the spike locates at $R^*_0$ whereas $R^*_0\nearrow \hat R_\infty\approx3.8696$ as $\chi\rightarrow \infty$ and $u_i(r)\rightarrow 25\pi \delta_{\hat R_\infty}(r)$ as $\chi\rightarrow\infty$.}\label{mexicanhatasymptotic}
\end{figure}
\begin{lemma}\label{lemma33}
Let $\chi>\chi_2^*$ and $(u_i,v_i)$ be the solution given by (\ref{volcano2}).  Then
\[u_i(r)\rightarrow  M\delta_{\hat R _0}(r) \text{~and~}v_i(r)\rightarrow
\left\{\begin{array}{ll}
\frac{M}{4\pi \hat R_0 I_1(\hat R_0)}I_0(r),& r\in [0,\hat R_0],\\
-\frac{M}{4\pi \hat R_0 T_1(\hat R_0;R)}T_0(r;R),& r\in(\hat R_0,R],
\end{array}
\right. \text{~pointwisely in $[0,R]$ as~} \chi\rightarrow\infty,\]
where $\hat R_0\in(0,R)$ is the unique root of (\ref{314}).
\end{lemma}

\section{Higher-Order Modes and Airy Patterns}\label{section4}

We next demonstrate that (\ref{ss}) supports higher-order modes ring patterns for $\chi$ large.  Our results show that the indices of its modes can be tuned by adjusting the size of the chemotaxis rate.  The approach provides a path to better understand (\ref{01}) which can have simultaneously large mode areas and large separations between the each mode.

We continue to study the solutions that extend the local (vertical) bifurcation branch at $\chi=\chi_k$ which give rise to nontrivial patterns.  Suppose that $\chi>\chi_k$, $k\geq3$.  For each integer $k_0\in[2,k]$, let us define
\begin{equation}\label{underlineR0k}
\bar R_0^{(k_0)}:=\sup_{\tilde R_0<\frac{j_{1,k-k_0+2}}{\omega}}\Bigg\{\tilde R_0\in(0,R)\Big| \frac{J_1(\omega \tilde R_0)}{Y_1(\omega \tilde R_0)}= \frac{J_1(\omega R)}{Y_1(\omega R)}, \omega R\in(j_{1,k},j_{1,k+1}],k\geq 3\Bigg\}.
\end{equation}
Note that $\bar R_0$ in (\ref{underlineR0}) is $\bar R_0^{(2)}$ in (\ref{underlineR0k}).  Then we find $(u,v)$ such that $v'$ changes sign in $(0,R)$ for $(k_0-1)$ times in the following lemma.  See Figure \ref{branch3} for their bifurcations.
\begin{figure}[h!]
\vspace{-3mm}
  \centering
  \includegraphics[width=1\textwidth]{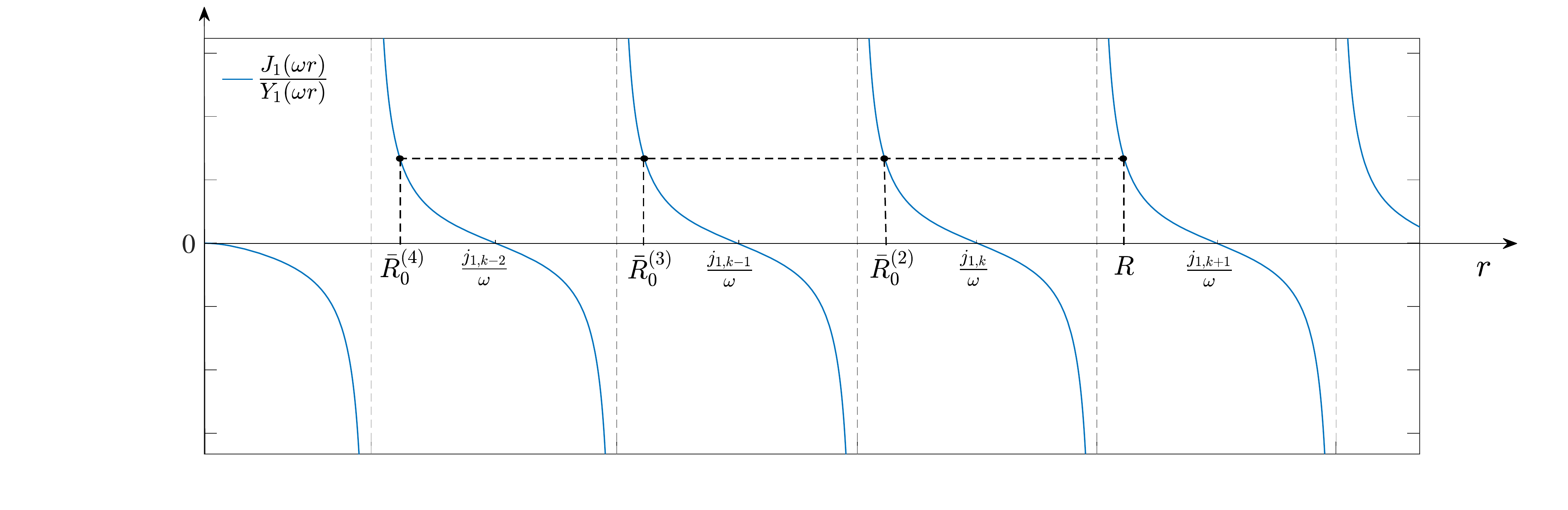}\vspace{-5mm}
  \caption{Function $\frac{J_1(r)}{Y_1(r)}$ and an illustration of $\bar R_0^{(k_0)}$ given by (\ref{underlineR0k}).}\label{functionYJ}
\end{figure}
\begin{lemma}\label{lemma41}
Suppose that $\chi>\chi_k$ and let $\bar R_0^{(k_0)}$ be given by (\ref{underlineR0k}).  Then for each $k_0\in[2,k]$, (\ref{ss}) has a one-parameter family of solutions $(u,v)$ such that $v'$ changes sign $(k_0-1)$ times; moreover, for any $R_0\in[\underline R_0,\bar R_0^{(k_0)}]$, the solution has the following explicit formula
\begin{equation}\label{mexicanhat-new}
\left\{\begin{array}{ll}
u^{(k_0)}_d(r)=&\left\{\begin{array}{ll}
\mathcal A_1\Big(J_0(\omega r)-J_0(\omega r_1)\Big),&r\in[0,r_1],\\
0,& r\in (r_1,\bar R_0^{(k_0-1)}-r_4),\\
\mathcal A_4 \Big(S_0(\bar R_0^{(k_0-1)}-r;\omega,\bar R_0^{(k_0-1)})-S_0(r_4;\omega,\bar R_0^{(k_0-1)})\Big),&r\in[\bar R_0^{(k_0-1)}-r_4,R],\\
\end{array}
\right. \\
\\
v^{(k_0)}_d(r)=&\left\{\begin{array}{ll}
\mathcal A_1 \Big(\frac{J_0(\omega r)}{\chi}-J_0(\omega r_1)\Big),&r\in[0,r_1],\\
\mathcal B T_0(r;R_0),& r\in (r_1,\bar R_0^{(k_0-1)}-r_4),\\
\mathcal A_4\Big(\frac{1}{\chi}S_0(\bar R_0^{(k_0-1)}-r;\omega,\bar R_0^{(k_0-1)})-S_0(r_4;\omega,\bar R_0^{(k_0-1)})\Big),&r\in[\bar R_0^{(k_0-1)}-r_4,R],\\
\end{array}
\right. \\
\end{array}
\right.
\end{equation}
where the coefficients $\mathcal A_1$, $\mathcal A_1$ and $\mathcal B$, determined by the continuity of $v(r)$ and conservation of total mass, are given by (\ref{A1A4B})
with
\[\bar C_1:=\frac{-\omega^2J_0(\omega r_1)}{\pi\chi r_1^2J_2(\omega r_1)T_0(r_1;R_0)}, \bar C_2:=\frac{\omega^2}{\pi\chi \Big(2(\bar R_0^{(k_0)}\!-r_4)T_1(\bar R_0^{(k_0)}-r_4;R_0)+ [R^2-(\bar R_0^{(k_0)}\!-r_4)^2]T_0(\bar R_0^{(k_0)}\!-r_4;R_0)\Big)};\]
\begin{figure}[h!]
\vspace{-5mm}
\centering
  \includegraphics[width=1\textwidth]{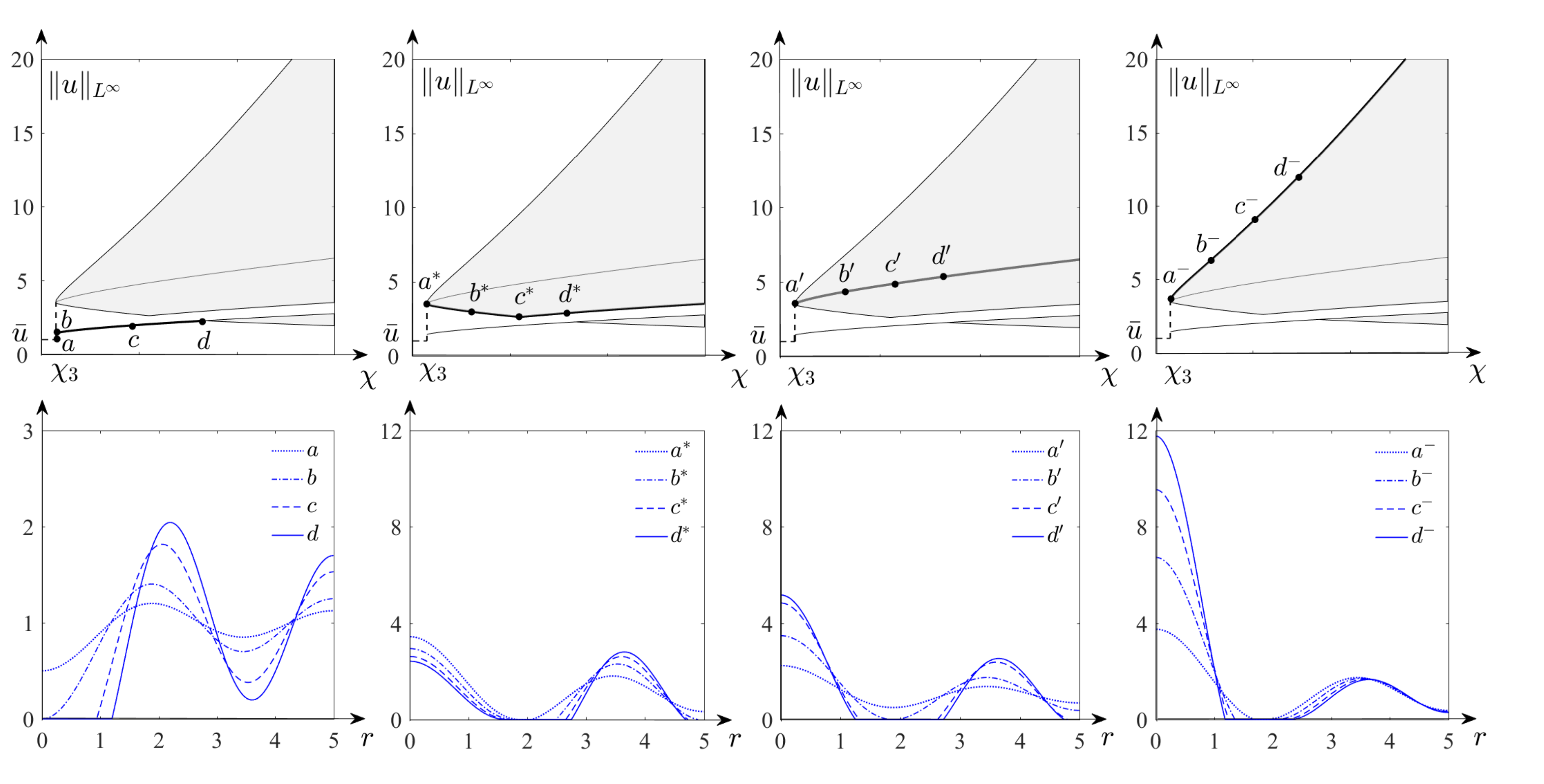}\vspace{-5mm}
  \caption{Bifurcation branches out of $\chi_3$ and their global continuums.  Similar as in Figure \ref{branch2}, this gives to a few classes of non-monotone solutions of (\ref{ss}) from (\ref{mexicanhat-new}) and (\ref{volcano1-new}) for $\chi$ large.  The spatial structure of the stationary solutions becomes complex and rich as $\chi$ increases.}\label{branch3}
\end{figure}

moreover, there also exist solutions of the following form
\begin{equation}\label{volcano1-new}
 \left\{\begin{array}{ll}
u^{(k_0)}_i(r)=&\left\{\begin{array}{ll}
0,&r\in[0,\bar R_0^{(k_0)}-r_2^*],\\
\mathcal A^*_2 \left(S_0(\bar R_0^{(k_0)}-r;\omega,\bar R_0^{(k_0)})-S_0(r_2^*;\omega,\bar R_0^{(k_0)})\right),& r\in (\bar R_0^{(k_0)}-r_2^*,R],\\
\end{array}
\right.\\
\\
v^{(k_0)}_i(r)=&\left\{\begin{array}{ll}
\mathcal B^*_2 I_0(r),&r\in[0,\bar R_0^{(k_0)}-r_2^*],\\
\mathcal A^*_2 \left(\frac{S_0(\bar R_0^{(k_0)}-r;\omega,\bar R_0^{(k_0)})}{\chi}-S_0(r_2^*;\omega,\bar R_0^{(k_0)})\right),& r\in (\bar R_0^{(k_0)}-r_2^*,R],\\
\end{array}
\right.
\end{array}
\right.\\
\end{equation}
where $\mathcal A^*_2$ and $\mathcal B^*_2$, determined by the mass $M$ and the continuity at $\bar R_0^{(k_0)}$, are given by (\ref{volA2}) and (\ref{volB2}) with $\bar R_0$ replaced by $\bar R_0^{(k_0)}$.
\end{lemma}
We would like to point out that, one can continue to construct more solutions following the approaches above by sending $\chi$ larger and larger.  This will give rises to spatial patterns with more and higher modes as we described above even for a fixed $k_0$.  Indeed, similar as for the Volcano solutions, there exists some $\chi_{k_0}^*>\chi_k$ such that (\ref{volcano1-new}) will take a new formula such that $u_i$ touches zero at the end point $r=R$ as $\chi$ surpass $\chi_{k_0}^*$, hence it stays compactly supported of the same form as (\ref{volcano2}).  This results in another class of solutions for (\ref{ss}), which can be explicitly obtained but are skipped here for brevity.

Figure \ref{Airy} illustrates the airy patterns developed in (\ref{ss}).  For each configuration, the circular aperture has a bright central region called the airy disk, and a series of concentric rings around called the Airy pattern.  Such disk and rings phenomenon is well adopted to describe the appearance of a bright star seen through a telescope under high magnification.
\begin{figure}[h!]
\vspace{-4mm}
\centering
  \includegraphics[width=1\textwidth]{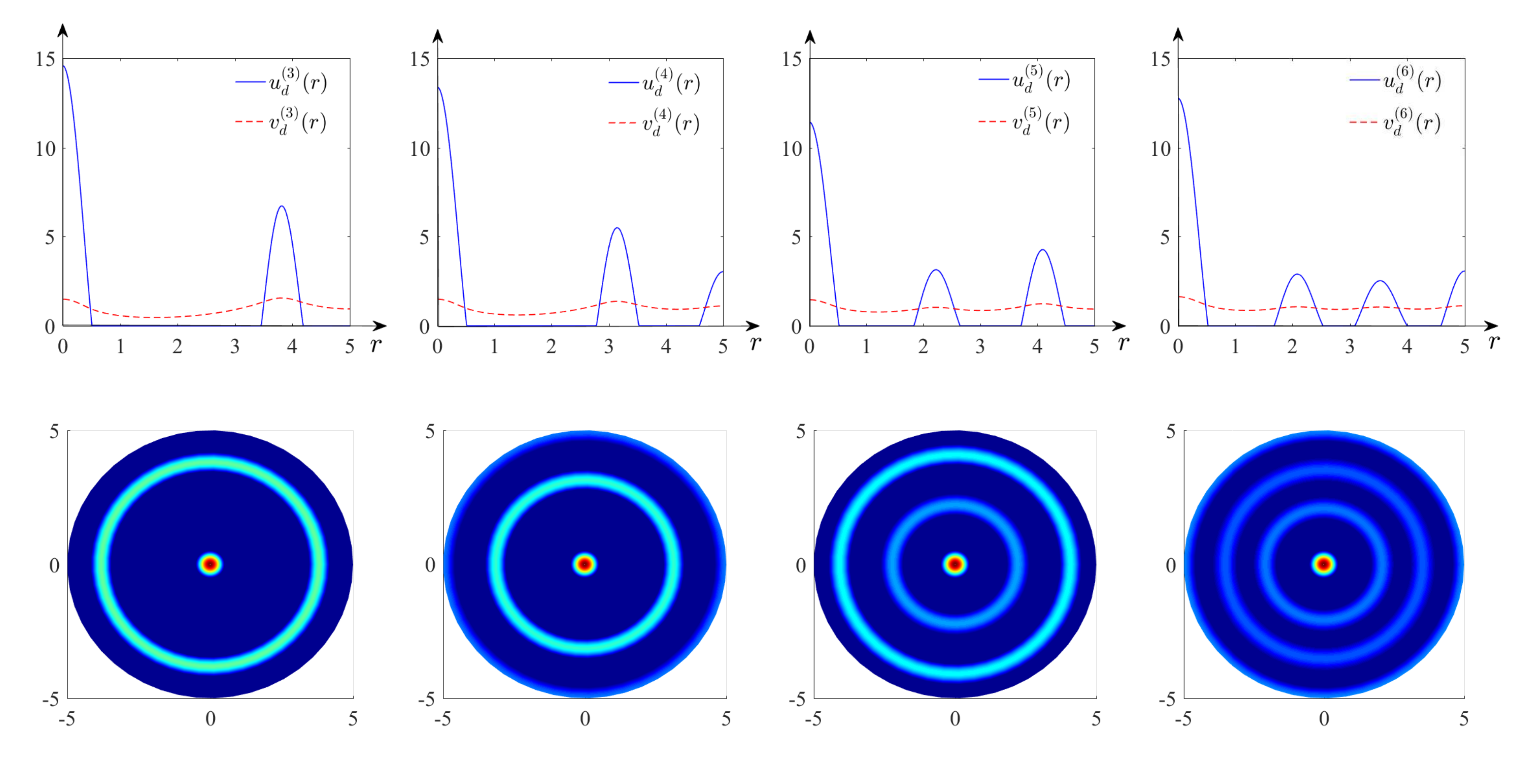}\vspace{-5mm}
  \caption{Airy patterns of (\ref{ss}) over $B_0(5)$ with $\chi=26$ and $k_0=3,4,5$ and $6$ from the left to the right.  Lemma \ref{lemma41} implies that for each $k_0\leq7$ there exists non-monotone patterns such that $v'$ changes sign $(k_0-1)$ times.  \textbf{Top}: plots of $(u_d^{(k_0)}(r),v_d^{(k_0)}(r))$; \textbf{Bottom}: each configuration $u_d^{(k_0)}$ contains a bright central region and a series of concentric rings around.}\label{Airy}
\end{figure}

There are several observations we made when analyzing the solutions above.  While rigourous proofs are lacking, we record them here for future reference and the reader's interest.  Our analysis suggests the following findings: suppose that the closure of support of $u$ consists of a finite number of disjoint components, i.e., $\overline{\text{Supp}(u)}=\cup_{i=1}^N\text{Spt}_i$, then the solution (\ref{ss}) has a freedom of $N-1$ in the sense that it can be uniquely determined by $N-1$ parameters.  With that being said, one readily has that $(u,v)$ is unique if the support of $u$ is connected.  This is well supported whenever $\chi$ is slightly large than $\chi_k$.

\section{Hierarchy of Free Energies and Non-radial Solutions}\label{section6}
Now we proceed to study the free energy (\ref{freeenergy}) and its dissipation (\ref{dissipation}).  This free energy allows for a gradient flow structure of (\ref{01}) in a product space, see \cite{CC,CLM,BL,BCKKLL}. The hybrid gradient flow structure introduced in \cite{BL,BCKKLL} treats the evolution of the cell density in probability measures while the evolution of the chemoattractant is done in the $L^2$-setting. The gradient flow structure used in probability densities follows the blueprint of the general gradient flow equations treated in \cite{JKO98,AGS,Vil_1,CMcCV03}. Moreover, solutions were proved to be unique among the class of bounded densities \cite{CLM}.  We would like to mention that (\ref{freeenergy}) is a Lyapunov functional since steady states $(u_s, v_s)$ are characterized by zero dissipation $\mathcal I(u_s, v_s)=0$.  Therefore, we readily obtain that for a steady state $(u_s, v_s)$, the quantity $u_s -\chi v_s$ must be constant in each connected component of the support of the cell density $u_s$.  Note that the $v$--equation of (\ref{ss}) readily gives us
\[\int_{B_0(R)} \big(|\nabla v_s|^2+v_s^2\big)d \textbf{x}=\int_{B_0(R)} u_sv_sd \textbf{x}.\]
Therefore
\begin{equation}\label{63}
\mathcal E(u_s, v_s)=\frac{1}{\chi}\int_{B_0(R)} u_s (u_s-\chi v_s)d \textbf{x}.
\end{equation}
Moreover, since $u_s-\chi v_s=\lambda _i$ for some constants $\lambda _i$ on each of the (possibly countably many) connected components of the support of $u_s$, denoted by $\text{sppt}_i$, we have from (\ref{63}) that
\begin{equation}\label{64}
\mathcal E(u_s,v_s)=\frac{1}{\chi}\sum_{i}\int_{\text{sppt}_i} \lambda _i u_sd \textbf{x}.
\end{equation}
We now study the energy of the steady states constructed above, and in light of the explicit formulas, we shall give a hierarchy of the energies.  Among other things, we show that the constant solution has the largest energy among all solutions and the single interior bump has the least energy in the radial class.

\subsection{Hierarchy of Free Energies}
First of all, we find that the constant solution $(\bar u,\bar v)$ defined in $(0,R)$ has energy
\[\mathcal E(\bar u, \bar v)=\frac{(1-\chi)\bar u}{\chi}M=-\frac{\omega^2M^2}{\chi \pi R^2};\]
moreover, the one-parameter family of bifurcation solutions (\ref{bifurcation}) have the same energy as the constant solution since $u^{(k)}_\varepsilon(r)-\chi_k v^{(k)}_\varepsilon(r)=(1-\chi_k)\bar u$ independent of $\varepsilon$
\[\mathcal E(u^{(k)}_\varepsilon(r),v^{(k)}_\varepsilon(r))=-\frac{\omega^2M^2}{\chi_k\pi R^2},k\in\mathbb N^+.\]

\begin{lemma}\label{lemma61}
For each $\chi>\chi_1$, we always have that $\mathcal E(u^-,v^-),\mathcal E(u^+,v^+)<\mathcal E(\bar u,\bar v)$.  In general, for any annulus $(a,b)$ and each $\chi>\chi_{a,b}$, $\mathcal E(\mathbb U_3,\mathbb V_3),\mathcal E(\mathbb U_4,\mathbb V_4)<\mathcal E(\bar u,\bar v)$, where $(\mathbb U_i,\mathbb V_i)$ are given by (\ref{227}) and (\ref{229}).
\end{lemma}
\begin{proof}
We first prove that $\mathcal E(u^-, v^-)<\mathcal E(\bar u, \bar v)$ for any $\chi>\chi_1$.  Note that the inner ring spike $(u^-,v^-)$ given by (\ref{interring}) has energy
\[\mathcal E(u^-, v^-)=\frac{\omega^2M^2}{\chi\pi} \cdot\frac{\omega J_0(\omega r_1)}{2r_1J_1(\omega r_1)-\omega r_1^2J_0(\omega r_1)}.\]
To show the inequality above, it suffices to prove $2r_1J_1(\omega r_1)+\omega (R^2-r_1^2)J_0(\omega r_1)<0$.  Because $r_1$ is the root of (\ref{23}), this inequality is equivalent to $2r_1 T_1(r_1;R)+(R^2-r_1^2)T_0(r_1;R)>0$.  By straightforward calculations that $\partial_r\Big(2rT_1(r;R)+(R^2-r^2)T_0(r;R)\Big)=(R^2-r^2)T_1(r;R)<0$ for $r<R$, hence $(2r_1T_1(r_1;R)+(R^2-r_1^2)T_0(r_1;R))>2RT_1(R;R)=0$ and $\mathcal E(u^-,v^-)-\mathcal E(\bar{u},\bar{v}) <0$ as expected.

The outer ring solution $(u^+,v^+)$ given by (\ref{outerring}) has energy
\[\mathcal E(u^+, v^+)=-\frac{\omega^2M^2}{\chi\pi} \cdot\frac{\omega S_0(r_2;\omega,R)}{2(R-r_2)S_1(r_2;\omega,R)+\omega r_2 (2R-r_2)S_0(r_2;\omega,R)}.\]
We claim that $\mathcal{E}(u^+,v^+)<\mathcal{E}(\bar{u},\bar{v})$.  To see this we only show that $2S_1(r_2;\omega,R)-\omega(R-r_2)S_0(r_2;\omega,R)<0$, which is equivalent to $2I_1(R-r_2)-(R-r_2)I_0(R-r_2)<0$ thanks to (\ref{210}).  However, this is an immediate consequence of the recurrence relations of the modified Bessel functions (e.g.,\cite{Abramowitz})
\[2I_1(R-r_2)-(R-r_2)I_0(R-r_2)=-(R-r_2)I_2(R-r_2)<0,\]
and this readily implies $\mathcal E(u^+,v^+)-\mathcal E(\bar{u},\bar{v})<0$.

Next we study the energies of the monotone solutions $(\mathbb U_3,\mathbb V_3)$ in (\ref{227}) and $(\mathbb U_4,\mathbb V_4)$ in (\ref{229}) in an annulus $(a,b)$.  To begin with, we know that energy of the constant solution is
\[\mathcal E(\bar u_{ab}, \bar v_{ab})=\frac{(1-\chi)\bar u_{ab}}{\chi}M=-\frac{\omega^2 M^2}{\chi \pi (b^2-a^2)},\]
where $(\bar u_{ab},\bar v_{ab})=\frac{M}{\pi(b^2-a^2)}$.  Thanks to (\ref{227}) we find
\[\mathcal E(\mathbb U_3, \mathbb V_3)=\frac{\omega^2M^2}{\chi\pi} \cdot\frac{\omega S_0(-r_3;\omega,a)}{2(a+r_3)S_1(-r_3;\omega,a)-\omega r_3 (2a+r_3)S_0(-r_3;\omega,a)}.\]
We claim that $\mathcal E(\mathbb U_3, \mathbb V_3)<\mathcal E(\bar u_{ab}, \bar v_{ab})$.  To this end, we only need to show that $2(a+r_3)T_1(a+r_3;b)+\big(b^2-(a+r_3)^2 \big)T_0(a+r_3;b)>0$ in light of (\ref{228}).  Denote
\[y_{11}(r;a,b):=2(a+r)T_1(a+r;b)+\left(b^2-(a+r)^2 \right)T_0(a+r;b),\]
then straightforward calculations give that $\partial_ry_{11}(r;a,b)=\left(b^2-(a+r)^2 \right)T_1(a+r;b)<0$, therefore $y_{11}(r_3;a,b)>y_{11}(b-a;a,b)=0$ as expected, hence $\mathcal E(\mathbb U_3,\mathbb V_3)-\mathcal E(\bar{u}_{ab},\bar{v}_{ab})<0$.

To study the monotone increasing solutions $(\mathbb U_4,\mathbb V_4)$, we first find
\[\mathcal E(\mathbb U_4, \mathbb V_4)=-\frac{\omega^2M^2}{\chi\pi} \cdot\frac{\omega S_0(r_4;\omega,b)}{2(b-r_4)S_1(r_4;\omega,b)+\omega r_4 (2b-r_4)S_0(r_4;\omega,b)}.\]
Denote
\[y_{12}(r;a,b):=2(b-r)T_1(b-r;a)-\left((b-r)^2-a^2 \right)T_0(b-r;a),\]
then a routine computation gives that $\partial_ry_{12}(r;a,b)=\left((b-r)^2-a^2 \right)T_1(b-r;a)>0$ for $r\in(0,b-a)$, therefore $y_{12}(r_4;a,b)<y_{12}(b-a;a,b)=0$ and this implies through (\ref{231}) that
 \[2(b-r_4)S_1(r_4;\omega,b)-\omega\left((b-r_4)^2-a^2 \right)S_0(r_4;\omega,b)<0.\]
Therefore we have $\mathcal E(\mathbb U_4,\mathbb V_4)-\mathcal E(\bar{u}_{ab},\bar{v}_{ab})<0$.  This completes the proof.
\end{proof}

\begin{figure}[h!]
\hspace{-2mm}\begin{minipage}{0.5\columnwidth}
\includegraphics[width=\columnwidth,height=6cm]{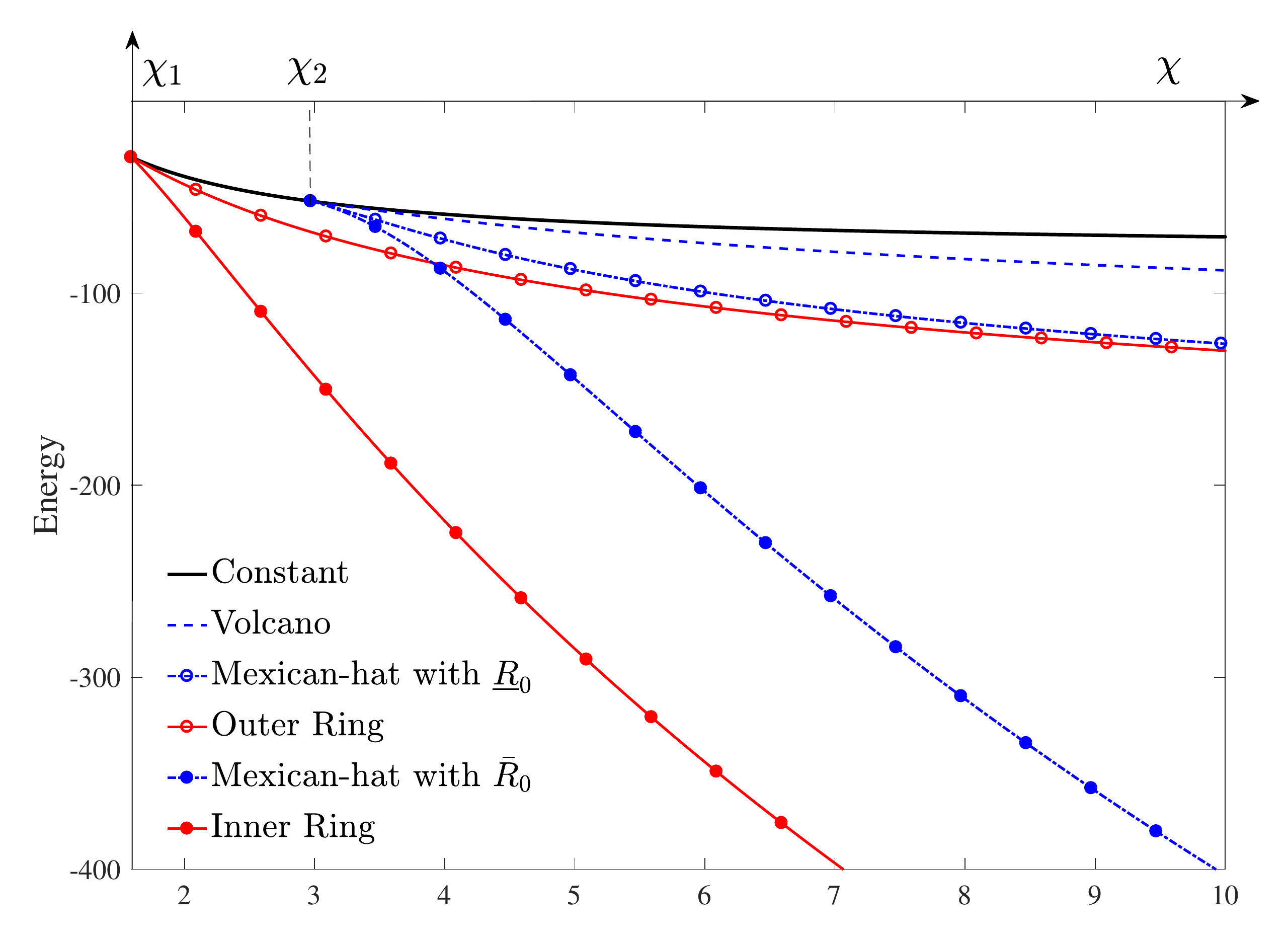}
\vspace{-10mm}\caption*{(1) Energies of the steady states for $\chi$ small.}
\includegraphics[width=\columnwidth,height=6cm]{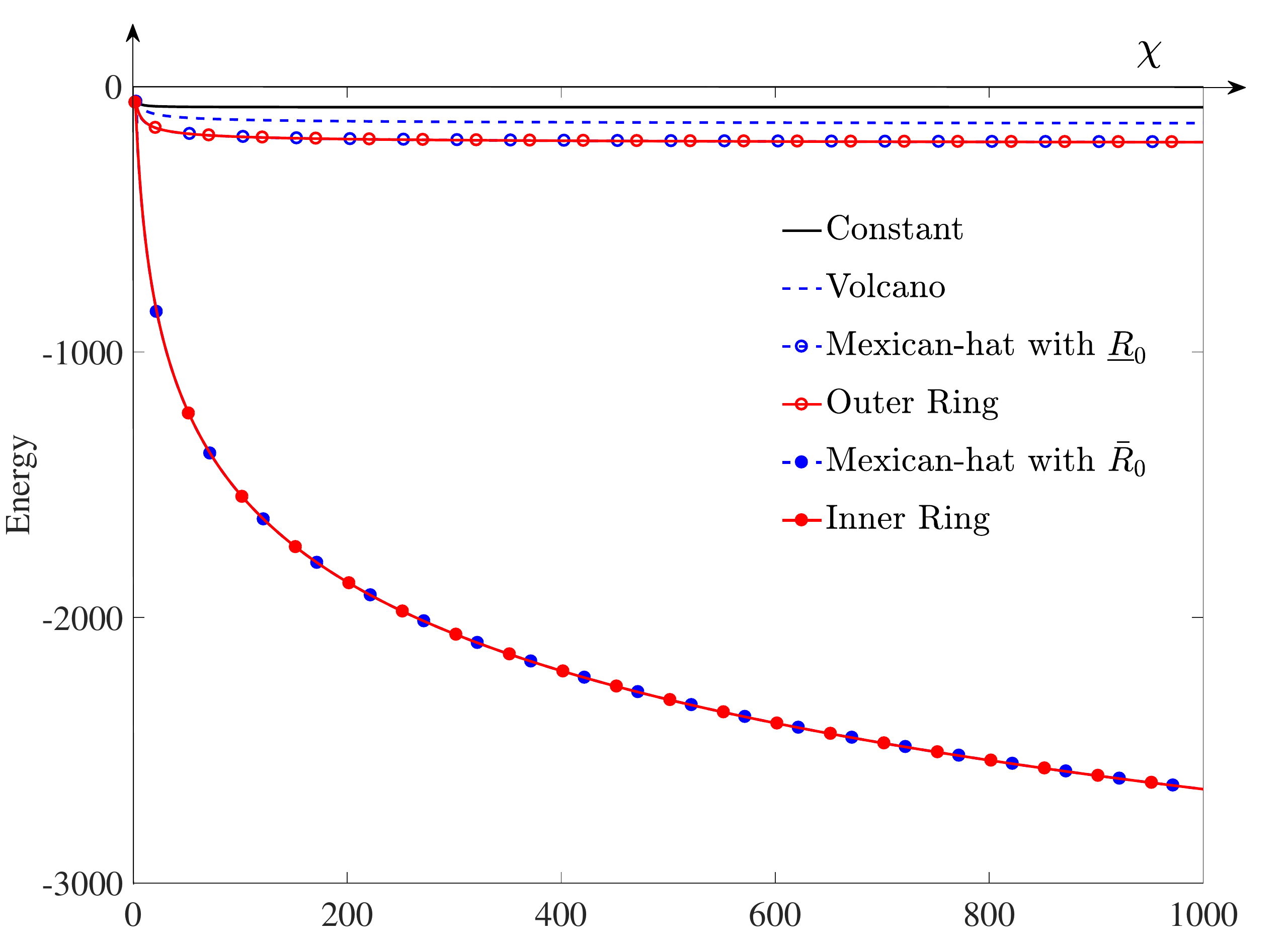}
\vspace{-7.5mm}\caption*{(2) Energies of the steady states for $\chi$ large.}
\end{minipage}
\begin{minipage}{0.5\columnwidth}
\includegraphics[width=\columnwidth,height=12cm]{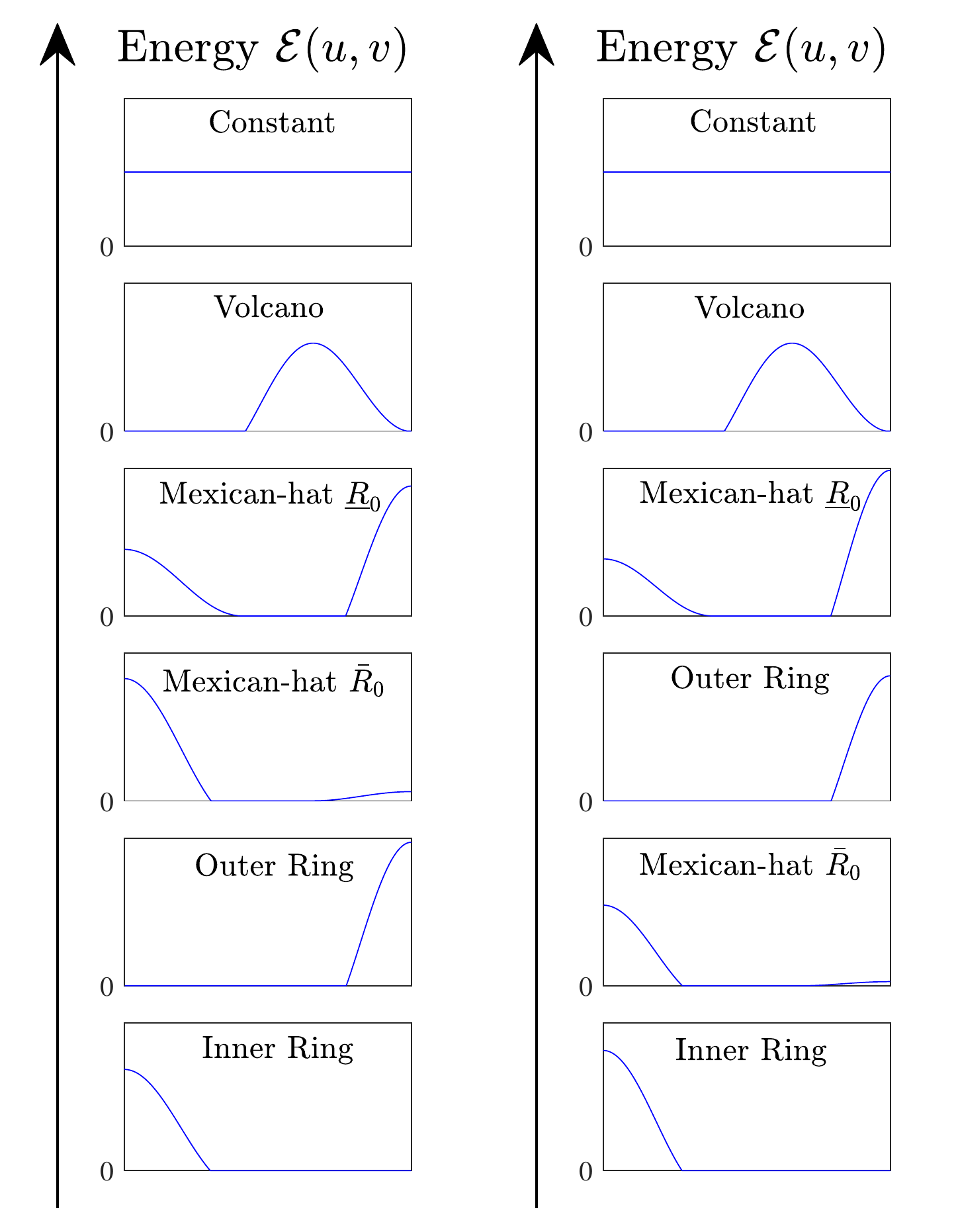}
\vspace{-7.5mm}\caption*{ (a) $\chi<3.98$ \quad \quad  \quad  \quad  \quad  \quad (b) $\chi>3.98$}\vspace{-1mm}
\end{minipage}
\vspace{-1mm}\caption{The hierarchy of energy (\ref{freeenergy}) for the constant, monotone ring, and non-monotone ring solutions on the $\chi$-axis.  Our results indicate that the inner ring solution (\ref{interring}) has the least energy and the constant solution has the largest energy among all radial solutions.  \textbf{Left Column}: (1) the local curves of the energy for these solutions for $\chi$ small.  When $\chi=\chi_1$, the constant solution has the same energy as the bifurcation solutions (\ref{bifurcation}) for any $\varepsilon$; for each $\chi>\chi_1$, the energy of the inner ring is always smaller than that of the outer ring.  Indeed, the inner ring is the global minimizer of energy (\ref{freeenergy}) in the radial class.  The Mexican-hat and Volcano solution emerge as $\chi$ surpasses $\chi_2$, and both have an energy smaller than the constant. (2): the global curves of these energies for $\chi$ large.  They suggest that in the large limit of chemotaxis rate, the Mexican-hat solutions with a fixed $R_0=\underline R_0$ and $R_0=\bar R_0$ admit energies which converge to those of the outer ring and inner ring solutions, respectively.  However, we do not illustrate this (with full details) here for the sake of succinctness.  \textbf{Right Column:} the hierarchy of energy listed on the left column and an illustration of the diagram.  The inner ring has the least energy and the constant has the largest energy, while exists a critical value $3.98$ that outer ring has a smaller energy than the Mexican-hat with $R_0=\bar R_0$ if $\chi<3.98$ and a larger energy if $\chi>3.98$.}\label{energy}
\end{figure}

Now we are ready to prove the Theorem \ref{theorem13}.
\begin{proof} [Proof\nopunct] \emph{of Theorem} \ref{theorem13}.   Part \emph{(i)} has been proved in Lemma \ref{lemma61} and we now prove \emph{(ii)} and \emph{(iii)}.  To show \emph{(ii)}, we find that
\[\mathcal E(u_i,v_i)=- \frac{M^2\omega ^3 S_0(r_2;\omega,R^*_0)/(\pi \chi)}{2(R^*_0-r_2)S_1(r_2;\omega,R^*_0)-2(R^*_0+r_3)S_1(-r_3;\omega,R^*_0)+\omega(r_3^2-r_2^2+2R_0^*(r_2+r_3)) S_0(r_2;\omega,R^*_0)}\]
hence
\begin{equation}\label{inequality}
\mathcal E(u_i,v_i)-\mathcal E(\bar u,\bar v)=-\frac{M^2\omega ^2}{\pi \chi}\left(\frac{1}{h_1(r_2;\omega,R_0^*)+h_2(r_2;\omega,R_0^*)}-\frac{1}{R^2}\right),
\end{equation}
where the functions $h_i$ are given by $h_1(r;\omega,R_0^*):=2(R_0^*-r)\frac{I_1(R_0^*-r)}{I_0(R_0^*-r)}+r(2R_0^*-r)$ and $h_2(r;\omega,R_0^*):=-2(R_0^*+r)\frac{T_1(R_0^*+r;R)}{T_0(R_0^*+r;R)}+r(2R_0^*+r)$.  One can find that $\partial_r h_1=\frac{2(R_0^*-r)I^2_1(R_0^*-r)}{I^2_0(R_0^*-r)}>0$ and $\partial_r h_2=\frac{2(R_0^*+r)T^2_1(R_0^*+r;R)}{T^2_0(R_0^*+r;R)}>0$, therefore we have
\[h_1(r_2;\omega,R_0^*)<h_1(R^*_0;\omega,R_0^*)=(R_0^*)^2 \text{~and~}h_2(R-R_0^*;\omega,R_0^*)<h_1(R^*_0;\omega,R_0^*)=R^2-(R_0^*)^2.\]
This readily implies from (\ref{inequality}) that $\mathcal E(u_i,v_i)<\mathcal E(\bar u,\bar v)$ as expected.  The fact that $\mathcal E(u_d,v_d)<\mathcal E(\bar u,\bar v)$ can be verified by the same calculations, and we skip the details here.

To find the asymptotic energies for $\chi\gg 1$, we calculate
\begin{align}
\mathcal E(u^-,v^-)=&\frac{M^2\omega ^2}{\pi \chi} \frac{\omega J_0(\omega r_1)}{2r_1 J_1(\omega r_1)-\omega r_1^2J_0(\omega r_1)} \notag\\
                   =& \frac{M^2\omega ^2}{\pi \chi} \frac{T_0(r_1;R)}{2r_1 T_1(r_1;R)- r_1^2T_0(r_1;R)}  \quad \quad \text{since~$r_1$~is a root of (\ref{23})} \notag\\
                   =& \frac{M^2\omega ^2}{\pi \chi} \frac{-I_1(R)\ln\frac{r_1}{2}}{-2I_1(R)+r_1^2I_1(R)\ln \frac{r_1}{2}}+O(1) \quad \quad \text{since~$r_1\rightarrow \frac{j_{0,1}}{\omega}$ as $\chi\rightarrow \infty$} \notag\\
=&-\frac{M^2\ln \omega}{2\pi}+O(1), \notag
\end{align}
which is as expected; on the other hand, we compute for the Mexican-hat with $R_0=\bar R_0$
\begin{align*}
\mathcal E(u_d,v_d)|_{R_0=\bar R_0}=&\frac{m_1^2\omega ^2}{\pi \chi} \frac{\omega J_0(\omega r_1)}{2r_1 J_1(\omega r_1)-\omega r_1^2J_0(\omega r_1)}-\frac{m_4^2\omega^2}{\pi\chi(R^2-\bar R_0^2)} \notag\\
= &\frac{M^2\omega ^2}{\pi \chi} \frac{\omega J_0(\omega r_1)}{2r_1 J_1(\omega r_1)-\omega r_1^2J_0(\omega r_1)}+o(1) \quad\quad \text{since $m_1\rightarrow M$ and $m_4\rightarrow 0$ as $\chi\rightarrow \infty$}\notag\\
= &\mathcal E(u^-,v^-)+o(1)
\end{align*}
also as expected; moreover, the asymptotic expansions for $\mathcal E(u^+,v^+)$ and $E(u_d,v_d)$ with $R_0=\underline R_0$ can be verified by the same calculations and we skip the details here.

We are left to verify \emph{(iii)}.  According to Section \ref{section2} and our discussions above, in each annulus $(a,b)$ we have $\mathcal E(\mathbb U_3,\mathbb V_3),\mathcal E(\mathbb U_4,\mathbb V_4)< \mathcal E(\bar u_{ab},\bar v_{ab})$, where $(\mathbb U_i,\mathbb V_i)$ and $(\bar u_{ab},\bar v_{ab})$ are given by Section \ref{section2}.  Let $(u_\text{cpt},v_\text{cpt})$ be an arbitrary solution of (\ref{ss}) with $u$ being compactly supported such that $v$ has $k$ critical points $\{R_i\}_{i=1}^k$ such that
\[[0,R)=\cup_{i=1}^k (R_{i-1},R_i) \text{~with~} 0=R_0<R_1<...<R_{k-1}<R_k=R.\]
Note that $k$ must be finite since $\min{(R_i-R_{i-1})}>\frac{j_{0,1}}{\omega}$.  On each interval $(R_{i-1},R_i)$, the solution takes the form $(\mathbb U_3,\mathbb V_3)$ or $(\mathbb U_4,\mathbb V_4)$.  Then according to (\ref{64}) we find that
\begin{align*}
\mathcal E(u_\text{cpt},v_\text{cpt})=&\frac{1}{\chi}\int_{B_0(R)} u(u-\chi v)d\textbf{x}   \\
=&\frac{1}{\chi}\sum_{i=1}^k \int_{B_0(R_i)\backslash B_0(R_{i-1})}u(u-\chi v)d\textbf{x}   \\
<&\sum_{i=1}^k\mathcal E(\bar u_{R_{i-1}R_i},\bar v_{R_{i-1}R_i})=\frac{1}{\chi}\sum_{i=1}^k\frac{m_i^2(1-\chi)}{\pi(R_i^2-R_{i-1}^2)}\\
=&-\frac{\omega^2}{\pi\chi}\sum_{i=1}^k\frac{m_i^2}{R_i^2-R_{i-1}^2},
\end{align*}
where $m_i$ denotes the cellular mass in the $i$-th interval (annulus) with $m_i>0$ for each $i$ and $\sum_{i=1}^k m_i=M$.  Therefore, to prove (iii), it is sufficient to prove that for any partition $\{R_i\}_{i=1}^k$
\[F(k):=\sum_{i=1}^k \frac{m_i^2}{R_i^2-R_{i-1}^2}> \frac{M^2}{R^2}.\]
First of all, this inequality holds for $k=2$.  Suppose that this is true for $k\geq2$, then for $k+1$ we have
\[F(k+1)=\sum_{i=1}^{k+1}\frac{m_i^2}{R_i^2-R_{i-1}^2}> \frac{(M-m_{k+1})^2}{R_k^2}+\frac{m^2_{k+1}}{R^2-R^2_{k}}=\frac{\big((M-m_{k+1})R^2-MR^2_{k}\big)^2}{R^2R^2_{k}(R^2-R^2_{k})}+\frac{M^2}{R^2}>\frac{M^2}{R^2},\]
which implies that $F(k+1)>\frac{M^2}{R^2}$ as expected.  Therefore $F(k)>\frac{M^2}{R^2}$ for each $k\in\mathbb N^+$, and this finishes the proof of Theorem \ref{theorem13}.
\end{proof}

\subsection{Non-radial Stationary Solutions: Numerical Evidence}
Our studies of (\ref{01}) have been restricted to its stationary problem (\ref{ss}) in the radial setting so far.  A natural question that arises is whether (\ref{01}) has non-radial steady states.  Though the rigorous analysis of this problem is out of the scope of this paper, we conduct some numerical experiments to give a confirmative answer to this question.  Our numerical simulations suggest that the spatial-temporal dynamics of (\ref{01}) are quite intricate and phase transition is ubiquitous for large $\chi$ whence its stationary problems admit many spiky solutions.
\begin{figure}[h!]
    \centering
    \begin{subfigure}[b]{0.210\textwidth}
        \includegraphics[width=\textwidth]{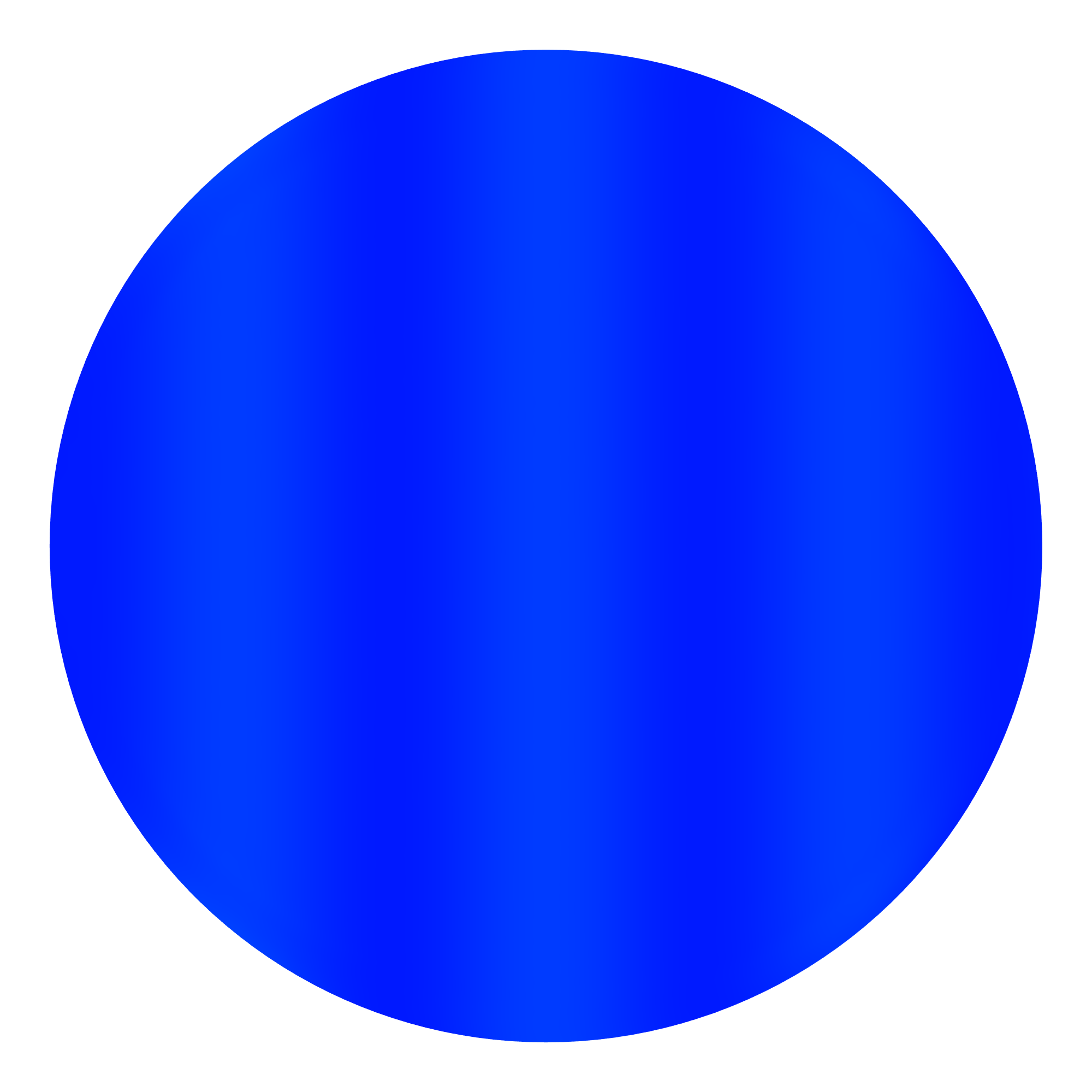}
        \caption*{$t=0$}
    \end{subfigure}\hspace{0.5mm}
    \begin{subfigure}[b]{0.210\textwidth}
        \includegraphics[width=\textwidth]{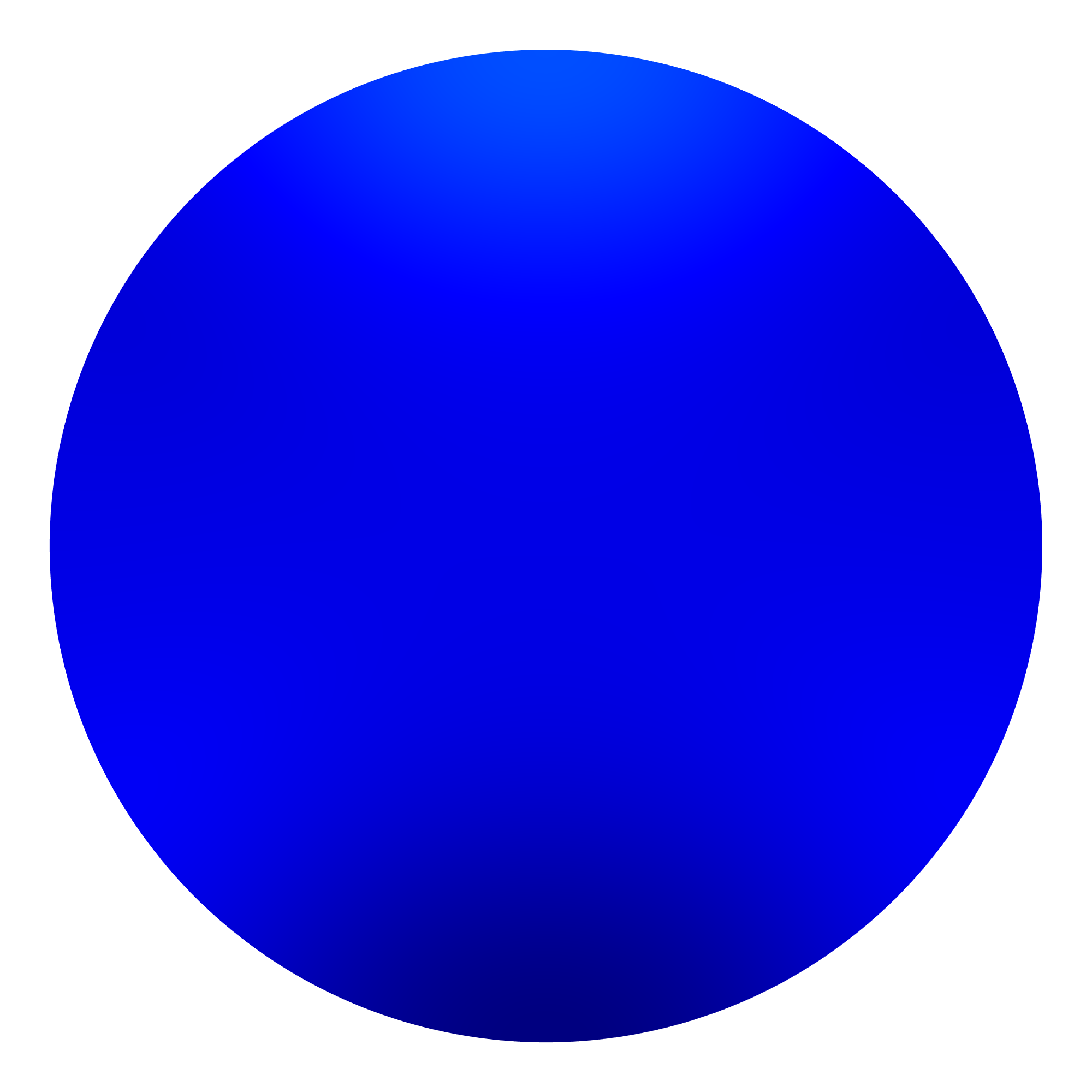}
        \caption*{$t=2.0$}
    \end{subfigure}\hspace{0.5mm}
    \begin{subfigure}[b]{0.210\textwidth}
        \includegraphics[width=\textwidth]{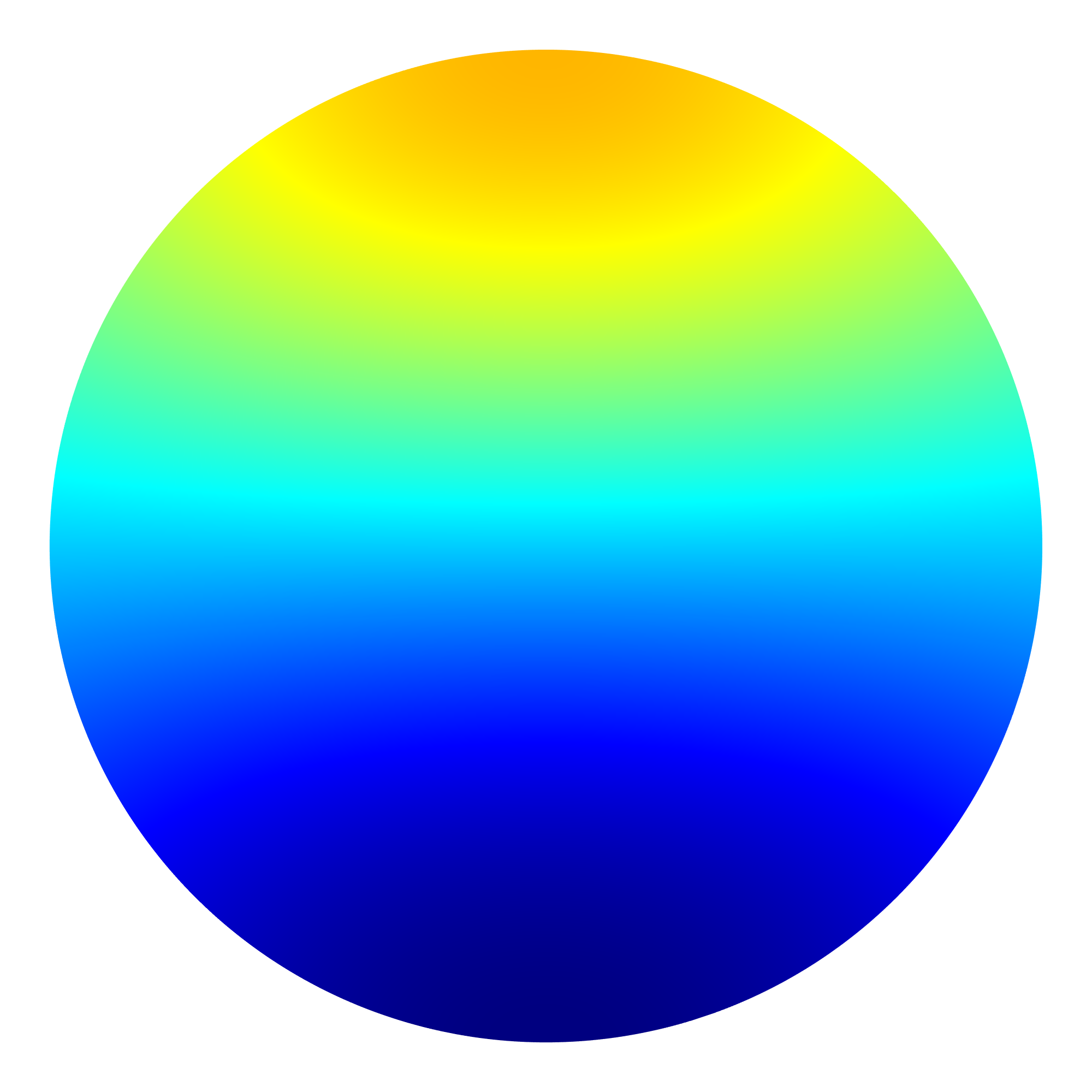}
        \caption*{$t=17$}
    \end{subfigure}\hspace{0.5mm}
    \begin{subfigure}[b]{0.210\textwidth}
        \includegraphics[width=\textwidth]{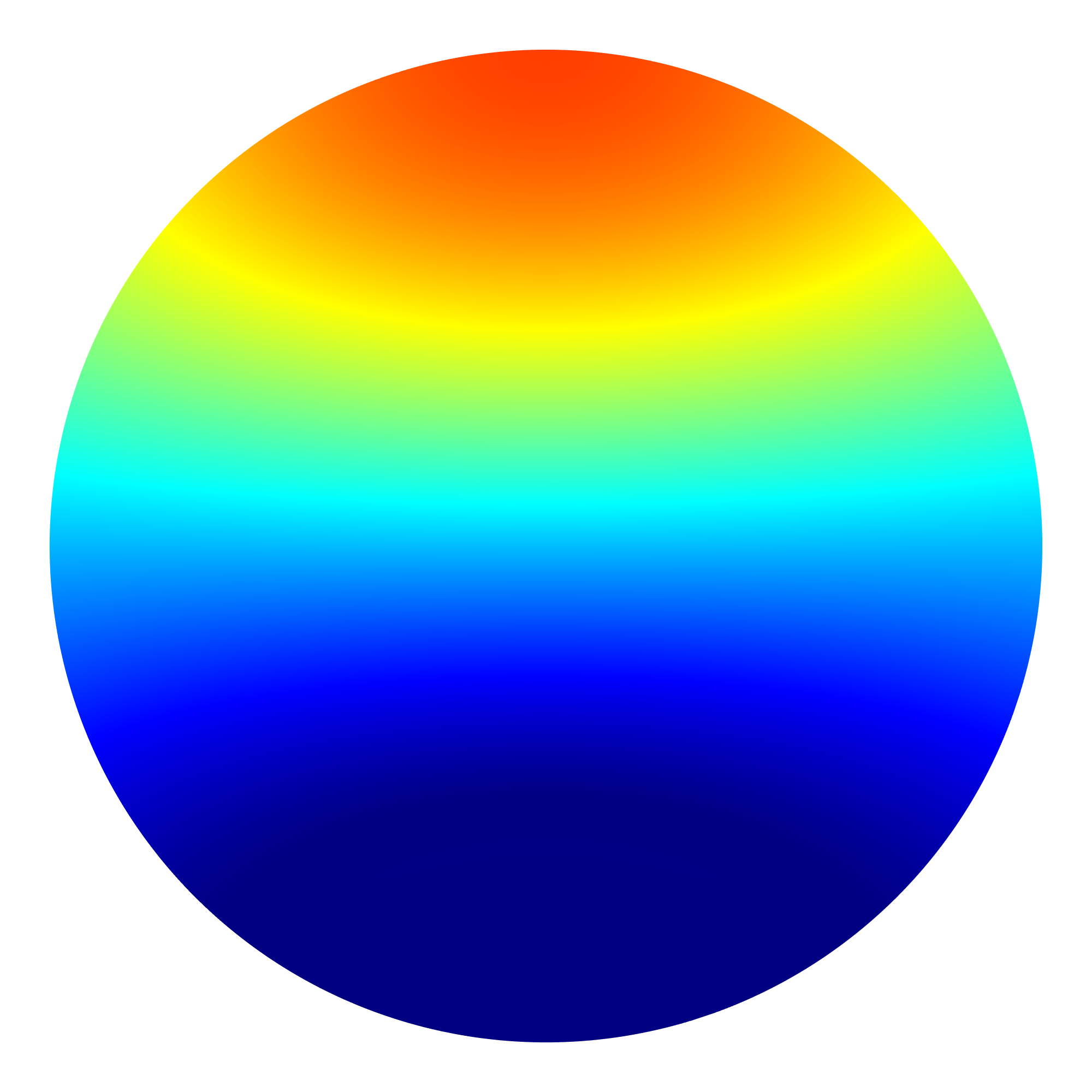}
        \caption*{$t=200$}
    \end{subfigure}\hspace{0.5mm}
    \caption{Formation of nontrivial stationary pattern in $u(\textbf{x})$ within (\ref{01}) over the unit disk out of small perturbation initial data $(u_0,v_0)=(1,1)+0.01(\cos 10 x,\sin 5y)$.  Here we choose $\chi=5>4.49$, the latter being the critical value after which the constant solution becomes unstable.  The small-amplitude initial data develop into a spatial pattern concentrated in the north pole in a large time.}\label{BS1}
\end{figure}

In the three experiments, we fix the domain to be the unit disk $B_0(1)$ and initial data $(u_0,v_0)=(1,1)+0.01(\cos 10 x,\sin 5x)$ be small perturbations of the constant pair.  Our focus is to investigate the effects of chemotaxis rate $\chi$ on the formation of nontrivial patterns, in particular those non-radially symmetric.  According to \cite{CCWWZ}, the constant solution $(\bar u,\bar v)$ is globally stable in generally if and only $\chi<\chi_1$, where $\chi_1\approx 4.39$ since $3.39$ is the (approximated) principal Neumann eigen-value.
\begin{figure}[h!]
    \centering
    \begin{subfigure}[b]{0.210\textwidth}
        \includegraphics[width=\textwidth]{ua000.png}
        \caption*{$t=0$}
    \end{subfigure}\hspace{0.5mm}
    \begin{subfigure}[b]{0.210\textwidth}
        \includegraphics[width=\textwidth]{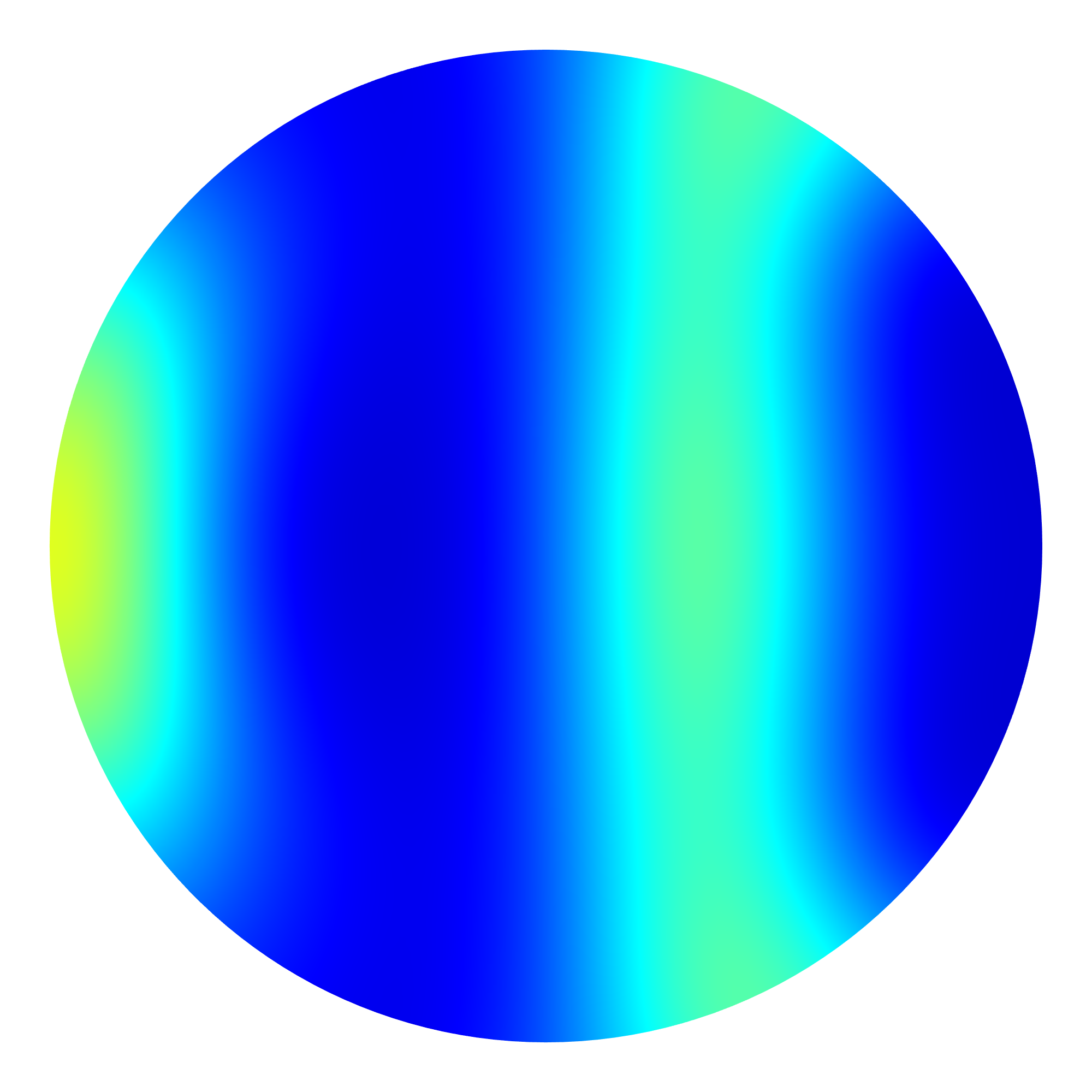}
        \caption*{$t=1.0$}
    \end{subfigure}\hspace{0.5mm}
    \begin{subfigure}[b]{0.210\textwidth}
        \includegraphics[width=\textwidth]{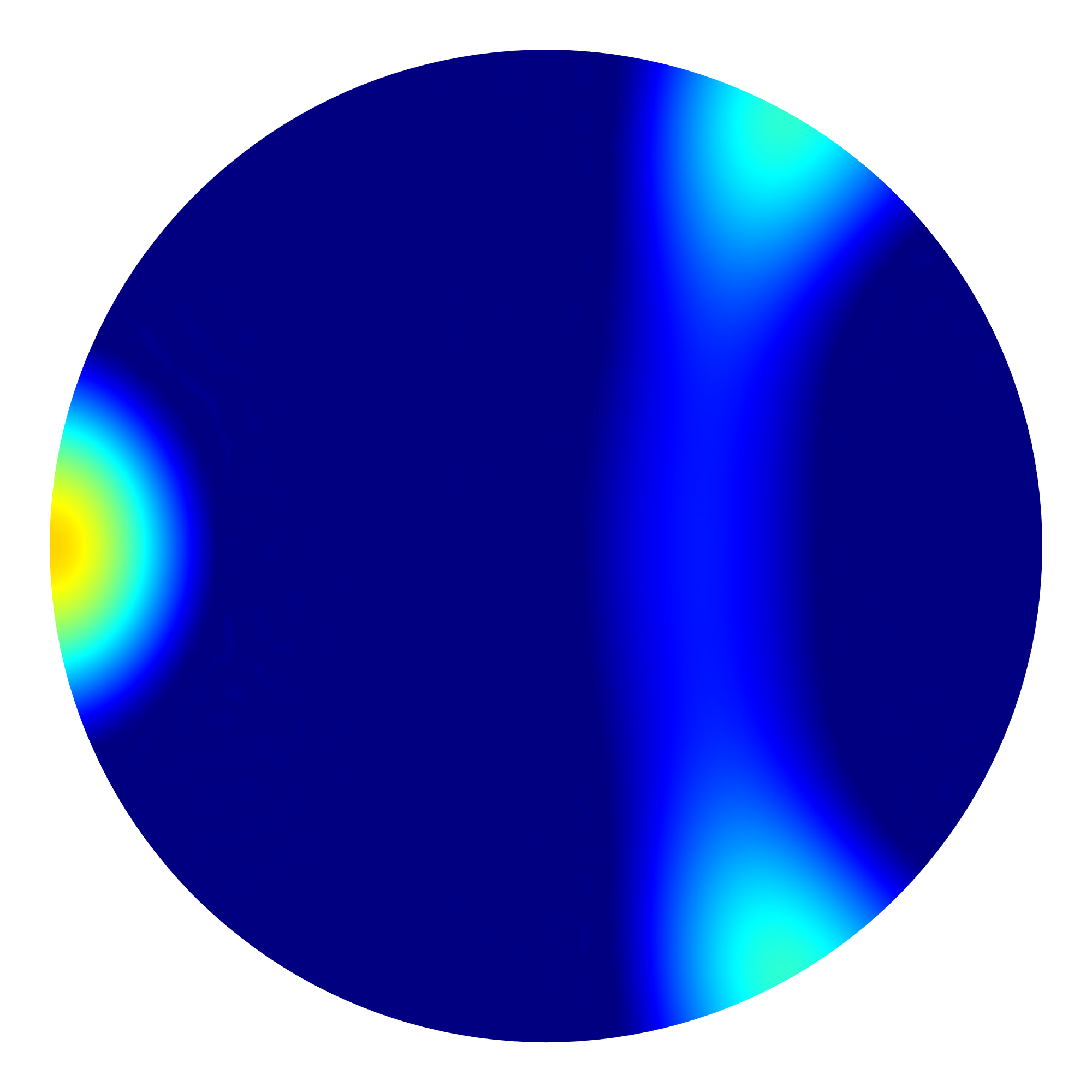}
        \caption*{$t=2.4$}
    \end{subfigure}\hspace{0.5mm}
        \begin{subfigure}[b]{0.210\textwidth}
        \includegraphics[width=\textwidth]{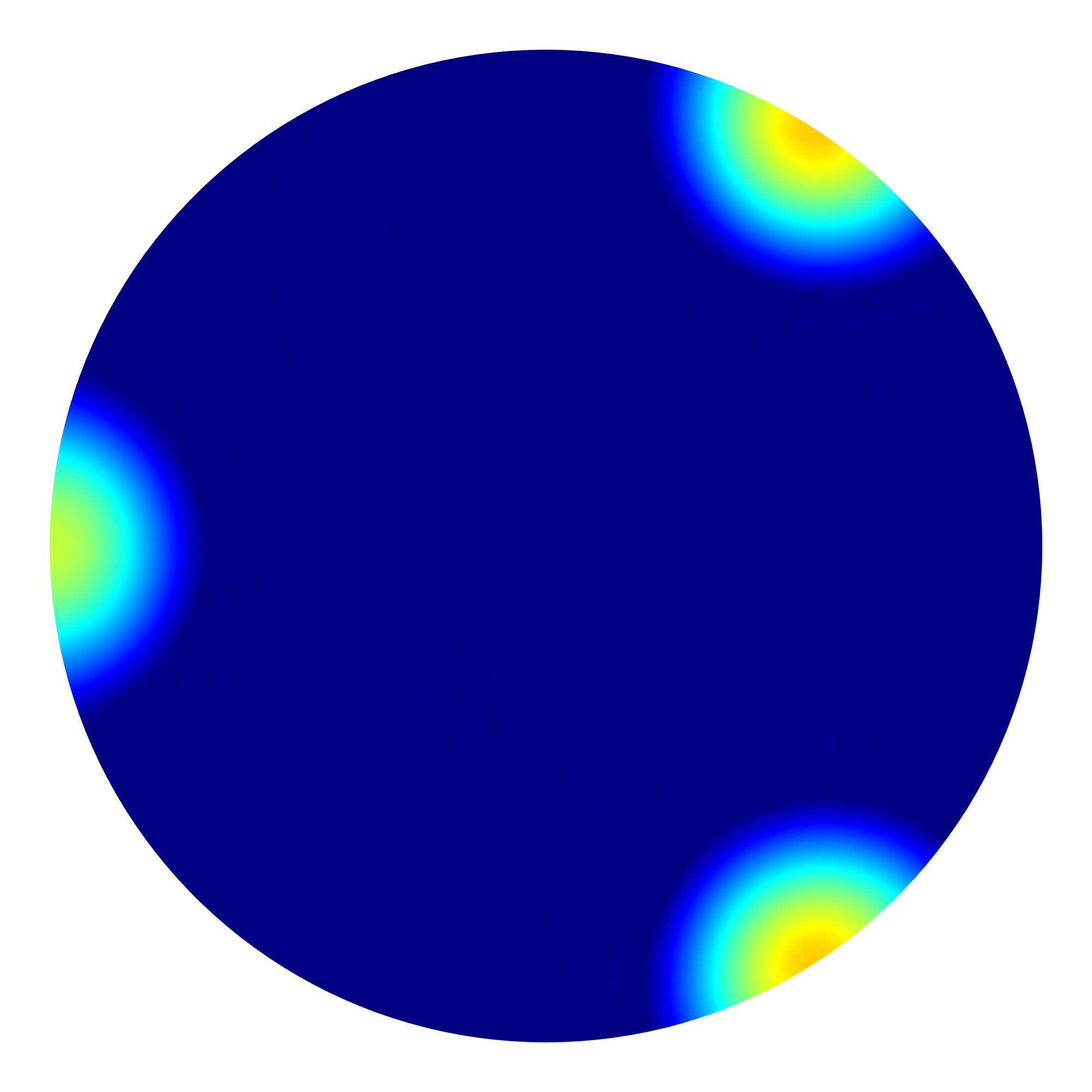}
        \caption*{$t=4.0$}
    \end{subfigure}\hspace{0.5mm}
        \begin{subfigure}[b]{0.210\textwidth}
        \includegraphics[width=\textwidth]{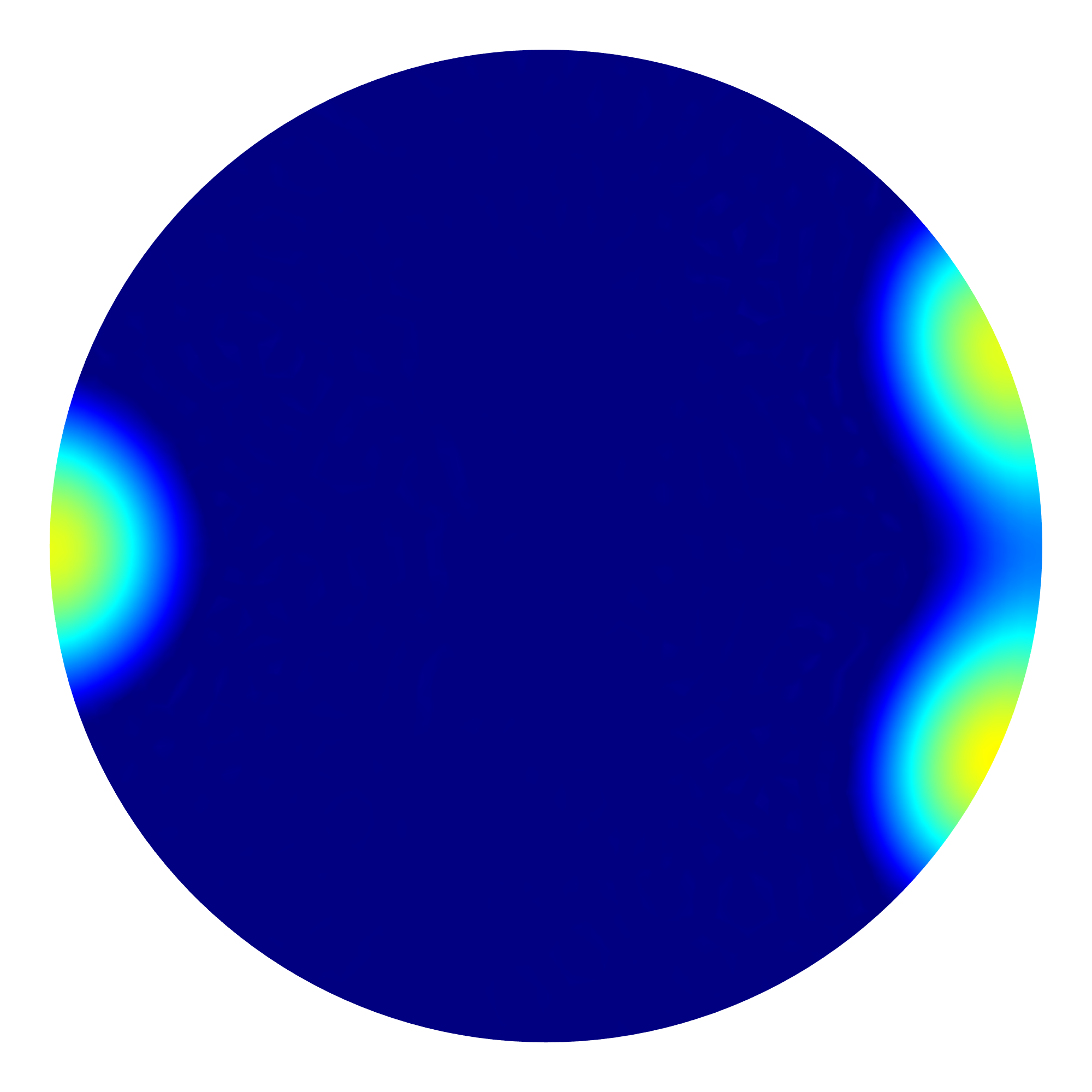}
        \caption*{$t=17.0$}
    \end{subfigure}\hspace{0.5mm}
    \begin{subfigure}[b]{0.210\textwidth}
        \includegraphics[width=\textwidth]{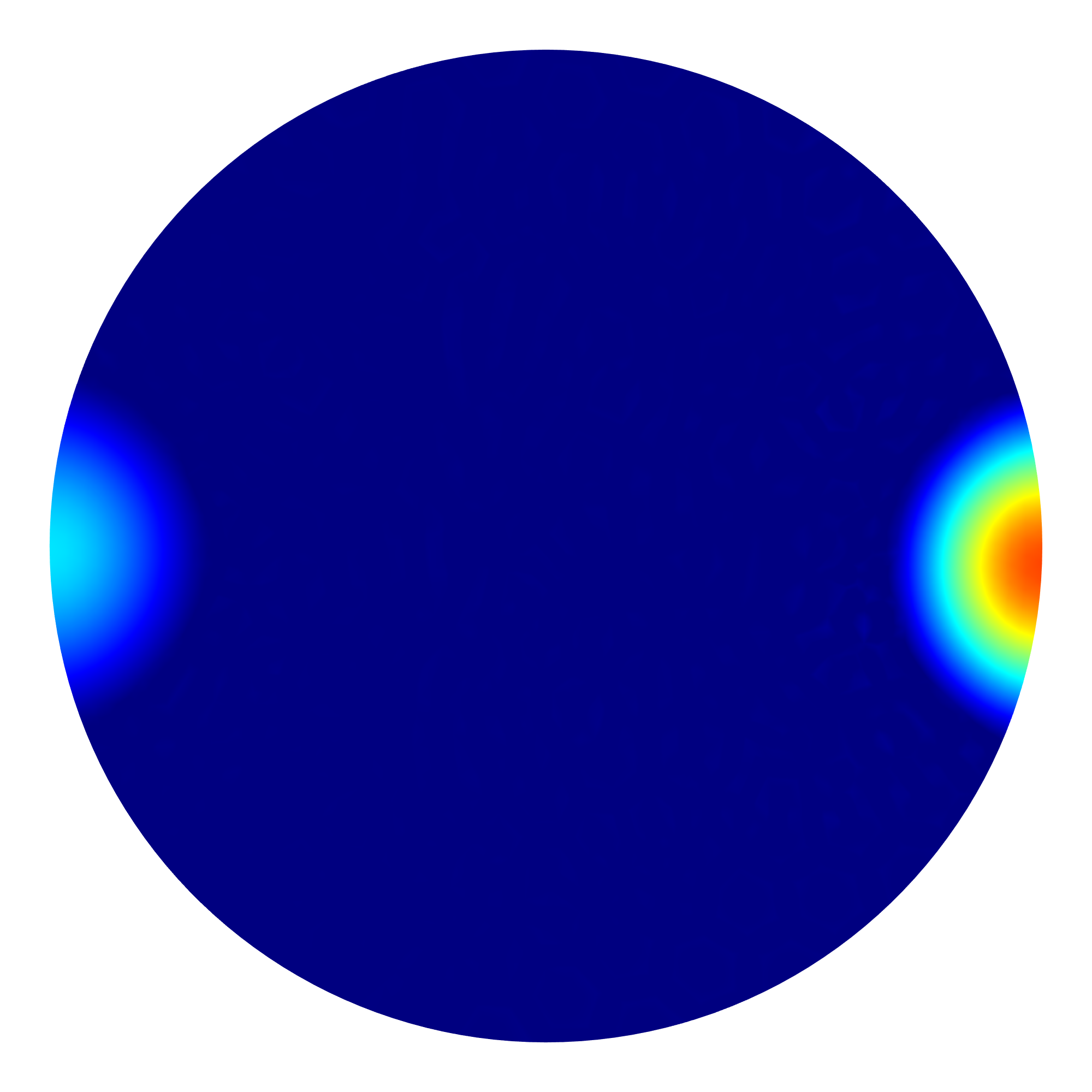}
        \caption*{$t=18.0$}
    \end{subfigure}\hspace{0.5mm}
    \begin{subfigure}[b]{0.210\textwidth}
        \includegraphics[width=\textwidth]{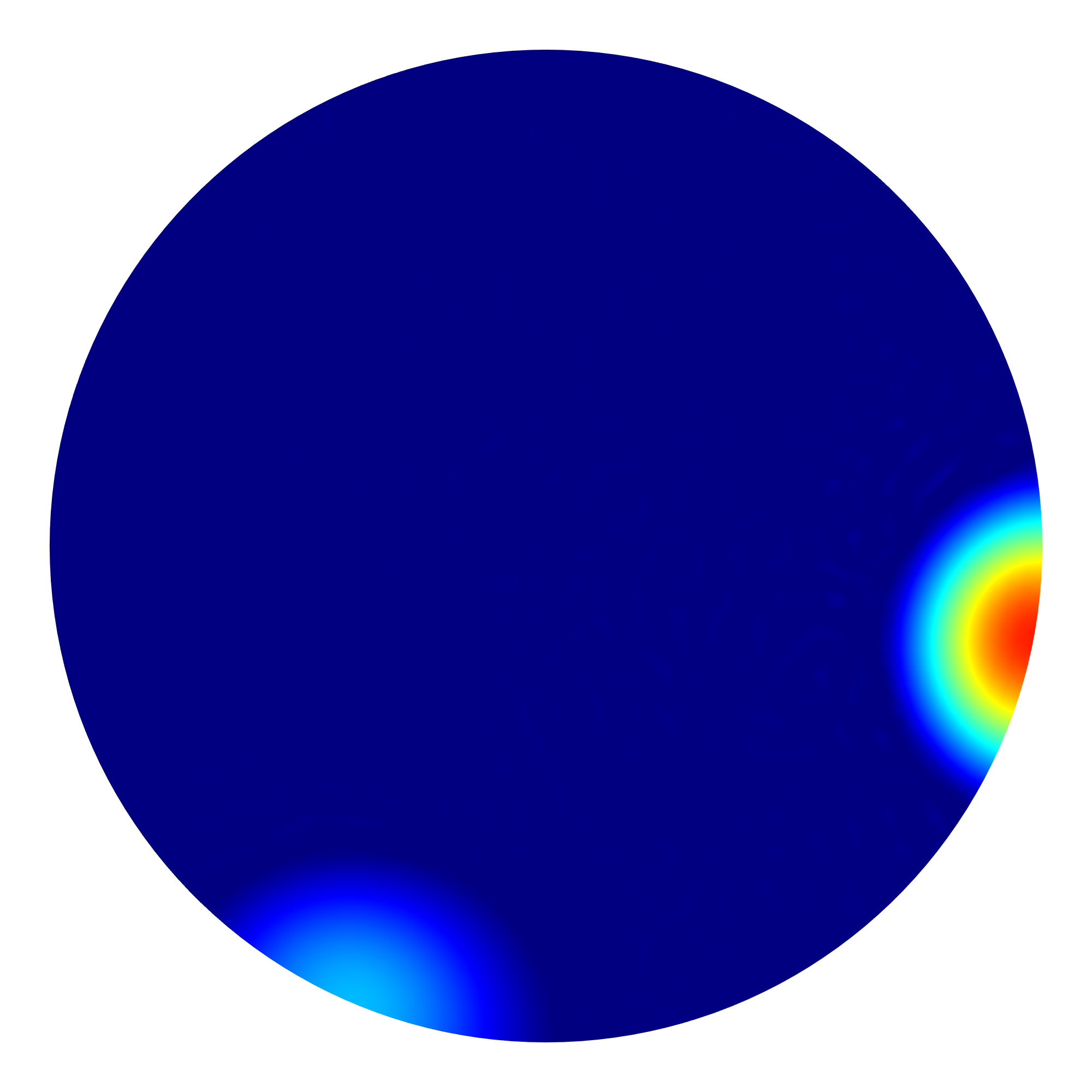}
        \caption*{$t=61.0$}
    \end{subfigure}\hspace{0.5mm}
        \begin{subfigure}[b]{0.210\textwidth}
        \includegraphics[width=\textwidth]{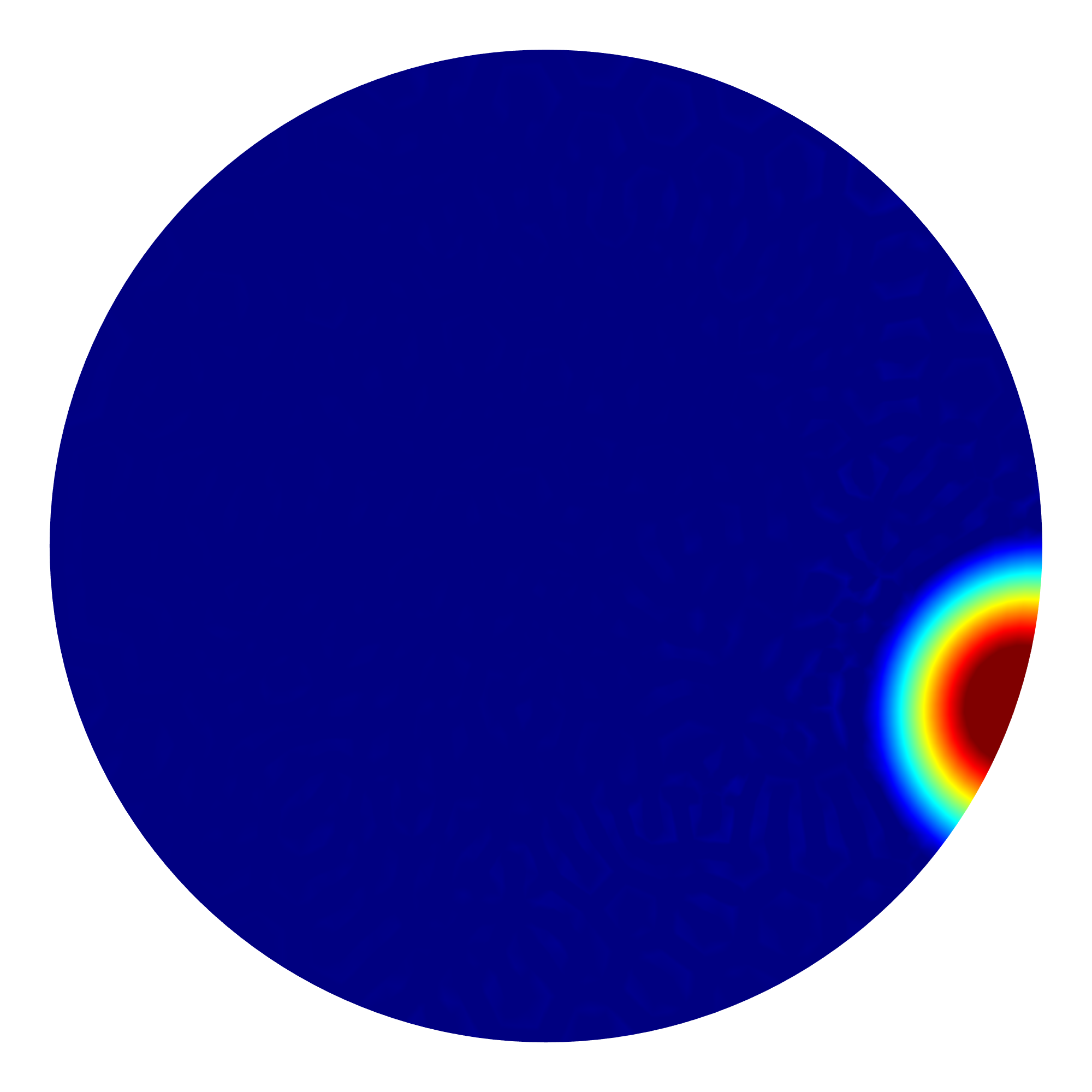}
        \caption*{$t=100$}
    \end{subfigure}\hspace{0.5mm}
    \caption{Formation of nontrivial stationary pattern in $u(\textbf{x})$ within (\ref{01}) under the same settings of Figure \ref{BS1} except that $\chi=60$.   We observe the development of a nontrivial pattern at $t\approx1.0$, and then it develops into three spikes on the boundaries, which attract each other and eventually form a stable single boundary spike in a large time.}\label{BS2}
\end{figure}
In Figure \ref{BS1}, we study the spatial-temporal dynamics of (\ref{01}) by choosing the chemotaxis rate $\chi=5$, which is slightly larger than $\chi_1=4.39$.  Then in a large time, (\ref{01}) develops a spatially nontrivial steady state which is not radial, and concentrates on the north pole in the final time.  Indeed, the only radial solution is the constant in this case according to our discussions above.

We next consider the same problem in Figure \ref{BS2} with $\chi=60$.  One observes the emergence of spatial patterns, their developments into aggregates, and the formation of three boundary spikes, all in a relatively short time.  Then these boundary spikes keep their profiles for a quite long time before two of them attract each other and merge into a new but larger aggregate at $t\approx18$, resulting in two aggregates on the boundary, the small one in the west and the large one in the east.  Then again, these two spikes endure a very long time before a phase transition happens with the small spike moving towards the large one along the boundary in a long time.  They eventually merge and form a stable aggregate located at the southeast on the boundary.

\begin{figure}[h!]\vspace{-5mm}    \centering
    \begin{subfigure}[b]{0.210\textwidth}
        \includegraphics[width=\textwidth]{ua000.png}
        \caption*{$t=0$}
    \end{subfigure}\hspace{0.5mm}
    \begin{subfigure}[b]{0.210\textwidth}
        \includegraphics[width=\textwidth]{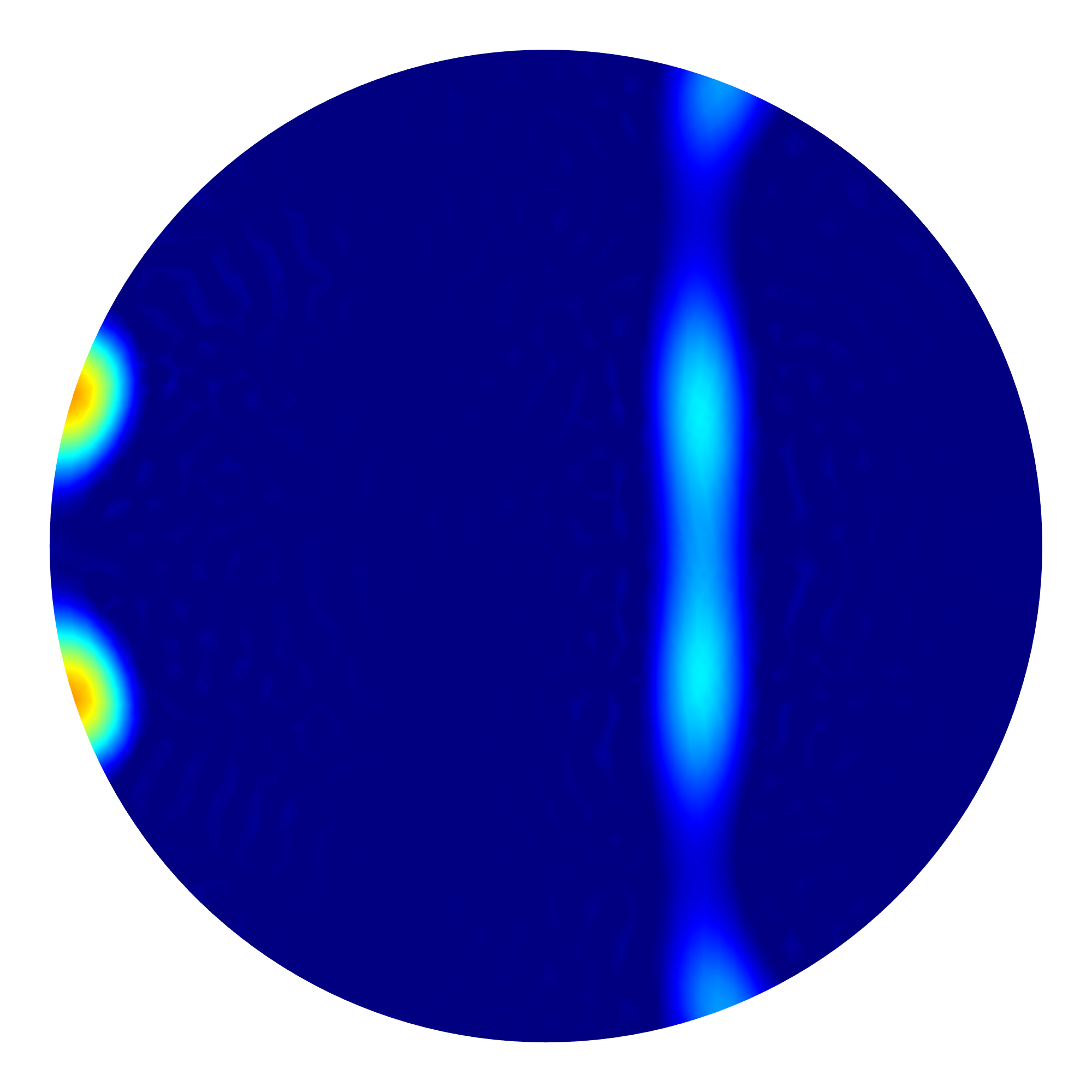}
        \caption*{$t=0.5$}
    \end{subfigure}\hspace{0.5mm}
    \begin{subfigure}[b]{0.210\textwidth}
        \includegraphics[width=\textwidth]{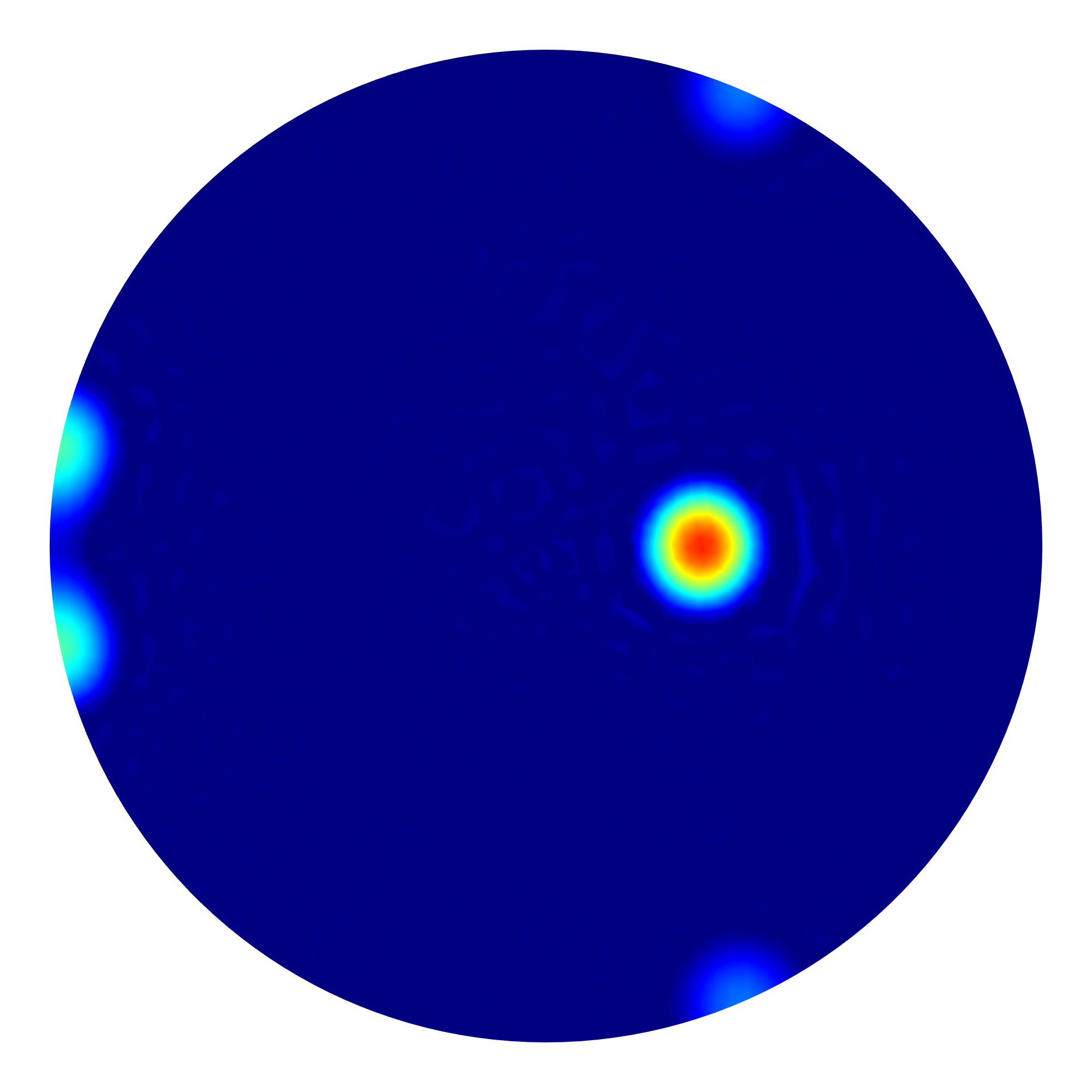}
        \caption*{$t=1.0$}
    \end{subfigure}\hspace{0.5mm}
        \begin{subfigure}[b]{0.210\textwidth}
        \includegraphics[width=\textwidth]{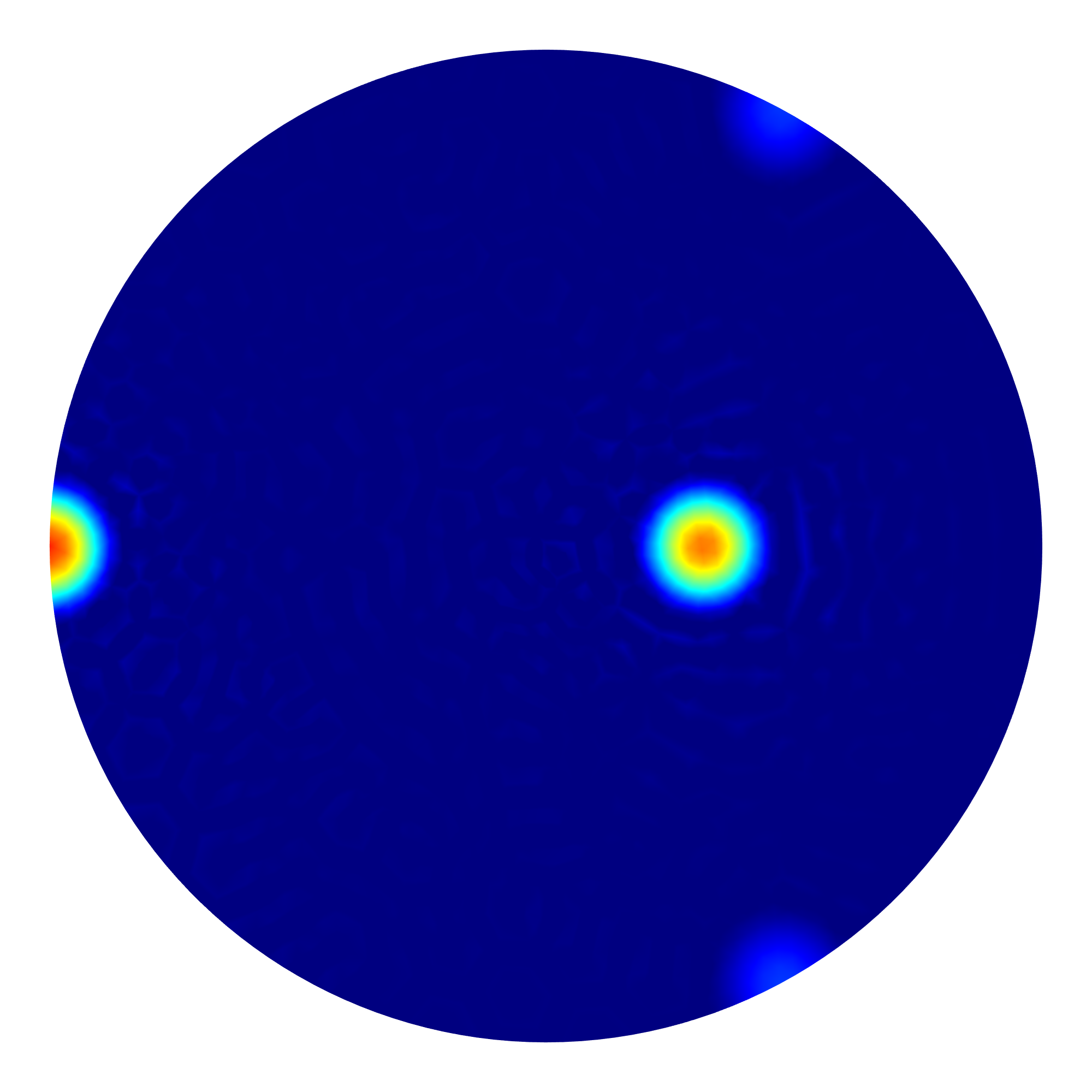}
        \caption*{$t=2.0$}
    \end{subfigure}\hspace{0.5mm}
        \begin{subfigure}[b]{0.210\textwidth}
        \includegraphics[width=\textwidth]{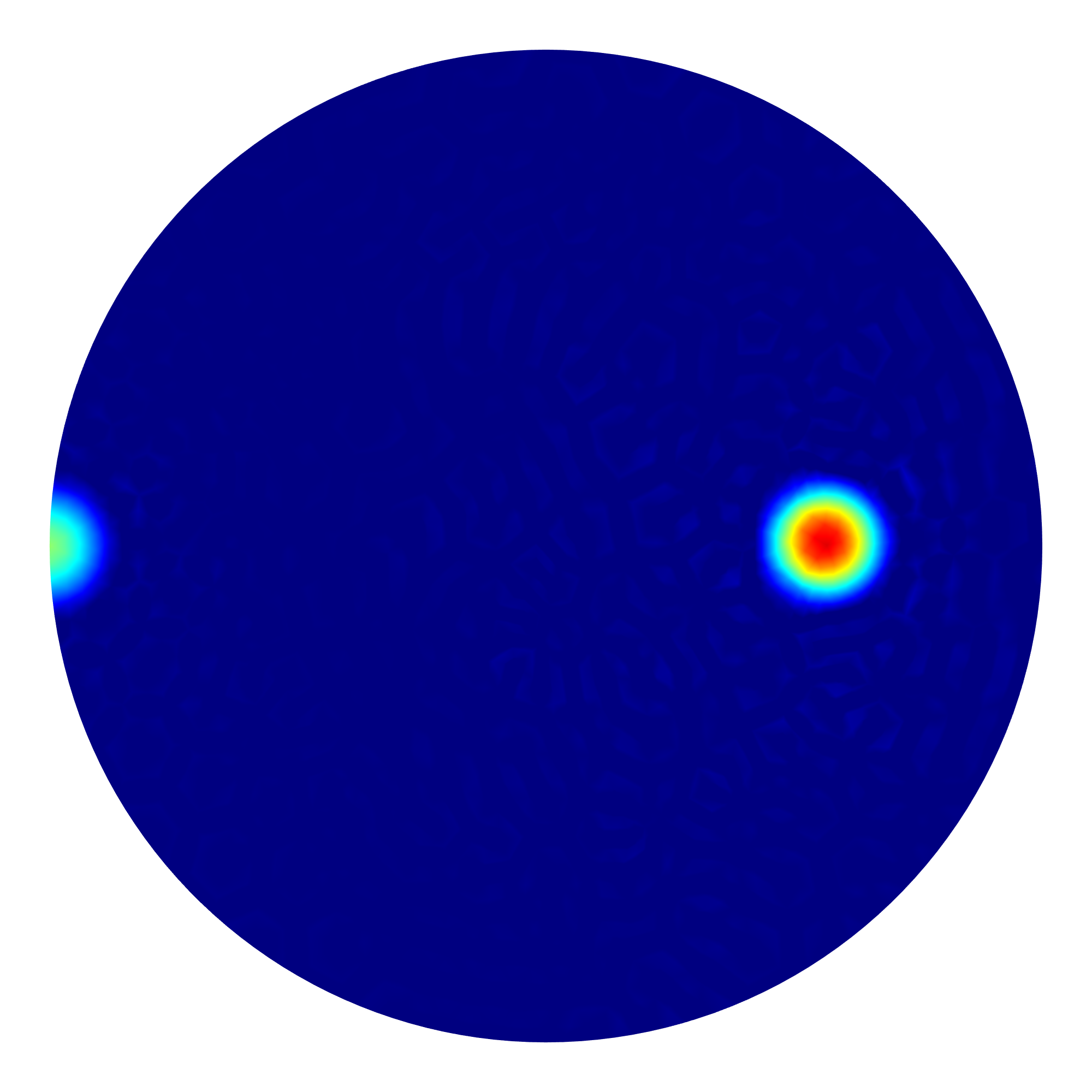}
        \caption*{$t=7.0$}
    \end{subfigure}\hspace{0.5mm}
    \begin{subfigure}[b]{0.210\textwidth}
        \includegraphics[width=\textwidth]{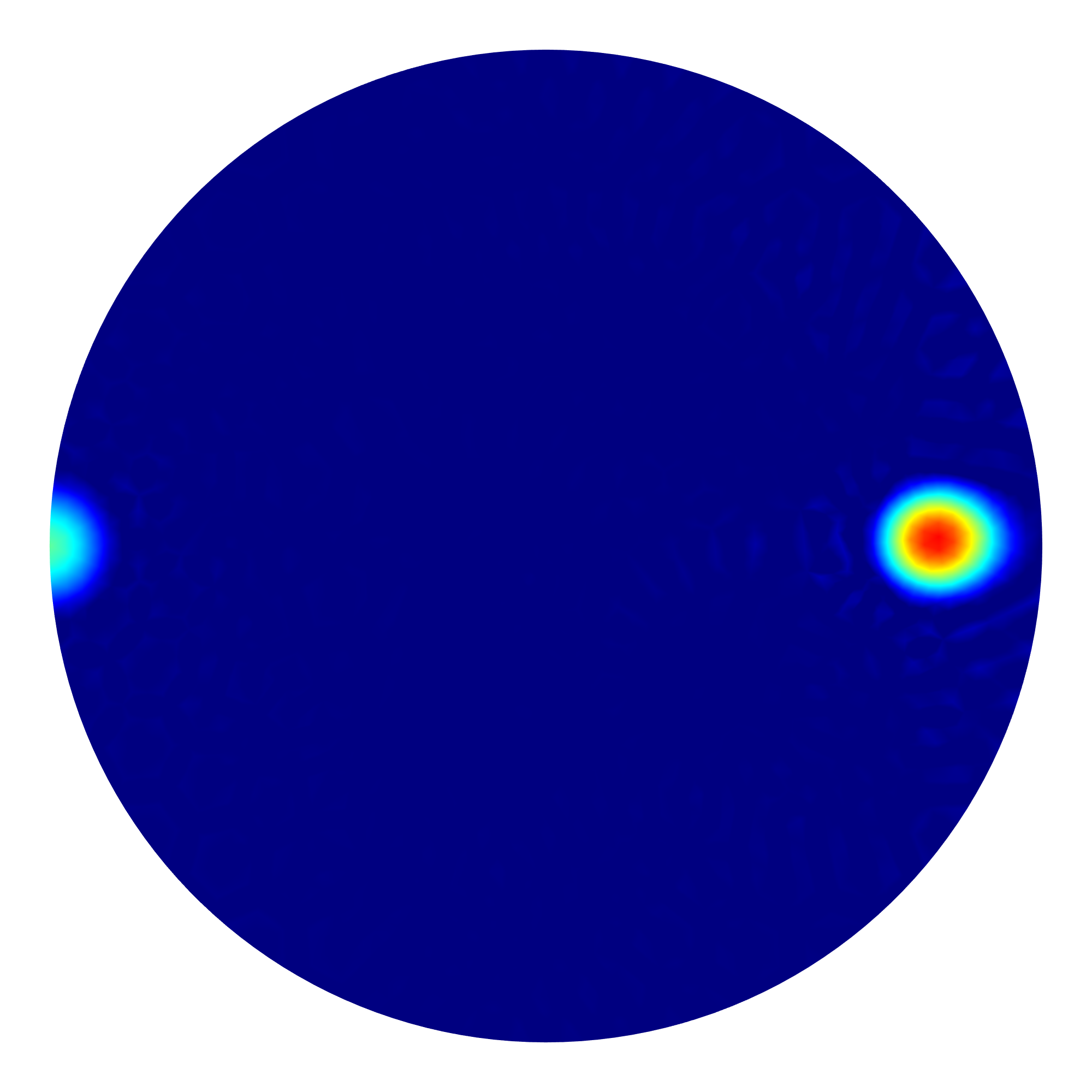}
        \caption*{$t=8.5$}
    \end{subfigure}\hspace{0.5mm}
    \begin{subfigure}[b]{0.210\textwidth}
        \includegraphics[width=\textwidth]{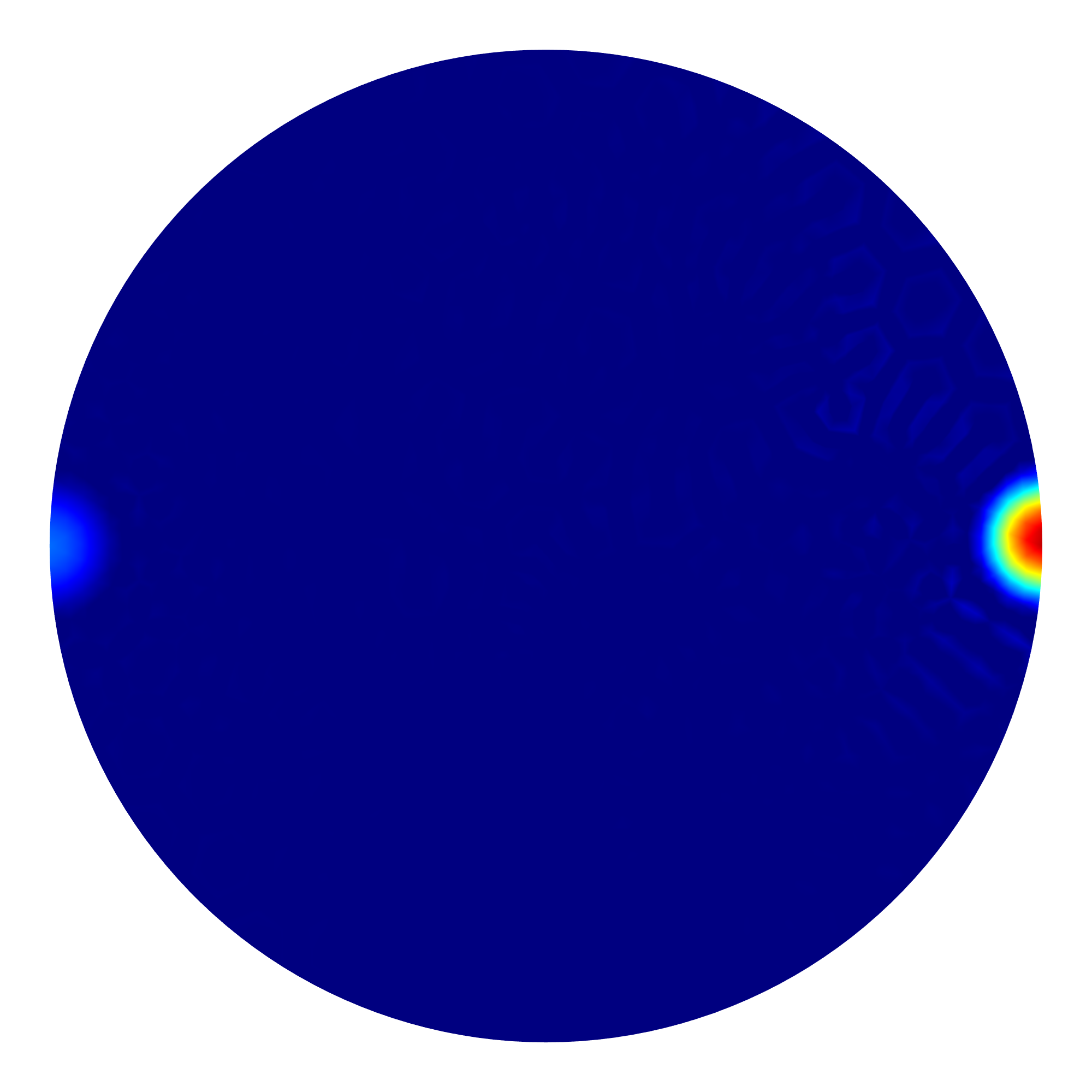}
        \caption*{$t=9.0$}
    \end{subfigure}\hspace{0.5mm}
        \begin{subfigure}[b]{0.210\textwidth}
        \includegraphics[width=\textwidth]{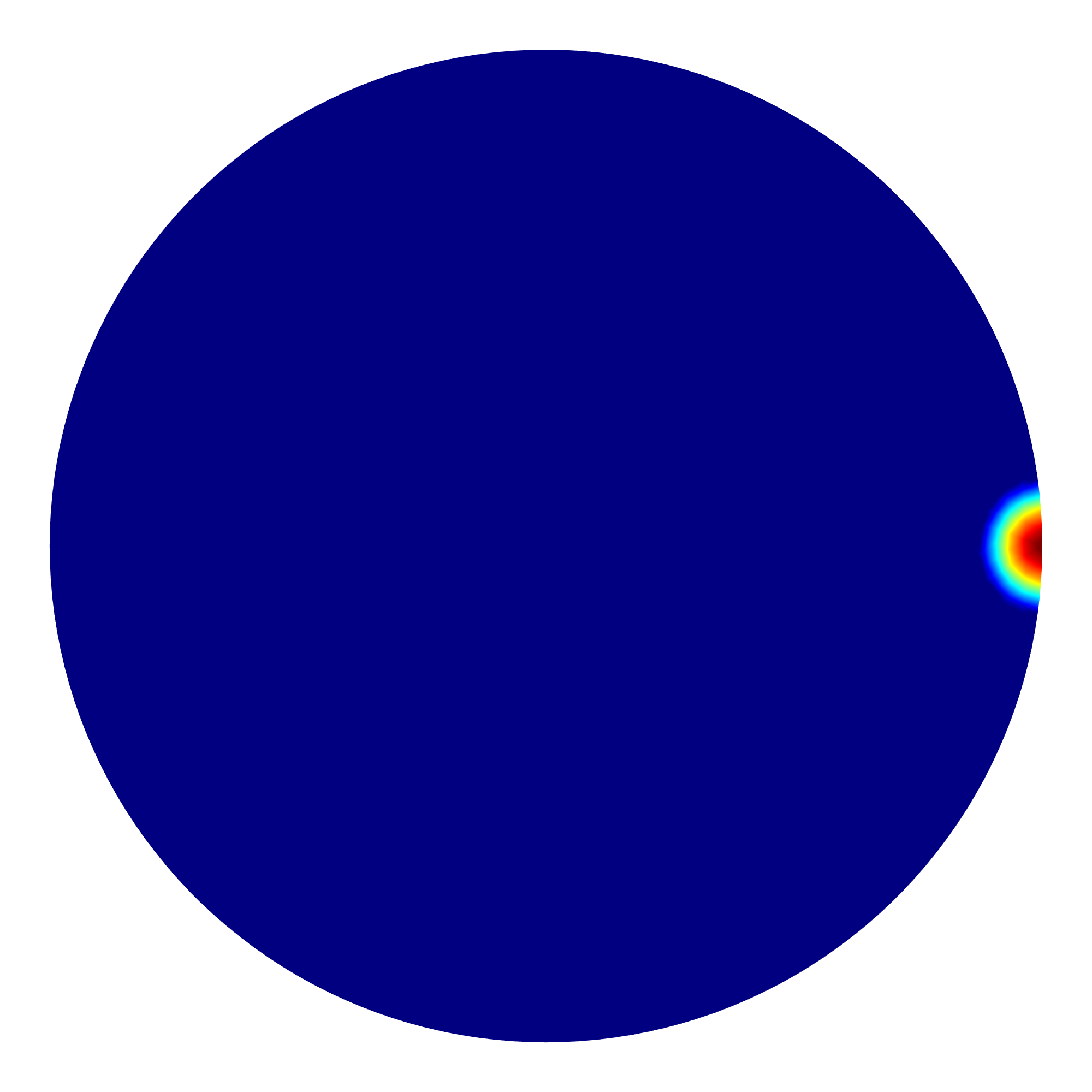}
        \caption*{$t=100$}
    \end{subfigure}\hspace{0.5mm}
    \caption{Formation of non-radial stationary pattern in $u(\textbf{x})$ within (\ref{01}) under the same settings as in Figure \ref{BS1} except that $\chi=300$.  Similar as above, the time-dependent system endures several phase transition processes, and eventually develops into the single boundary spike which is (the most) stable.}\label{BS4}
\end{figure}

Finally, we conduct another numerical experiment of (\ref{01}) in Figure \ref{BS4} under the same conditions as above except that $\chi=300$.  Now the chemotaxis rate is stronger than previously, spatial patterns emerge and develop into one interior spike and four boundary spikes in an even shorter period.  Then one of the adjacent boundary spikes attract each other and quickly merge into a single one on the boundary at time $t\approx 2$.  Then this spatial pattern with one boundary spike in the west, one interior spike, and two boundary spikes (symmetric about the $x_1$-axis) slowly evolve; moreover, as the time involves, the interior spike and the two boundary spikes move towards the east end and they endure a phase transition at time $t\approx 7$ when the two boundary spikes disappear.  Then the interior spike keeping moving eastwards and it touches the boundary at $t\approx 9$, resulting in a large spike at the east end and small spike at the west end.  Then these two spikes stay stable for a very long time, and such mesa-stability is not destroyed until an even longer time when the small one starts moving towards the large one along the boundary.  The phase transition occurs in a relatively short period and eventually, the spatial pattern develops into a single stable spike on the boundary.  Our numerical studies suggest that a single boundary spike is the most stable hence has the smallest energy among all solutions.  As we mentioned earlier, a rigorous study of such non-radial solutions is out of the scope of this paper.

\section{Quadratic Diffusion Models in the Whole Space $\mathbb R^2$}\label{section5}
Finally, let us consider the counter-part of (\ref{ss}) in $\mathbb R^2$, the solution $(\mathbb U,\mathbb V)$ of which satisfies
\begin{equation}\label{sswholespace}
\left\{\begin{array}{ll}
0=(r \mathbb U(\mathbb U-\chi \mathbb V)_r)_r,&r\in (0,\infty),\\
0= \mathbb V_{rr}+\frac{1}{r}\mathbb V_r-\mathbb V+\mathbb U, &r\in (0,\infty),\\
\mathbb U\in C^0([0,\infty)), \mathbb V\in C^2([0,\infty)), \mathbb U(r)\geq 0, \mathbb V(r)>0,&r\in [0,\infty),\\
\int_{\mathbb R^2} \mathbb U(r)d\textbf{x}=M; \lim_{r\rightarrow \infty}\mathbb U(r)=\mathbb V (r)=0.
\end{array}
\right.
\end{equation}
Problem (\ref{sswholespace}) is a natural extension of (\ref{ss}) with $R=\infty$, however not all the results about (\ref{ss}) can be applied to the whole space by formally taking the limit.  For instance, (\ref{ss}) suggests $(0,0)$ is the constant solution of (\ref{sswholespace}), which apparently does not satisfy the conservation of mass; moreover, our coming results indicate that (\ref{sswholespace}) has  relatively simple spatial-temporal dynamics compared to (\ref{ss}).

Setting $R=\infty$ in (\ref{bifvalue}), Theorem \ref{theorem11} tends to imply the solvability of (\ref{sswholespace}) for each $\chi>1$.  We will verify this in this section.  Before proceeding further, it seems necessary to point out that transition of any solution of (\ref{sswholespace}) leads to another solution, therefore for the simplicity of (uniqueness) statement we set
\begin{equation}\label{centermax}
\max_{\mathbb R^N} \mathbb U(\textbf{x})=\mathbb U(\textbf{0}),
\end{equation}
then we will show that (\ref{sswholespace})-(\ref{centermax}) has a unique solution for each $\chi>1$.

Note that (\ref{sswholespace}) is the stationary problem of the following parabolic system in the radial class
\begin{equation}\label{quadraticwholespace}
\left\{\begin{array}{ll}
u_t=\nabla\cdot(u \nabla u-\chi u \nabla v),&\textbf{x} \in \mathbb R^N,t>0,\\
0=\Delta v -v+u, & \textbf{x} \in \mathbb R^N,t>0,\\
u(\textbf{x},0),v(\textbf{x},0)\geq 0, \not\equiv 0,& \textbf{x} \in \mathbb R^N,\\
\int_{\mathbb R^2}u(\textbf{x},0)d\textbf{x}=M,&
\end{array}
\right.
\end{equation}
the second equation of which can be solved in terms of the fundamental solution of $(\Delta -I)$ in $\mathbb R^N$ hence results in
\begin{equation}\label{aggdiff}
\left\{\begin{array}{ll}
u_t=\nabla \cdot\Big(u\nabla u^{m-1}+\nabla(W*u\big)\Big),&\textbf{x}\in\mathbb R^N,t\geq0,\\
u(\textbf{x},0)\geq,\not\equiv 0,&\textbf{x}\in\mathbb R^N,\\
\int_{\mathbb R^2}u(\textbf{x},0)d\textbf{x}=M,&
\end{array}
\right.
\end{equation}
with $m=2$ and $W$ being the Bessel potential leading to (\ref{quadraticwholespace}).

System (\ref{aggdiff}) has a delicate variational structure in the sense that it can be recognized as a gradient flow for probability measures of the following free energy
\[\mathcal E_{\text{f}}(u):=\frac{1}{m}\int_{\mathbb R^N} u^m(\textbf{x}) d \textbf{x}+\frac{1}{2}\int_{\mathbb R^N}u(W*u)(\textbf{x})d\textbf{x};\]
this free energy is also a Lyapunov functional and its minimizers are natural candidates for the stationary solutions of (\ref{aggdiff}) with its variations leading to the PDE.  Then the properties of stationary states can be studied by analyzing the local or global minimizers of the free energy to start with.

To obtain the global minimizers, one can start by proving that the free energy functional is bounded from below and the minimizing sequence is compact.  However, the main challenge arises is that the mass can be uncontrolled for the minimizing sequence due to the translational invariance of the energy functional.  Many works have been done toward the existence  \cite{BDF,CCVSIAM2015,CCH,CHMVCVPDE2018,CSIndiana2018} and uniqueness \cite{BDF,KYSIAM2012,CSIndiana2018} of the global minimizers of (\ref{aggdiff}) with attractive potential such as Newtonian potential or Riesz potential.  In general, this global minimizer is a stationary solution of (\ref{quadraticwholespace}), and then one can apply the decreasing rearrangement argument to show its radial symmetry and uniqueness up to a translation \cite{CCVSIAM2015,CHMVCVPDE2018,CSIndiana2018}.  However, it seems to us that most studies of the stationary solutions are concerned about the global minimizers of the free energy.

In a recent work \cite{CHVYInvent}, J. A. Carrillo \emph{et al.} proved among others that if $m>1$ and the potential $W$ is no more singular than Newtonian, any stationary solution of (\ref{aggdiff}) is radially symmetric and decreasing; moreover, if $W$ is more singular than Newtonian, every bounded stationary solution of (\ref{aggdiff}) must be radially decreasing.  These results apply to a very large class of aggregation-diffusion equations without restricting to their global minimizers.  However, it seems necessary to mention that no uniqueness of the compactly-supported of $u$ was known in \cite{CHVYInvent}, except for problems with certain drift.  Very recently, M. Delgadino, X. Yan and Y. Yao proved in \cite{DYY} the uniqueness of $m\geq 2$ with a general potential $W\in C^\infty(\mathbb R^N\backslash\{\textbf{0}\})$ which is no more singular than Riesz potential for some $k>-N$; however, for $1<m<2$ they demonstrated that there exist certain attractive potentials such that (\ref{aggdiff}) has infinitely many radially decreasing stationary solutions.  In another recent work for $1<m<2$ but with the Riesz potential, V. Calvez, J. A. Carrillo and F. Hoffmann \cite{CCH2} proved the uniqueness for (\ref{aggdiff}).  Readers can find reviews of aggregation-diffusion (\ref{aggdiff}) in \cite{CCH,CCY,DYY}.

We would like to point out that, though the potential $W$ in \cite{CHVYInvent,DYY} was set to Newtonian (or alike), the arguments there apply to (\ref{aggdiff}) with the Bessel potential which has the same singularity as Newtonian but with a better decay at infinity.  Therefore one concludes from \cite{CHVYInvent,DYY} that any stationary solution of (\ref{quadraticwholespace}) must be radially decreasing hence reduces to (\ref{sswholespace}).

In this section, we study the explicit and unique solution for (\ref{sswholespace})-(\ref{centermax}), and show that its solution is unique, and compactly supported as naturally expected.  In simple words, our results indicate that, with $W$ being the Bessel potential, (\ref{aggdiff}) is equivalent as (\ref{sswholespace}) with its stationary solution to be uniquely and explicitly given.

\subsection{Existence and Uniqueness of Stationary Solution in $\mathbb R^2$}

The discussions above imply that $\mathbb U$ in (\ref{sswholespace}) must be radial and decreasing within its support.  We claim that it must be compactly supported in $\mathbb R^2$ whenever $\chi>1$.  Indeed, if $\mathbb U>0$ in $\mathbb R^2$, we can find that $\mathbb V(r)=C_1J_0(\omega r)-\frac{C_2}{\chi}$ in $[0,\infty)$ for constants $C_i$.  Then the decaying condition $\mathbb V(\infty)=0$ readily implies that $\mathbb V(r)=C_1J_0(\omega r)$.  However, this implies that $\mathbb V$ changes sign in $(0,\infty)$, which is impossible.  Therefore, we must have that $\mathbb U$ is compactly supported.

Now we proceed to find explicit formula of the unique solution of (\ref{sswholespace}).  With that being said, we look for $(\mathbb U,\mathbb V)$ such that $\mathbb U-\chi \mathbb V=\bar C$ for $r\in[0,r^*)$ and $\mathbb U\equiv 0$ for $r\in [r^*,\infty)$.  Rewriting (\ref{22}) with $(0,R)$ replaced by $(0,\infty)$ gives us
\begin{equation}\label{52}
\left\{\begin{array}{ll}
\mathbb V_{rr}+\frac{1}{r}\mathbb V_r+(\chi-1)\mathbb V+\bar{C} =0,& r\in[0,r^*),\\
\mathbb V_{rr}+\frac{1}{r}\mathbb V_r-\mathbb V=0,& r\in[r^*,\infty),\\
\mathbb V_r(0)=0,\lim_{r\rightarrow \infty}\mathbb V(r)=0.
\end{array}
\right.
\end{equation}
Solving (\ref{52}) gives that
\begin{align}\label{wveq}
\mathbb V(r)=\left\{\begin{array}{ll}\
C_1J_0(\omega r)-\frac{\bar C}{\chi-1}, &r\in[0,r^*),\\
C_2K_0(r),&r\in[r^*,\infty),
\end{array}
\right.
\end{align}
with $r^*$ to be determined.   Before proceeding further, we want to point out that if $\chi\leq 1$, (\ref{sswholespace}) does not have any solution.  The case $\chi=1$ readily holds since $0$ is not an eigen-value of Neumann $-\Delta_r$ (otherwise $v$ equals some non-negative constant, which is impossible).  When $\chi<1$, let us denote $\tilde \omega:=\sqrt{1-\chi}$, then the explicit solution $\mathbb V(r)=C_1I_0(\tilde\omega r)-\frac{\bar C}{\chi-1}$ for $r\in[0,r^*)$ and $\mathbb V(r)=C_2K_0(r)$ for $r\in[r^*,\infty)$.  In this case, the continuities of $\mathbb V'$ and $\mathbb V''$ at $r^*$ enforce the following identity
   \begin{equation}\label{salgebraic55}
\frac{\tilde\omega I_0(\tilde\omega r^*)}{I_1(\tilde\omega r^*)}=-\frac{K_0(r^*)}{K_1(r^*)};
\end{equation}
however, one finds that $\frac{\tilde\omega I_0(\tilde\omega r)}{I_1(\tilde\omega r)}>-\frac{K_0(r)}{K_1(r)}$ for any $r>0$, therefore (\ref{salgebraic55}) is impossible hence (\ref{sswholespace}) does not have any solutions if $\chi\leq1$.  Indeed, one can apply the results above to conclude that it has no nonconstant solution at all.

Now for each $\chi>1$, the continuity of $\mathbb V'(r)$ and $\mathbb V''(r)$ for (\ref{wveq}) at $r=r^*$ implies
\[
\left\{\begin{array}{ll}
-C_1\omega J_1(\omega r^*)=-C_2K_1(r),\\
-C_1\omega^2\Big(J_0(\omega r^*)-\frac{J_1(\omega r^*)}{\omega r^*}\Big)=C_2\Big(\frac{K_1(r^*)}{r^*}+K_0(r^*)\Big)\bigg),
\end{array}
\right.
\]
which entails that $r^*$ is a root of the limiting algebraic equation of (\ref{23})
\begin{equation}\label{walgebraic}
f(r^*;\omega,\infty)=\frac{\omega J_0(\omega r^*)}{J_1(\omega r^*)}+\frac{K_0(r^*)}{K_1(r^*)}=0.
\end{equation}
As an analog of Lemma \ref{lemma21}, for each $\chi>1$, $f(r^*;\omega,\infty)$ in (\ref{walgebraic}) admits a unique root $r^*$ in $(\frac{j_{0,1}}{\omega},\frac{j_{1,1}}{\omega})$. With this unique root in hand, we collect from (\ref{wveq}) that
\begin{equation}\label{58}
\mathbb U(r)=\left\{\begin{array}{ll}
\!\!\mathcal A\_\big(J_0(\omega r)-J_0(\omega r^*)\big),\!\!&\!\!r\in[0,r^*),\\
\!\!0,\!\!&\!\!r\in[r^*,\infty),
\end{array}
\right.
\mathbb V(r)=\left\{\begin{array}{ll}
\mathcal A\_\big(\frac{J_0(\omega r)}{\chi}-J_0(\omega r^*)\big),&r\in[0,r^*),\\
\mathcal B\_ K_0(r),&r\in[r^*,\infty),
\end{array}
\right.
\end{equation}
where the coefficients $\mathcal A\_$ and $\mathcal B\_$ are explicitly given by
\[\mathcal A\_=\frac{M \omega}{\pi(2r^*J_1(\omega r^*)-\omega (r^*)^2J_0(\omega r^*))} \text{~and~} \mathcal B\_=-\frac{M\omega^2J_0(\omega r^*)}{\pi \chi(r^*)^2J_2(\omega r^*)K_0(r^*)}.\]
One sees that all the arguments and calculations for the interior spike in $B_0(R)$ holds an arbitrary $R$.  In summary, we have the following results:
\begin{theorem}\label{theorem51}
If $\chi\leq 1$, (\ref{sswholespace}) has no solution.  For each $\chi>1$, (\ref{sswholespace})-(\ref{centermax}) has one and a unique one solution $(\mathbb U(r),\mathbb V(r))$; moreover, this solution is explicitly given by (\ref{58}) such that $\mathbb U(r)$ is supported in the disk $B_0(r^*)$ for some $r^*\in(\frac{j_{0,1}}{\sqrt{\chi-1}},\frac{j_{1,1}}{\sqrt{\chi-1}})$ determined by (\ref{walgebraic}); furthermore, the following asymptotics hold in the large limit of $\chi$:

(i) $r^*$, the size of support of $\mathbb U$, is monotone decreasing in $\chi$ and $r^*=O(\frac{1}{\sqrt{\chi}})$ for $\chi\gg 1$;

(ii) $\Vert\mathbb U\Vert_{L^\infty(\mathbb R^2)}$ is monotone increasing in $\chi$ and $\Vert\mathbb U\Vert_{L^\infty(\mathbb R^2)}=O(\chi)$ for $\chi\gg 1$;

(iii) $\mathbb U(r)\rightarrow M\delta_0(r)$ and $\mathbb V(r)\rightarrow \frac{M}{2\pi} K_0(r)$ pointwisely in $\mathbb R^2$ as $\chi\rightarrow \infty$.
\end{theorem}

Theorem \ref{theorem51} states that in the radial setting, the solution of (\ref{sswholespace})-(\ref{centermax}) is uniquely and explicitly given by (\ref{58}).  Then according to \cite{CHVYInvent,DYY}, these statements also hold for (\ref{quadraticwholespace}) and we have the followings:
\begin{corollary}
For each $\chi>1$, (\ref{quadraticwholespace}) has one and a unique one stationary solution $(u,v)$; moreover, this solution has the explicit formula $(\mathbb U(r),\mathbb V(r))$ in Theorem \ref{theorem51}, hence it is radially symmetric, with $u$ being compactly supported and monotone decreasing within its support.
\end{corollary}

\subsection{A Cousin Problem with Logarithmic Potential in $\mathbb R^2$}
Finally, we consider the following cousin problem of the steady states of (\ref{quadraticwholespace})
\begin{equation}\label{logpotential}
\left\{\begin{array}{ll}
0=\nabla\cdot(\mathcal U \nabla \mathcal U-\chi \mathcal U \nabla \mathcal V),&\textbf{x} \in \mathbb R^2,\\
0=\Delta \mathcal V+\mathcal U, & \textbf{x} \in \mathbb R^2,\\
(\mathcal U,\mathcal V)\in C^0(\mathbb R^2)\times C^2(\mathbb R^2),\int_{\mathbb R^2}\mathcal U(\textbf{x})d\textbf{x}=M.&
\end{array}
\right.
\end{equation}
It is well known that any solution of (\ref{logpotential}) must be radial and compactly supported such that
\[(\mathcal U(\textbf{x}),\mathcal V(\textbf{x}))=(\mathcal U(r),\mathcal V(r)) \text{~and~}\mathcal U\equiv 0 \text{~for $r$ large}.\]
We will explore the explicit formula for this unique solution in this subsection.  It is necessary to point out that for any constant $C$, $(\mathcal U,\mathcal V+C)$ is a solution of (\ref{logpotential}) whenever $(\mathcal U,\mathcal V)$ is one, therefore the uniqueness here is understood as that of $(\mathcal U,\mathcal V_r)$ as we shall see later.

Similar as above, we find that the support of $\mathcal U$ must be the form of a disk $B_0(r^*)$, therefore $\mathcal U-\chi \mathcal V=\bar C$ a constant in $B_0(r^*)$ and $\mathcal U\equiv 0$ in $\mathbb R^2\backslash B_0(r^*)$.  Therefore the $\mathcal V$-equation implies that $\Delta_r(\mathcal V+\frac{\bar C}{\chi})+\chi(\mathcal V+\frac{\bar C}{\chi})=0$ in $(0,r^*)$ and $\Delta_r \mathcal V=0$ in $(r^*,\infty)$.
\[\mathcal U(r)=\left\{\begin{array}{ll}
\chi \mathcal V+\bar C,&r\in[0,r^*),\\
0, & r\in(r^*,\infty),
\end{array}
\right.
\quad
\mathcal V(r)=\left\{\begin{array}{ll}
C_1J_0(\sqrt{\chi} r)-\frac{\bar C}{\chi},&r\in[0,r^*),\\
C_2\ln r+C_3, & r\in(r^*,\infty).
\end{array}
\right.
\]
We claim that $r^*=\frac{j_{0,1}}{\sqrt{\chi}}$.  Before showing this, we first show that $C_2\neq 0$.  If not, then $\mathcal V\equiv C_3$ in $(r^*,\infty)$ and the $C^1$-continuity of $\mathcal V(r)$ at $r^*$ implies that $r^*=\frac{j_{1,1}}{\sqrt{\chi}}$, which leads a contradiction to the $C^2$ continuity of $\mathcal V$ at $r^*$.

Next we enforce the $C^1$ and $C^2$ continuity of $\mathcal V(r)$ at $r^*$ and find that $J_0(wr^*)=0$ hence $r^*=\frac{j_{0,1}}{\sqrt{\chi}}$ as claimed.  Moreover, the integral constraint implies that $C_2=-\frac{M}{2\pi}$.  In summary, we collect the following explicit formula of the unique solution $(\mathcal U,\mathcal V)$ to (\ref{logpotential}) as
\begin{equation}\label{511}
\mathcal U(r)=\left\{\begin{array}{ll}
\frac{M\chi}{2\pi j_{0,1}} \frac{J_0(\sqrt{\chi} r)}{J_1(j_{0,1})},&r\in[0,\frac{j_{0,1}}{\sqrt{\chi}}),\\
0, & r\in(\frac{j_{0,1}}{\sqrt{\chi}},\infty),
\end{array}
\right.
\text{~with~}
\mathcal V_r(r)=\left\{\begin{array}{ll}
-\frac{M\sqrt{\chi}}{2\pi j_{0,1}}\frac{J_1(\sqrt{\chi} r)}{J_1(j_{0,1})},&r\in[0,\frac{j_{0,1}}{\sqrt{\chi}}),\\
-\frac{M}{2\pi} \frac{1}{r}, & r\in(\frac{j_{0,1}}{\sqrt{\chi}},\infty).
\end{array}
\right.
\end{equation}
In contrast to (\ref{sswholespace}) where $\mathcal V$ describes concentration of the stimulating chemical, $\mathcal V$ can only be uniquely determined by its derivative in the sense of (\ref{511}).  On the other hand, $\mathcal V$ becomes negative for $r$ sufficiently large, hence it should be understood as an attractive potential rather than a physical chemical concentration in (\ref{logpotential}).

\section{Appendix}\label{section7}
Here we establish the finer estimate $r_2\in(\frac{\pi}{2\omega},\frac{5\pi}{4\omega})$ by the asymptotic expansions of Bessel functions and the uniform bounds of the ratios of modified Bessel functions $J_n(x)$ and $Y_n(x)$, $n\in\mathbb N$.  The results in Chap 7.3 of \cite{Watson} state that for any $x>0$, the finite sum within the asymptotic expansions of $J_n(x)$ and/or $Y_n(x)$ converges to each function, with the corresponding remainders being of the same sign of the finite sum with the last term ignored.  To be precise, one has
\begin{align}
  & J_n(x)=\sqrt{\frac{2}{\pi x}}\left(P(x,n) \cos\big(x-\frac{n\pi}{2}-\frac{\pi}{4}\big)- Q(x,n)\sin \big(x-\frac{n\pi}{2}-\frac{\pi}{4}\big) \right), \label{71} \\
  & Y_n(x)=\sqrt{\frac{2}{\pi x}}\left(P(x,n) \sin\big(x-\frac{n\pi}{2}-\frac{\pi}{4}\big) +Q(x,n) \cos \big(x-\frac{n\pi}{2}-\frac{\pi}{4}\big) \right),  \label{72}
\end{align}
where $P(x,n)$ and $Q(x,n)$ can be presented by
$$P(x,n)=\sum_{m=0}^{s-1}\frac{(-1)^m\cdot(\frac{1}{2}-n)_{2m}\cdot(\frac{1}{2}+n)_{2m}}{(2m)!(2x)^{2m}}+\epsilon_1\frac{(-1)^s\cdot(\frac{1}{2}-n)_{2s}\cdot(\frac{1}{2}+n)_{2s}}{(2s)!(2x)^{2s}},$$
for some $|\epsilon_1|<1$ if $2s\geq n-\frac{1}{2}$, and
$$Q(x,n)=\sum_{m=0}^{s-1}\frac{(-1)^m\cdot(\frac{1}{2}-n)_{2m+1}\cdot(\frac{1}{2}+n)_{2m+1}}{(2m+1)!(2x)^{2m+1}}+\epsilon_2\frac{(-1)^s\cdot(\frac{1}{2}-n)_{2s+1}\cdot(\frac{1}{2}+n)_{2s+1}}{(2s+1)!(2x)^{2s+1}},$$
for some $|\epsilon_2|<1$ if $2s\geq n-\frac{3}{2}$.  Thanks to these approximations, there exist some constants $\theta_1$, $\theta_2$, $\theta_3$ and $\theta_4$ in $(0,1)$ such that
\begin{align}
   & P(x,0)=1-\frac{9\theta_1}{128x^2},\quad \quad\quad P(x,1)=1+\frac{15\theta_3}{128x^2}, \label{73}\\
   & Q(x,0)=-\frac{1}{8x}+\frac{75\theta_2}{1024x^3},\quad Q(x,1)=\frac{3}{8x}-\frac{105\theta_4}{1024x^3}.\notag
\end{align}
We also want to note that the positivity of each $\theta_i$ is promised by the facts that the remainders after two terms of $P(x,n)$ and $Q(x,n)$ are of the same sign as the sum of the first two term.  Using (\ref{71}) and (\ref{72}), we have
\begin{align}
  & S_0\left(\frac{\pi}{2\omega};\omega,R\right)=\frac{2}{\pi \omega\sqrt{R\left(R-\frac{\pi}{2\omega}\right)}}  \left(Q(\omega R,1)\cdot P(\omega R-\frac{\pi}{2},0) -P(\omega R,1)\cdot Q(\omega R-\frac{\pi}{2},0) \right), \notag\\
  &  S_1\left(\frac{\pi}{2\omega};\omega,R\right)=\frac{2}{\pi \omega\sqrt{R\left(R-\frac{\pi}{2\omega}\right)}}  \left(P(\omega R,1)\cdot P(\omega R-\frac{\pi}{2},1)+Q(\omega R,1)\cdot Q(\omega R-\frac{\pi}{2},1) \right). \notag
\end{align}
Substituting (\ref{73}) into the above expressions, using the boundedness of $\theta_i$ and the fact $\omega R>j_{1,1}$, we obtain from straightforward calculations
$$\frac{\omega  S_0\left(\frac{\pi}{2\omega};\omega,R\right)}{S_1\left(\frac{\pi}{2\omega};\omega,R\right)}<\omega\frac{\frac{1}{2(\omega R-\frac{\pi}{2})}+\frac{15\theta_3}{1024\omega^2R^2(\omega R-\frac{\pi}{2})}}{1+\frac{15\theta_3}{128\omega^2R^2}}<\frac{1}{2(R-\frac{\pi}{2\omega})}.$$
Let us recall the following inequality from \cite{Amos}
$$\frac{I_0(R-\frac{\pi}{2\omega})}{I_1(R-\frac{\pi}{2\omega})}\geq \frac{1+\sqrt{4(R-\frac{\pi}{2\omega})^2+1}}{2(R-\frac{\pi}{2\omega})},$$
therefore $f_2(\frac{\pi}{2\omega};\omega,R)<0$ implies that $r_2>\frac{\pi}{2\omega}$ as asserted.  By the same token, one can show that $s_0^{(1)}\in(\frac{\pi}{4 \omega},\frac{\pi}{2\omega})$ and $s_1^{(1)}\in(\frac{\pi}{\omega},\frac{5\pi}{4\omega})$ by the monotonicity of $S_0(r;\omega,R)$ and $S_1(r;\omega,R)$, and then obtain their asymptotic expansions in terms of $P(x,n)$ and $Q(x,n)$.

Next, we claim that $r_2 \rightarrow (\frac{\pi}{2\omega})^+ $ and $s_0^{(1)} \rightarrow (\frac{\pi}{2\omega})^-$ in the limit of $\chi\rightarrow \infty$.  To see this, we first find the following asymptotic behavior of Bessel functions in this limit as
\begin{align}
  &S_0(r;\omega,R)=-\frac{2}{\pi \omega \sqrt{R(R-r)}} \cos (\omega r)+O((\omega R)^{-2}), \label{74} \\
  & S_1(r;\omega,R)=\frac{2}{\pi \omega \sqrt{R(R-r)}} \sin (\omega r)+O((\omega R)^{-2}), \label{75}
\end{align}
where the big $O$ is uniform in $r$.  These asymptotic expansions indicate that $s_0^{(1)}\rightarrow (\frac{\pi}{2\omega})^-$ and $s_1^{(1)}\rightarrow (\frac{\pi}{\omega})^+$.  Recall that $r_2$ is the first positive root of (\ref{210}), the left hand side of which approaching $-\frac{\omega}{\tan(\omega r)}$, the right hand side approaching the positive bound constant $\frac{I_0(R)}{I_1(R)}$ in large $\chi$, then an immediate consequence of these facts is that $r_2 \rightarrow (\frac{\pi}{2\omega})^+$ as $\chi\rightarrow \infty$, which is expected.
\section*{Acknowledgements}
We would like to thank Jose A. Carrillo for helpful comments and encouragement, and bringing \cite{CCH2,CCY} to our attention; and thank Yao Yao for helpful discussions about \cite{CHVYInvent} and bringing \cite{DYY} to our attention.  QW is partially supported by NSF-China (11501460) and the Fundamental Research Funds for the Central Universities (JBK1805001).


\bibliographystyle{abbrv}

\end{document}